\documentclass [a4paper,12pt]{report}
\usepackage[T1]{fontenc}
\usepackage{graphicx,color}

\usepackage[utf8]{inputenc}
\usepackage{graphicx}
\usepackage{amsmath,amsthm}
\usepackage{amssymb}
\usepackage[english]{babel}
\usepackage{xcolor}
\usepackage{fullpage}
\usepackage{comment}
\usepackage{mathrsfs}
\usepackage{pdfpages}

\newcommand{\<}{\langle}
\renewcommand{\>}{\rangle}

\usepackage{hyperref}
\usepackage[hyperpageref]{backref}
\usepackage{geometry, indentfirst}
\usepackage{amsfonts, latexsym, amssymb, amsmath}
\usepackage[all,knot,arc,import,poly]{xy}
\usepackage[english]{}

\renewcommand{\baselinestretch}{1.5}

\newtheorem{theorem}{Theorem}[section]

\newtheorem{corollary}[theorem]{Corollary}

\newtheorem{definition}[theorem]{Definition}

\newtheorem{lemma}[theorem]{Lemma}

\newtheorem{proposition}[theorem]{Proposition}
\newtheorem{remark}[theorem]{Remark}

\begin{document}
\pagenumbering{roman}

%

\thispagestyle{empty}

\mbox{}
\begin{center}

\textbf{{\large Universidade Federal da Paraíba}}

\textbf{{\large Programa de Pós-Graduação em Matemática}}

\textbf{{\large Doutorado em Matemática}}

\vspace {4.6cm}

\textbf{{\huge Fractional one-sided measure theoretic second-order elliptic operators and applications to stochastic partial differential equations}}

\vspace {4.0cm}

\textbf{por}

\end{center}

\vspace {0.3cm}

\begin{center}
\textbf{{\Large Kelvin Jhonson Ribeiro de Sousa Almeida Silva}}
\end{center}

\vspace {4.0cm}
\begin{center}
\textbf{João Pessoa - PB}

\textbf{Fevereiro/2022}\\[0pt]
\end{center}

\newpage

\setcounter{page}{2}
\pagestyle{plain}

\begin{center}
\textbf{\huge  Fractional one-sided measure theoretic second-order elliptic operators and applications to stochastic partial differential equations}
\end{center}

\begin{center}
\textbf{por}
\end{center}

\begin{center}
\textbf{{\large Kelvin Jhonson Ribeiro de Sousa Almeida Silva}}{
\setcounter{footnote}{1}\renewcommand{%
\thefootnote}{\fnsymbol{footnote}}\footnote{Este trabalho contou com
apoio financeiro da CAPES.}}
\end{center}

\vspace {0.1cm}

\begin{center}
\textbf{sob orientação do}
\end{center}

\vspace {0.1cm}

\begin{center}
\textbf{{\large Prof. Dr. Alexandre de Bustamante Simas}}
\end{center}




\vspace{3.0cm} {\hspace{4cm}
\begin{minipage}[t]{10cm}
Tese apresentada ao Corpo Docente do Programa de Pós-Graduação em Matemática - UFPB,
como requisito parcial para obtenção do título de Doutor em Matemática.
\end{minipage}

\vspace{4.0cm}}

\begin{center}
\textbf{João Pessoa - PB}

\textbf{Fevereiro/2022}\\[0pt]
\end{center}

%

\newpage

\mbox{} \vspace{2.0cm}

\begin{center} \textbf{{\huge Abstract}} \end{center}

\vspace{1.5cm}

\noindent In this work we introduce and study fractional measure theoretic elliptic operators on the torus and a new stochastic process named W-Brownian motion. We establish some regularity and spectral results related to the operators cited above, more precisely, we were able to provide sharp bounds for the growth rate of eigenvalues to an associated eigenvalue problem. Moreover, we show how the Cameron-Martin space associated to the W-Brownian motion relates to Sobolev spaces connected with the elliptic operators mentioned above. Finally applications of the theory developed on stochastic partial differential equations are given.

\vspace{1.0cm}

%

\newpage

\mbox{} \vspace{2.0cm}

\begin{center} \textbf{{\huge Resumo}} \end{center}

\vspace{1.5cm}

\noindent Neste trabalho apresentamos e estudamos os operadores elípticos unilaterais fracionários no sentido teorético da medida sobre o toro e um novo processo estocástico denominado movimento W-Browniano. Estabelecemos alguns resultados de regularidade e espectrais relacionados aos operadores citados acima, mais precisamente, fomos capazes de fornecer limites nítidos para a taxa de crescimento de autovalores para um problema de autovalor associado. Além disso, mostramos como o espaço de Cameron-Martin associado ao movimento W-Browniano se relaciona com espaços Sobolev naturalmente associado aos operadores elípticos mencionados anteriormente. Finalmente, fornecemos algumas aplicações da teoria deselvolvida em equações diferenciais parciais estocásticas.

\vspace{1.0cm}

\newpage
\mbox{} \vspace{1.0cm}

\noindent \textbf{{\huge Agradecimentos}}

\vspace{1.5cm}
\noindent

Agradeço primeiramente a DEUS por ter abençoado meu caminho e ter permitido que eu chegasse até aqui, garantindo a minha saúde e paz fronte as eventuais adversidades que nem mesmo pude perceber que estavam presentes.

\vspace{0.2cm}
\noindent

Ao professor Alexandre de Bustamente Simas por toda a paciência e apoio, pelos valorosos conselhos de vida e também por todo conhecimento matemático que ele compartilhou comigo ao longo desta etapa de minha vida. Sua participação foi de fundamental importância para a realização deste sonho.

\vspace{0.2cm}
\noindent

Aos meus famíliaares por todo apoio moral ao longo deste período do Doutorado e também por entenderem os evetuais momentos em que precisei estar um pouco ausente, em especial a minha querida mãe Ivanda Ribeiro de Sousa.

\vspace{0.2cm}
\noindent

A minha amada Mirian dos Santos Mendes por ter acreditado no meu potencial, por todo o companheirismo e por estar do meu lado em um dos momentos mais difíceis que tive ao longo do Doutorado. Meus agradecimentos também a sua família por todo o apoio.

\vspace{0.2cm}
\noindent

A todos os amigos, colegas e professores que tive a oportunidade de conhecer e estudar ao longo de todo meu período acadêmico, o acolhimento de vocês foi essencial. Agradeço a cada um destes que, entre um copo e outro de café, puderam me ensinar um pouco mais sobre a vida e me tornaram uma pessoa melhor.
\vspace{0.2cm}
\noindent

A CAPES, pelo apoio financeiro.

\vspace{0.2cm}
\noindent

A todos que de alguma forma me ajudaram.

\newpage

\mbox{} \vspace{5.0cm}


\vspace{12.0 cm} {\hspace{4.0 cm}
\begin{minipage}[t]{10.0 cm}
{\it ``Que ninguém se engane, só se consegue a simplicidade através de muito trabalho.''}\\
\begin{flushright}
{\it Clarice Lispector}
\end{flushright}
\end{minipage}

\newpage

\mbox{} \vspace{3.0cm}

\noindent \textbf{{\huge Dedicatória}}

\vspace{15.0 cm} {\hspace{7.0 cm}
\begin{minipage}[t]{7.0 cm}
A todos os meus professores, e colegas do meio acadêmico.
\end{minipage}

\renewcommand{\baselinestretch}{1.2}
\tableofcontents

\newpage \setcounter{page}{1}
\pagenumbering{arabic}

%

\pagenumbering{arabic} \setcounter{page}{1} \thispagestyle{empty} 
\addcontentsline{toc}{section}{{\bf Introduction}}
\chapter*{Introduction}

A class of second order differential operators $\frac{d}{dV}\frac{d}{dx}$ introduced by W. Feller in \cite{feller} gave rise to research in many different mathematical fields, for instance in the field of stochastic differential equations, real analysis, fractal geometry, etc. We can cite \cite{uta, franco, farfansimasvalentim} and \cite{wsimas} to name a few. On the other hand, \cite{pouso}, along with some other works from the same authors, have focused their attention in an interesting way to define a consistent local definition of the derivative in the sense of Stieltjes in such a way that some useful results in analysis can be re-obtained as a consequence of local definition of such Stieltjes derivatives. On another direction, Franco and Landim \cite{franco} studied the formal adjoint of Feller's original operator, which is given by $\frac{d}{dx}\frac{d}{dW}$, that is, in a way (since \cite{uta} requires the measure induced by $W$ to be nonatomic, whereas in \cite{franco} there is no such a restriction), a particular case of the operator $\frac{d}{dV}\frac{d}{dW}$ introduced by Uta in \cite{uta}. However, unlike Uta \cite{uta}, Franco and Landim \cite{franco} started with a local operator that was defined in a manner that it depends on a type of Stieltjes derivative of càdlàg functions. The link between the operators in \cite{uta} and \cite{franco} is very interesting and some intriguing questions emerge from the different approaches adopted to the study these second-order operators.

On the second chapter of this thesis we provide a general operator that can connect, in some sense, all of the operators above. We begin by considering the approach of \cite{franco} to generalize the formal operator $\frac{d}{dV}\frac{d}{dW}$. This operator, which comes from a local definition, agrees with the ``global'' operator introduced in \cite{uta} in the case the measures induced by $V$ and $W$ are atomless. Note that we do not require such a restriction. Further, we work with a generalization of the one-sided derivatives, which can be understood as a one-sided derivative in the Stieltjes sense. More precisely, we fix two increasing functions $W,V:\mathbb{R}\to\mathbb{R}$ where $W$ is a càdlàd (a right-continuous function with left limits) function and $V$ is a càglàd (a left-continuous function with right limits) function that satisfy some periodic conditions to properly define finite Borel measures $dW$ and $dV$ on $\mathbb{T}=\mathbb{R}/\mathbb{Z}$. With the assumptions on $V$ and $W$ in mind, we say that $f:\mathbb{T}\to\mathbb{R}$ is $W$-left differentiable or $V$-right differentiable, respectively, according to the existence of the limits 
$$D^{-}_{W}f(x):=\lim_{h\to 0^{-}}\frac{f(x+h)-f(x)}{W(x+h)-W(x)}$$
or $$D^{+}_{V}f(x):=\lim_{h\to 0^{+}}\frac{f(x+h)-f(x)}{V(x+h)-V(x)}$$
for all $x\in\mathbb{T}.$ By using these notions of generalized one-sided derivatives, we introduce the domain $\mathfrak{D}_{W,V}$ of the new symmetric and non-positive second order operator $$D^{+}_{V}(D^{-}_{W}):\mathfrak{D}_{W,V}\subset L^{2}_{V}(\mathbb{T})\to L^{2}_{V}(\mathbb{T}).$$ At this point is very important pay attention on the subscript ``W,V'', whenever we change the order of the functions, it means that we are considering the ``dual'' scenario, in which we reverse the operations between $V$ and $W$. Furthermore, the operator $D^{+}_{V}(D^{-}_{W})$ satisfies all the required hypothesis to ensure the existence of the Friedrichs extension. Let, then,  $\left(I-\Delta_{W,V}\right):\mathcal{D}_{W,V}\subset L^{2}_{V}(\mathbb{T})\to L^{2}_{V}(\mathbb{T})$ be the Friedrichs extension of the operator $I-D_{V}^+D_W^-$, and let the associated energetic space be given by $H_{W,V}(\mathbb{T})$. Moreover, define the formal second order operator $$\Delta_{W,V}:=I-(I-\Delta_{W,V}):\mathcal{D}_{W,V}\subset L^{2}_{V}(\mathbb{T})\to L^{2}_{V}(\mathbb{T})$$ 
that can be understood as the model for the measure theoretic Laplacian $\frac{d}{dV}\frac{d}{dW}$ introduced on the torus. One notable result we were able to obtain is refers to the regularity of the eigenvectors of the problem 
\begin{equation*}
		\begin{cases}
			-\Delta_{W,V}\phi = \lambda\phi,\,\lambda\in\mathbb{R}\setminus\{0\};\\
			u\in\mathcal{D}_{W,V}(\mathbb{T}).
		\end{cases}
	\end{equation*}
We proved that each solution $\phi$ of the above problem belongs to the space $C^{\infty}_{W,V}(\mathbb{T})$, that is the space of càdlàg functions $f:\mathbb{T}\to\mathbb{R}$ such that $D^{-}_{W}f$ exists and is a zero-mean càglàd function, $D^{+}_{V}(D^{-}_{W}f)$ exists and is zero-mean càdlàg function, $D^{-}_{W}[D^{+}_{V}(D^{-}_{W}f)]$ exists and is a zero-mean càglàd function, and so on. The space $C^{\infty}_{W,V}(\mathbb{T})$ will work as a new space of test functions to define the notion of $W$-left weak derivative that will allow us to characterize the space $H_{W.V}(\mathbb{T})$ as a $W$-$V$-Sobolev space. Furthermore, we can characterize the space $H_{W,V}(\mathbb{T})$ in terms of the Fourier coefficients of the functions. More precisely, we prove that $$H_{W,V}(\mathbb{T})=\left\{f\in L^{2}_{V}(\mathbb{T}); f=\alpha_{0}+\sum_{i=1}^{\infty}\alpha_{i}\nu_{i}; \sum_{i=1}^{\infty}\lambda_{i}\alpha_{i}^2<\infty\right\},$$
where $\{\lambda_{i}\}_{i\geq 1}$ are the eigenvalues of the operator $-\Delta_{W,V}$ which exist due to the compact immersion of $H_{W,V}(\mathbb{T})$ into $L^{2}_{V}(\mathbb{T})$, which we proved in Theorem $\ref{compactemb}.$ After some applications on partial differential equations, we introduce, on the subsection 2.7 a new generalization of the Brownian motion, and called $W$-Brownian motion, denoted by $B_{W}$, which possess the following properties:
\begin{enumerate}
	\item The finite-dimensional distributions of $B_W$ are Gaussian;
	\item The sample paths of $B_W$ are c\`adl\`ag and may have jumps.
\end{enumerate}
Satisfy both conditions at the same time are uncommon features in the literature in the range of generalizations of the Brownian motion. Moreover, we establish a very deep and intimate connection between the $W$-Brownian motion when we obtain the Cameron-Martin space of $B_{W}$ as the space $H_{W,V,\mathcal{D}}(\mathbb{T}) = \{f\in H_{W,V}(\mathbb{T}): f(0) = 0\}$. This shows that the Cameron-Martin space of $B_W$ is, actually, our $W$-$V$-Sobolev space with a Dirichlet condition. This, in turn, shows that the $W$-Brownian motion can be obtained as the process whose distribution in $L^{2}_{V}(\mathbb{T})$ is the unique Gaussian measure associated to the Cameron-Martin space $H_{W,V,\mathcal{D}}(\mathbb{T})$. Finally, some applications to stochastic partial differential equations are presented.

In Chapter 3, we provide sufficient conditions for a functions on $C^\infty_{W,V}(\mathbb{T})$ to be represented as an analogue of MacLaurin series expansion. This representation can be seen as the $W$-$V$-version of analytical functions. Thus, we provide sufficient conditions for a function to be analytical in some sense.  More precisely, if we define $F_{1}(x)=W(x)$ and recursively define
\begin{equation*}
    F_{n}(x,s)=\begin{cases}
         \displaystyle\int_{[0,s)}\left[F_{n-1}(x,x)-F_{n-1}(x,\xi)\right]dV(\xi),\; \mbox{n even};\\
        \displaystyle\int_{(0,s]}\left[F_{n-1}(x,x)-F_{n-1}(x,\xi)\right]dW(\xi),\; \mbox{n odd},
    \end{cases}
\end{equation*}
then the sequence $\{F_{k}(x,x)\}_{k\geq 1}$ plays the role of the polynomial terms $\dfrac{x^{n}}{n!}$ in the classical MacLaurin expansion for $W(x)=V(x)=x$. We provide suitable assumptions on the decay of the sequence given by the one-sided derivatives of $f$ at $0$ denoted by $D^{n}_{W,V}f(0)$, under which, we can obtain the uniform convergence of $\sum_{k=1}^{k}D^{(n)}_{W,V}f(0)F_{n}(x,x)$ on $\mathbb{T}$. We also found the compatibility conditions to ensure that the series will be well defined on $\mathbb{T}$. Moreover, all these results allowed us to determine a representation of the eigenvectors $\nu_{i}$ of $-\Delta_{W,V}$ in terms of the measure theoretic trigonometric functions $S_{W,V}(\alpha,x)$, $C_{W,V}(\alpha,x)$, $S_{V,W}(\alpha,x)$ and $C_{V,W}(\alpha,x)$ that are generalizations of the  classical trigonometric functions $\cos(\alpha x)$ and $\sin(\alpha x)$. In this chapter we also introduce the fractional order $W$-$V$-Sobolev spaces, which as defined in terms of their Fourier coefficients: $$\mathscr{D}(I-\Delta_{W,V})^{s/2}:=H^s_{W,V}(\mathbb{T})=\left\{f\in L^{2}_{V}(\mathbb{T}); f=\alpha_{0}+\sum_{i=1}^{\infty}\alpha_{i}\nu_{i}; \sum_{i=1}^{\infty}\lambda_{i}^s\alpha_{i}^2<\infty\right\},$$
for $s\in\mathbb{R}$. By using measure trigonometrical characterization of the eigenvectors of $-\Delta_{W,V}$ we were able to prove, by using the theory of entire functions, the existence of $\rho\in(0,1/2]$ such that $$\sum_{i=1}^{\infty}\frac{1}{\lambda_{i}^s}$$
forall $s>\rho$, or equivalently, that the operator $(I-\Delta_{W,V})^{s/2}$ is trace class for $s>\rho.$ We then provide a generalization of the results for dimensions $d$. That is, we introduced the generalized Laplacian on the $d$-dimensional torus $\mathbb{T}^d$ by taking tensor products of the eigenvectors of $\Delta_{W,V}$. We, then, provide $d$-dimensional counterparts to several results obtained for dimensions 1. The most remarkable being that for $s>d/2$ the operator $(I-\Delta_{W,V})^{s/2}$ is trace class, recovering the same lower bound for the Laplacian on $\mathbb{T}^d$. Finally, Chapter 3 ends with some applications on second order stochastic differential equations.

The thesis is concluded with a final discussion on some interesting open problems related to this theory.

%
\chapter{One-sided measure theoretic Laplacian and applications}
Recently some works have studied classes of generalized differential operators based on the concept of derivative $D_{W}:=\frac{d}{dW}$ with respect to an increasing function $W:\mathbb{R}\to\mathbb{R}$ which is right-continuous (or left-continuous). See, for instance, $\cite{frag, pouso,wsimas,trig,uta}$ and $\cite{franco}$.

Note that the second-order differencial operator $\frac{d}{dx}\left(\frac{d}{dW}\right)$ introduced in $\cite{frag}$ have one of its derivatives with respect to $W$ which may have jumps, and surprisingly these jumps of $W$ can be dense in $\mathbb{R}.$ This last observations ensures that regular functions with respect to such derivatives will, most likely, have discontinuity points. Typically, we will look for functions that are right-continuous and have left limits. On the other hand, it is noteworthy that the differential operator $\frac{d}{dW}\left(\frac{d}{dx}\right)f$ introduced by Feller in his seminal paper $\cite{feller}$ requires an initial space contained on the space of continuous functions, this is due to the requirement of the existence of the first derivative in the strong sense (and being derivatives from both sides, not only lateral derivatives). One should note that all the aforementioned operators generalize the usual laplacian operator $\frac{d}{dx}\left(\frac{d}{dx}\right)$, which is the case when $W(x)=x$.

In \cite{wsimas,simasvalentim2} they introduced the $W$-Sobolev spaces and obtained some elliptic regularity. They showed that the generalized derivative can be seen in the Sobolev perspective. However, as will also be discussed in Section \ref{sect4}, the usage of two-sided derivatives, meaning that the standard limit is taken in the definition of the derivative, instead of lateral limits, had a serious impact on the regularity study of such operators. Indeed, the space of regular functions on the $W$-Sobolev spaces, that arise from eigenvectors of the operators, which should be considered smooth, cannot be evaluated pointwisely. This shows a need for an adaptation of this theory, in such a way that we create a new type of $W$-Sobolev space  (which here will be more general, and we will call $W$-$V$-Sobolev spaces), with a nice regularity theory that allows for pointwise evaluations, but also agrees with the $W$-Sobolev spaces in \cite{wsimas} and also with the standard Sobolev spaces, see Remark \ref{sobagree}. 

The first goal of this paper is to consider a general one-sided second order differential operator in such a way that it generalizes at the same time both differential operators we mentioned above, namely $\frac{d}{dW}\left(\frac{d}{dx}\right)$  and $\frac{d}{dx}\left(\frac{d}{dW}\right)$. Indeed, we explain on Remark \ref{difop} how we can re-obtain these differential operator by just imposing regularity conditions on the range of those operators. On the second moment, given the generalized Laplacian $\Delta_{W,V}$, introduced in Definition \ref{laplacianWV}, we will introduce the natural space $C^{\infty}_{W,V}(\mathbb{T})$ where all eigenfunctions of the problem 
	\begin{equation}\label{eq:1}
		\begin{cases}
			-\Delta_{W,V}\phi = \lambda\phi,\,\lambda\in\mathbb{R}\setminus\{0\};\\
			u\in\mathcal{D}_{W,V}(\mathbb{T}),
		\end{cases}
	\end{equation}
	belong, and following the ideas of $\cite{wsimas}$ we introduce and new $(W,V)$-Sobolev space $H_{W,V}(\mathbb{T})$ and characterize this last space from other different perspectives. The insight to consider the space $C^{\infty}_{W,V}(\mathbb{T})$ is due to the regularity problem associated to the eigenvectors determined by the operator $\mathcal{L}_{W}$ studied in $\cite{franco}$ which we expected that each of them (almost surely) inherit the "interacted" regularity of the initial space $\mathfrak{D}_{W}$. 
	We then apply the theory developed to analyze some existence results related to the one-sided second order elliptic equations.
	
	Finally, to conclude the paper, we introduce the $W$-Brownian motion, which at first do not appear to be connected to these $W$-$V$-Sobolev spaces. However, the connection between the $W$-Brownian motion and the $W$-$V$-Sobolev spaces is very deep. Indeed, we show that the $W$-$V$-Sobolev space is the Cameron-Martin space of the distribution of the $W$-Brownian motion. As it is well-known, the Cameron-Martin space, in a sense, uniquely determines the Gaussian measure. So, the Gaussian measure associated to the $W$-$V$-Sobolev space is the law of the $W$-Brownian motion. 
	
	\section{The one-sided differential operator}
	\paragraph{} In this section we begin by introducing our differential operator in the strong sense, acting locally on functions. Then, we use it to define a strong second-order differential operator that extends the second derivative. After that, 
	we show that we can extend this second-order differential operator to a self-adjoint operator which extends the Laplacian. This operator turns out to have a compact resolvent. 
	
Let us introduce some notation. First, we say that a function is c\`adl\`ag (from french, ``continue à droite, limite à gauche'') if the function is right-continuous and has limits from the left. Similarly, we say that a function is c\`agl\`ad, if the function is left-continuous and has limits from the right. 

Let $W, V : \mathbb{R}\to\mathbb{R}$ be increasing functions that are, respectively, càdlàg and càglàd. Further, we assume that they are periodic in the sense that 
	\begin{equation}\label{periodic}
		\forall x \in \mathbb{R},\quad 
		\begin{cases}
			W(x+1)-W(x)=W(1)-W(0);\\
			V(x+1)-V(x)=V(1)-V(0), 
		\end{cases}
	\end{equation}
	\paragraph{}Without loss of generality we will assume that both $W$ and $V$ are continuous at zero, that is, $W(0)=0$ and $V(0)=0$. Indeed, since they are increasing, they can only have countably many discontinuities points, so if they are not continuous at zero, we can simply choose another point in which both of them are continuous and translate the functions to make them continuous at zero. Note that $W$ and $V$, thus, induce finite measures on $\mathbb{T}=\mathbb{R}/\mathbb{Z}$.
	We will now provide definitions of the one-sided differential operators. We begin by provinding their definitions in a strong sense, meaning that we see them as pointwise lateral limits, thus they act locally on functions. 
	
	Note that in the above definition we allow the measures induced by $W$ and $V$ to have atoms. This is a weaker assumption than some of the common assumptions found in literature, e.g., \cite{uta}, \cite{trig}, among others. 
	
	We will denote the $L^2$ space with the measure induced by $V$ by $L^2_V(\mathbb{T})$ and its norm (resp. inner product) by $\|\cdot\|_V$ (resp. $\langle \cdot,\cdot\rangle_V$). Similarly, we will denote the $L^2$ space with the measure induced by $W$ by $L^2_W(\mathbb{T})$ and its norm (resp. inner product) by $\|\cdot\|_W$ (resp. $\langle \cdot,\cdot\rangle_W$).
	
	\begin{definition}\label{defgendif}
		We say that a function $f : \mathbb{T}\to\mathbb{R}$ is $W$-left differentiable if  
		$$D^{-}_{W}f(x):=\lim_{h\to 0^{-}}\frac{f(x+h)-f(x)}{W(x+h)-W(x)}$$
		exists for all $x\in \mathbb{T}$. Similarly, we say that a function $g : \mathbb{T}\to\mathbb{R}$ is $V$-right differentiable if the limit 
		$$D^{+}_{V}g(x):=\lim_{h\to 0^{+}}\frac{g(x+h)-g(x)}{V(x+h)-V(x)}$$
		exists for all $x\in \mathbb{T}$.
	\end{definition}

The above definition of generalized lateral derivatives, besides having its natual appeal as a straightforward generalization of Newton's quotient, is also naturally obtained as the derivative operator when differentiating point processes with respect to a Borel measure. See, for instance, \cite{point}. It is also noteworthy that this definition is very general and will allow a proper regularity study. For instance, \cite{wsimas,simasvalentim2, simasvalentimjde, franco} considered a similar operator when $V(x) = x$, and when the derivative is not a lateral one. The drawback in their approach is that one cannot obtain a pointwise regularity theory. For instance, the eigenvectors of the differential operators considered in \cite{wsimas,simasvalentim2, simasvalentimjde, franco} can only be viewed in the $L^2$-sense, which was the reason the authors in \cite{simasvalentimjde} needed to create an auxiliary space for test functions and also obtain several results exclusively to deal with the lack of regularity. We will discuss more about that in Section \ref{sect4}.

A natural question is related to Definition \ref{defgendif} is if we have a suitable class of functions that are differentiable in this sense. More specifically, we are interested in the functions that are twice differentiable, meaning that they are differentiable with respect to $D_W^-$, and that the resulting function is further differentiable with respect to $D_V^+$. Let us study such a class.
	
	\begin{definition}\label{strongdomain}
		We will denote by $\mathfrak{D}_{W,V}(\mathbb{T})$ the set of càdlàg functions $f$ such that:
		\begin{itemize}
			\item $f$ is $W$-left differentiable.
			\item $D^{-}_{W}f$ is a càglàd function that is $V$-right differentiable.
			\item $D^{+}_{V}(D^{-}_{W}f)$ is a càdlàg function.
		\end{itemize}  
	\end{definition}
	
	Let us prove an auxiliary lemma that will help us characterization the functions in $\mathfrak{D}_{W,V}(\mathbb{T})$. 
	
	\begin{lemma}\label{lm:1}
		If $f : \mathbb{T}\to\mathbb{R}$ is càglàd and $D^{+}_{V}f\equiv 0$, then $f\equiv C,$ where $C$ is a constant. Similarly, if $g : \mathbb{T}\to\mathbb{R}$ is càdlàg and $D^{-}_{W}g\equiv 0$, then $g\equiv C$, where $C$ is a constant.
	\end{lemma}
	\begin{proof}
		We will only prove the statement for $f$, since the other is entirely analogous.
	   To this end, let us assume, by contradiction, that $f(a)\neq f(b)$, for some $a,b$, with $a<b$. Let $$\epsilon := \frac{|f(b)-f(a)|}{2\left[V(b)-V(a)\right]},$$
		and $\displaystyle c=\inf \left\{x; x\in (a,b], |f(x)-f(a)|\geq \epsilon \left[V(x)-V(a)\right]\right \}$. Since $f$ if left continuous, we have that $c<b$ and $|f(c)-f(a)|\le\epsilon(\left[V(c)-V(a)\right].$ $D^{+}_{V}f\equiv 0$ implies that $(D^{+}_{V}f)(c)=0$. Therefore, there exists $\delta>0$ such that $0<h<\delta$ implies  $$|f(c+h)-f(c)|\le\epsilon\left[V(c+h)-V(c)\right].$$ 
		This, in turn, implies that for every $0<h<\delta$
		$$|f(c+h)-f(a)|\le |f(c+h)-f(c)| + |f(c)-f(a)| \le \epsilon \left[V(c+h)-V(a)\right].$$ 
		This contradicts the choice of $c$. 
	\end{proof}
	
We are now in a position to characterize the functions in the space $\mathfrak{D}_{W,V}(\mathbb{T})$.
	
	\begin{lemma}\label{lm:2}
		A function $f:\mathbb{T}\to \mathbb{R}$ belongs to $\mathfrak{D}_{W,V}(\mathbb{T})$ if, and only if, there exists a càdlàg function $g : \mathbb{T}\to\mathbb{R}$ such that
		\begin{equation}\label{c1}
			f(x)=f(0)+W(x)D^{-}_{W}f(0)+\int_{(0,x]}\int_{[0,y)} g(s) dV(s)dW(y),
		\end{equation}
		and
		$$W(1)D^{-}_{W}f(0)+\int_{(0,1]}\int_{[0,y)} g(s) dV(s)dW(y)=0,\;\;\; \int_{[0,1)} g(s) dV(s)=0.$$
	\end{lemma}
	\begin{proof}
		It follows directly from the Lemma \ref{lm:1}
	\end{proof}
	
	\begin{remark}\label{difop}
		Let us consider the special case in which $V(x)=x$, and $\mbox{ran}(D^{+}_{V}D^{-}_{W})\subseteq \{\mbox{continuous functions}\}$. In such a case, we have by (\ref{c1}) that 
	$$D_{W}f(x):=\lim_{h\to 0} \frac{f(x+h)-f(x)}{W(x+h)-W(x)}$$
	is well defined for all $x$ in $\mathbb{T}$, where $D_{W}f$ is continuous and differentiable, with $D_{x}D_{W}f$ continuous, i.e, the domain $\mathfrak{D}_{W,V}(\mathbb{T})$ is exactly the domain $\mathfrak{D}_{W}$ of $D_{x}D_{W}$ introduced in $\cite{franco}$. 
	In the same way, we can work with the operator $D^{-}_WD^{+}_V$ in the domain of càglàd functions $f : \mathbb{T}\to\mathbb{R}$ that are: (i) $V$-right differentiable;  (ii) $D^{+}_Vf$ is càdlàg and $W$-left differentiable; and (iii) $D^{-}_W(D^{+}_Vf)$ is càglàd. Therefore if $V(x)$ is continuous and $\mbox{ran}(D^{-}_WD^{+}_V) \subseteq \{\mbox{continuous functions}\}$, then the domain of $D^{-}_WD^{+}_V$ is exactly the domain $\mathfrak{D}(\mathbb{T})$ of the Feller Operator $D_{W}D^{-}_{V}$. More precisely, if $\overline{W}(x)=W(x-)$ we have $$D_{\overline{W}}D^{-}_{V}f=D_{W}D^{+}_{V}f.$$ 
	For more details on the operator $D_{W}D^{-}_{V}$, we refer the reader to $\cite{mandl}$.
	\end{remark}
	
\begin{remark}\label{orderVW}
		In the rest of the paper we need to pay attention with the position of $W$ and $V$ on the subindex. We will use the subscript $W,V$ for every structure stritly related to the operator $D^{+}_{V}D^{-}_{W}$. Similarly, whenever we use the subscript $V,W$, we will be referring to the, analogous, structure related to the operator $D^{-}_{W}D^{+}_{V}$. Notice that by doing this, we will keep changing between c\`adl\`ag and c\`agl\`ad functions. We ask the viewer to be attentive to these details as they are subtle.
\end{remark}
	
	\begin{theorem}\label{thm:1}
		The following statements are true
		\begin{enumerate}
			\item The set $\mathfrak{D}_{W,V}(\mathbb{T})$ is dense in $L^2_{V}(\mathbb{T})$
			\item The operator $D^{+}_{V}D^{-}_{W}: \mathfrak{D}_{W,V}(\mathbb{T})\subset L^2_{V}(\mathbb{T})\to L^2_{V}(\mathbb{T})$ is symmetric and non-positive. More precisely,
			$$\langle D^{+}_{V}D^{-}_{W}f,g\rangle_{V}=-\int_{\mathbb{T}}D^{-}_{W}f(s)D^{-}_{W}g(s) dW(s).$$
			\item (Poincaré-Friedrichs Inequality) Let $f$ be a càdlàg function such that $D^{-}_{W}f$  exists and $D^{-}_{W}f$ is a càglàg function. Then,
			\begin{equation}\label{i1}
				\|f\|^{2}_{V}\le W(1)V(1)\|D^{-}_{W}f\|^{2}_{W}+V(1)f_{\mathbb{T}}^{2},
			\end{equation}
			where $f_{A}=\frac{\int_{A}f(s)dV(s)}{\int_{A}1dV(s)}.$
		\end{enumerate}
	\end{theorem}
	\begin{proof}
		Begin by noting that the space of continuous functions is dense in $L^2_{V}(\mathbb{T})$. Therefore, it is enough to show that every continuous function $f : \mathbb{T}\to\mathbb{R}$ can be aproximated, in the uniform norm, by functions in $\mathfrak{D}_{W,V}(\mathbb{T})$. Since $\mathbb{T}$ is compact, given any $\epsilon>0$, there exists $n\in\mathbb{N}$ such that $|x-y|\le 1/n$ implies $|f(x)-f(y)|\le\epsilon$. Now, let us consider the function $g: \mathbb{T}\to \mathbb{R}$
		defined by
		$$g(x)=\sum_{j=0}^{n-1}\dfrac{f\left(\frac{j+1}{n}\right)-f\left(\frac{j}{n}\right)}{W\left(\frac{j+1}{n}\right)-W\left(\frac{j}{n}\right)}\textbf{1}_{\{(j/n,(j+1)/n]\}}(x),$$
		where $\textbf{1}_{A}$ stands for the indicator of the set $A$. Let $G : \mathbb{T}\to\mathbb{R}$ be given by
			$$G(x) = f(0)+\int_{(0,x]}g(y)dW(y).$$
			Therefore, for $\frac{j}{n}<x\le \frac{j+1}{n}$, $0\le j \le n-1$, we have that
		$$G(x)=f\left(\frac{j}{n}\right)+\dfrac{f\left(\frac{j+1}{n}\right)-f\left(\frac{j}{n}\right)}{W\left(\frac{j+1}{n}\right)-W\left(\frac{j}{n}\right)}\left(W(x)-W\left(\frac{j}{n}\right)\right).$$	
		Hence, $G(j/n)=f(j/n)$ and, for $\frac{j}{n}<x< \frac{j+1}{n}$, we have
		$$|G(x)-f(x)|\le|G(x)-f(j/n)|+|f(x)-f(j/n)|.$$
		Now, observe that our choice of $n$ implies that $\| G-f \|_{\infty}\le 2 \epsilon$. Note, also, that $g$ is càglàg, and 
		$$\int_{\mathbb{T}}g(s)dW(s)=0.$$ 
		Since $g$ is a càglàd function, given $\epsilon>0$ there exists a partition of $\{0=z_{0}<z_{1}<\cdots<z_{k}=1\}$ of $\mathbb{T}$ such that $|g(b)-g(a)|\le \epsilon$ for $z_{k-1}\le b$ and $a<z_{k}.$ Now, let us define a cadlàg function $p:\mathbb{T}\to\mathbb{R}$ given by 
		$$p(x)=\sum_{i=1}^{k}\dfrac{g(z_{i})-g(z_{i-1})}{V(z_{i})-V(z_{i-1})}\textbf{1}_{\{[z_{i-1},z_{i})\}}(x).$$
		Note that $p$ is càdlàd and $\int_{\mathbb{T}}p(s)dV(s)=0$. Let us now use this $p$ to define $P:\mathbb{T}\to\mathbb{R}$ by 
		$$P(x)=g(0)+\int_{[0,x)}p(s)dV(s).$$
		By the choice of the partition above, we have that $\|P-g\|_{\infty}\le 2\epsilon.$ Finally, note that
		$$h(x)=f(0)+\int_{(0,x]}dW(\xi)\left(b+g(0)+\int_{[0,\xi)}p(\eta)dV(\eta)\right),$$
		belongs to $\mathfrak{D}_{W,V}$, where $b=-g(0)-(W(1))^{-1}\int_{(0,1]}dW(\xi)\int_{[0,\xi)}p(\eta)dV(\eta)$. The continuity of $f$ implies that $g(0)\to 0$ as $n\to\infty$. By using this last fact, we can find a constant $C>0$, that does not depend on $n$ nor on $k$, such that  $$\|h-G\|_{\infty}\le C\epsilon.$$ 
		Finally, the triangular inequality yields $\|h-f\|_{\infty}\le (C+2)\epsilon$. This proves $(a)$.
		
		\paragraph{}To prove (b), suppose $g$ is càdlag $W$-left differentiable function with $D^{-}_{W}g$ càglàd, and that $f$ is a càglàd $V$-right differentiable function with $D^{+}_{V}f$ càdlàg. For any $\epsilon>0$ choose a partition $\{a=z_{1}<z_{2}<\cdots<z_{n+1}=b\}$ of $[a,b]\subset\mathbb{T}$ such that
		\[\left\{ \begin{array}{ll}
			|g(s)-g(t)|\le\epsilon, \mbox{if}\; z_{k}\le s,t<z_{k+1};\\\\
			|f(s)-f(t)|\le\epsilon, \mbox{if} \;z_{k}< s,t\le z_{k+1}.\end{array} \right. \] 
		We, then, have
		\begin{align}\label{eq1}
			\nonumber A_{1}  &= \left|\int_{[a,b)}g(s)D^{+}_{V}f(s) dV - \sum_{k=1}^{n}g(z_{k})[f(z_{k+1})-f(z_{k})]\right| \\
			& = \left|\sum_{k=1}^{n}\int_{[z_{k},z_{k+1})}D^{+}_{V}f(s)[g(s)-g(z_{k})]dV(s)\right|\le \epsilon\|D^{+}_{V}f\|_{\infty}[V(1)].
		\end{align}
		On the other hand
		$$\sum_{k=1}^{n}g(z_{k})[f(z_{k+1})-f(z_{k})]=[f(b)g(b)-f(a)g(a)]-\sum_{k=2}^{n+1}f(z_{k})[g(z_{k})-g(z_{k-1})]$$
		and 
		\begin{align}\label{eq2}
			\nonumber A_{2}  &= \left|\int_{(a,b]}f(s)D^{-}_{W}g(s) dW - \sum_{k=2}^{n+1}f(z_{k})[g(z_{k})-g(z_{k-1})]\right| \\
			& = \left|\sum_{k=1}^{n}\int_{(z_{k-1},z_{k}]}D^{-}_{W}g(s)[f(s)-f(z_{k})]dW(s)\right|\le \epsilon\|D^{-}_{W}g\|_{\infty}[W(1)].
		\end{align}
		Finally, (\ref{eq1}) and (\ref{eq2}) imply
		\begin{equation}
			\nonumber \left| \int_{[a,b)}g(s)D^{+}_{V}f(s) dV(s)+\int_{(a,b]}f(s)D^{-}_{W}g(s) dW(s) -[f(b)g(b)-f(a)g(a)]\right|\le \epsilon C,
		\end{equation}
		where $C=\{\|D^{+}_{V}f\|_{\infty}[V(1)]+\|D^{-}_{W}g\|_{\infty}[W(1)]\}$. Since $\epsilon$ is arbitrary, we obtain
		\begin{equation}\label{e3}
			\int_{[a,b)}g(s)D^{+}_{V}f(s) dV(s)=[f(b)g(b)-f(a)g(a)]-\int_{(a,b]}f(s)D^{-}_{W}g(s) dW(s).
		\end{equation}
		Using (\ref{e3}), we can easily see that, for $f,g \in \mathfrak{D}_{W,V}(\mathbb{T})$, we have
		\begin{equation}\label{e4}
			\langle D^{+}_{V}D^{-}_{W}f,g\rangle_{V}=\int_{\mathbb{T}}g(s)D^{+}_{V}D^{-}_{W}f(s) dV(s)=-\int_{\mathbb{T}}D^{-}_{W}f(s)D^{-}_{W}g(s) dW(s).
		\end{equation}
		
		To prove $(c)$ note that  $$\int_{\mathbb{T}}f^{2}(\xi)dV(\xi)-V(1)f_{\mathbb{T}}^{2}=\dfrac{1}{V(1)^2}\int_{\mathbb{T}}\left(\int_{\mathbb{T}}[f(\xi)-f(\eta)]dV(\eta)\right)^2dV(\xi).$$
		On the right hand side we use that $f(\xi)-f(\eta)=\int_{(\eta,\xi]}D^{-}_{W}f(s)dW(s)$, and then apply H\"older's inequality to obtain $(\ref{i1})$.
	\end{proof} 

The following consequence of Theorem \ref{thm:1} will be very useful through this work, but first we need a definition.
	
	\begin{definition}
		Let $L^2_{W,0}(\mathbb{T})$ be the subspace of $L^2_W(\mathbb{T})$ consisting of functions with mean zero: 
		$$L^{2}_{W,0}(\mathbb{T})=\left\{H\in L^{2}_{W}(\mathbb{T}); \int_{\mathbb{T}}HdW=0\right\}.$$ 
	\end{definition}
	\begin{corollary}\label{cr:1}
		The set of $W$-left-derivatives of functions in $\mathfrak{D}_{W,V}$ is dense in $L^{2}_{W,0}(\mathbb{T})$ in the uniform topology. In particular,
		$$\overline{\{D^{-}_{W}f; f\in \mathfrak{D}_{W,V}\}}^{\|\cdot\|_{L^2_{W}(\mathbb{T})}}=L^2_{W,0}(\mathbb{T}).$$
	\end{corollary}
	\begin{proof}
		Fix a function $f\in L^{2}_{W,0}(\mathbb{T})$. By the version of Theorem 1 for $\mathfrak{D}_{V,W}$ (that is, related to the $V,W$-structure, see Remark \ref{orderVW}) there is a sequence $p_{n}\in \mathfrak{D}_{V,W}$ such that $p_{n}\to f$ uniformly. Now, let $$f_{n}(x):=\int_{(0,x]}p_{n}dW-\dfrac{W(x)}{W(1)}\int_{\mathbb{T}}p_{n}dW.$$
		We have that $f_{n}\in \mathfrak{D}_{W,V}$ and $D_{-}f_n = p_{n}-\frac{1}{W(1)}\int_{\mathbb{T}}p_{n}dW.$  The triangular inequality yields
		$$\|D^{-}_{W}f_{n}-f\|_{\infty}\le \left|\frac{1}{W(1)}\int_{\mathbb{T}}p_{n}dW\right|+\|p_{n}-f\|_{\infty}.$$
		Using the uniform convergence of $p_{n}$ and the last estimate, we have that $D^{-}_{W}f_{n}\to f$ uniformly in $\mathbb{T}$. This yields the desired convergence in $L^{2}_{W}(\mathbb{T}).$
	\end{proof}

\begin{remark}
	Notice that the above Corollary is proved explicitly in a transparent manner. The idea was to exploit the symmetric properties between the theory in $(W,V)$ and the theory in $V,W$.
\end{remark}

\section{$W$-$V$-Sobolev spaces}

	\begin{definition}\label{sobolev1}
		We define the first order $W$-$V$-Sobolev space, denoted by $H_{W,V}(\mathbb{T})$, as the energetic space (in the sense of Zeidler, \cite[Section 5.3]{zeidler}) associated to the operator $(I-D^{+}_{V}D^{-}_{W}): \mathfrak{D}_{W,V}(\mathbb{T})\subset L^2_{V}(\mathbb{T})\to L^2_{V}(\mathbb{T})$. That is, we define the norm $$\|f\|_{1,2}^2 = \langle I - D^{+}_{V}D^{-}_{W}f,f\rangle_{V}$$
		in $\mathfrak{D}_{W,V}(\mathbb{T})$ and say that $f\in L^2_{V}(\mathbb{T})$ belongs to $H_{W,V}(\mathbb{T})$ if, and only if, the following conditions hold:
		\begin{enumerate}
			\item There exists a sequence $f_n\in \mathfrak{D}_{W,V}(\mathbb{T})$ such that $f_n\to f$ in $L^2_V(\mathbb{T})$;
			\item The sequence $f_n$ is Cauchy with respect to the energetic norm $\|\cdot\|_{1,2}$.
		\end{enumerate}
		A sequence $(f_n)_{n\in\mathbb{N}}$ in $\mathfrak{D}_{W,V}(\mathbb{T})$ satisfying 1 and 2 is called an \emph{admissible sequence}. 
	\end{definition}
	
	\begin{remark}
		Notice that, by condition 2 of Definition \ref{sobolev1}, the norm $\|\cdot\|_{1,2}$ can be uniquely extended to $H_{W,V}(\mathbb{T})$. Therefore, we endow $H_{W,V}(\mathbb{T})$ with this extended norm, which we will also denote by $\|\cdot\|_{1,2}$.
	\end{remark}
	
	\begin{theorem}\label{sobchar}
		A function $f\in L^{2}_{V}(\mathbb{T})$ belongs to $H_{W,V}(\mathbb{T})$ if, and only if, there is a function $F\in L^{2}_{W,0}(\mathbb{T})$ and a finite constant c such that
		\begin{equation}\label{c3}
			f(x)=c+\int_{(0,x]}F(s)dW(s),    
		\end{equation}
		$V$-a.e.
	\end{theorem}
	\begin{proof}
		Let $f\in H_{W,V}(\mathbb{T})$. By the definition of energetic space and the relation (\ref{e4}), there exists an admissible sequence $(f_{n})_{n\in\mathbb{N}}$ such that $f_{n}\to f$ in $L^{2}_{V}(\mathbb{T})$ and $\left(D^{-}_{W}f_{n}\right)_{n\in\mathbb{N}}$ is Cauchy in $L^{2}_{W}(\mathbb{T})$ norm. Without loss of generality, we can suppose that $D^{-}_{W}f_{n}\to G$ in $L^{2}_{W}(\mathbb{T})$. Define $\displaystyle g(x)=\int_{(0,x]}G(\xi)dW(\xi)$ and note that $g(1)=0$ and
		$$g(y)-g(x)=\int_{(x,y]}G dW=\lim_{n\to \infty}\int_{(x,y]}D^{-}_{W}f_{n}dW=\lim_{n\to \infty}[f_{n}(y)-f_{n}(x)].$$
		Now, we need to prove that, for each fixed $y\in\mathbb{T}$, we have
		\begin{equation}\label{e5} \int_{\mathbb{T}}[f_{n}(y)-f_{n}(x)]dV(x)\to \int_{\mathbb{T}}[g(y)-g(x)]dV(x).
		\end{equation} 
	Indeed, we have that $f_{n}(y)-f_{n}(x)\to g(y)-g(x)$ for each fixed $y$, and by the H\"older's inequality
		$$|f_{n}(y)-f_{n}(x)|^2\le W(1)\int_{\mathbb{T}}\left(D^{-}_{W}f_{n}\right)^{2}dW,$$
		with the term on the right-hand of the above inequality being bounded due to the fact that $\left(D^{-}_{W}f_{n}\right)_{n\in\mathbb{N}}$ is Cauchy. We can now use the dominated convergence theorem to obtain (\ref{e5}). Note that the convergence $f_{n}\to f$ in $L_{V}^{2}(\mathbb{T})$ implies that, $V$-a.e, we have
		\begin{align*}
			V(1)f(y)=\lim_{n\to\infty} f_{n}(y) &= \lim_{n\to\infty}\left[V(1)f_{n}(y)-\int_{\mathbb{T}}f_{n}(x)dV(x)\right]+\lim_{n\to\infty}\int_{\mathbb{T}}f_{n}(x)dV(x)\\ &= V(1)g(y)-\int_{\mathbb{T}}g(x)dV+\int_{\mathbb{T}}fdV.
		\end{align*} 
		That is, for $V$-a.e. $y$, we have
		$$f(y)= c+\int_{(0,y]}G(\xi)dW(\xi),$$
		where $c=V(1)^{-1}\left(\int_{\mathbb{T}}fdV-\int_{\mathbb{T}}g(x)dV\right)$. 
		Conversely, if  $F\in L^{2}_{W,0}$ satisfies (\ref{c3}), we have by Corollary \ref{cr:1} that $\left\{D^{-}_{W}g; g\in \mathfrak{D}^{2}_{W,V}(\mathbb{T})\right\}$ is dense in the closed subspace $L^{2}_{W,0}(\mathbb{T})$. Now, choose $\left(D^{-}_{W}g_{n}\right)_{n\in\mathbb{N}}$ such that $D^{-}_{W}g_{n}\to F$ in $L^{2}_{W}(\mathbb{T})$, and let
		$$f_{n}(x)=c+\int_{(0,x]}D^{-}_{W}g_{n}dW.$$
		It follows from the dominated convergence theorem that $f_{n}\to f$ in $L^{2}_{V}(\mathbb{T})$ and $f_{n}\in \mathfrak{D}_{W,V}(\mathbb{T})$ for all $n\in\mathbb{N}$, that is, $g_{n}$ is admissible for $f$. \end{proof}
	\begin{remark}\label{sobagree}
		If we set $V(x)=x$ in the above Lemma, our space agrees with the energetic space considered in \cite{franco}, that is,  $H_{W,x}(\mathbb{T})=H^{1}_{2}(\mathbb{T})$, where $H^{1}_{2}(\mathbb{T})$ is the space introduced in \cite{franco}. 
		It is noteworthy that even though we defined our space in terms of one-sided derivatives, the resulting energetic space for that particular case in which $V(x)=x$ agrees with the space defined in \cite{franco}, where "two-sided" derivatives were used. Note that the energetic space $H_2^1(\mathbb{T})$ in \cite{franco} coincides with the $W$-Sobolev space in \cite{wsimas} when the dimension is 1. This shows that our approach for the $W$-$V$-Sobolev space is a natural one, as it generalized the standard $W$-Sobolev spaces, and in particular, our $W$-$V$-Sobolev space also agree with the standard Sobolev space when $V(x)=W(x)=x$.
	\end{remark}

We will now prove a $W,V$-version of the well-known Rellich-Kondrachov's theorem. The proof is similar to the proof in \cite[Lemma 3]{franco}. We will present the details here for completeness.

	\begin{theorem}\label{compactemb}
		The embedding $H_{W,V}(\mathbb{T})\hookrightarrow L^{2}_{V}(\mathbb{T})$ is compact.
	\end{theorem}
	\begin{proof}
		Let $u_{n}(x)=c_{n}+\int_{(0,x]}U_{n}dW$ be a bounded sequence in $H_{W,V}(\mathbb{T})$, where $U_{n}\in L^{2}_{V,0}(\mathbb{T})$, $n\in\mathbb{N}$. By the defnition of the energetic norm, we have that $\|u_{n}\|^{2}_{1,2}=\|U_{n}\|_{W}^{2}+\|u_{n}\|_{V}\geq\|U_{n}\|_{W}^{2}$. Therefore, $U_{n}$ is bounded in $L^{2}_{W}(\mathbb{T})$. We can then use Cauchy-Swartz's inequality to obtain that $\int_{(0,x]}U_{n}dW$ is bounded in $L^{2}_{V}(\mathbb{T})$. Therefore 
		$$c_{n}=u_{n}(x)-\int_{(0,x]}U_{n}dW$$
		 is a bounded, since the right-hand side of this expression is bounded in $L^{2}_{V}(\mathbb{T})$. Further, since $L^2_W(\mathbb{T})$ is a separable Hilbert space and the sequence $(U_{n})_{n}$ is bounded in $L^{2}_{W}(\mathbb{T})$, there is a subsequence $(U_{n_{k}})_{k\in\mathbb{N}}$ such that $$U_{n_{k}} \rightharpoonup U$$ for some $U \in L^{2}_{W}(\mathbb{T})$. This shows that $c_{n_{k}}\to c$ for some $c\in \mathbb{R}$. Moreover, for all $x\in\mathbb{T}$, we have $\textbf{1}_{(0,x]}\in L^{2}_{W}(\mathbb{T})$, and thus
		$$\lim_{n\to\infty}u_{n_{k}}(x)=\lim_{n\to\infty}\left\{c_{n}+\int_{(0,x]}U_{k}dW\right\}=c+\int_{(0,x]}UdW.$$
		In addiction, $\int_{\mathbb{T}}UdW=0$. Finally, $|u_{n_{k}}(x)|^2\le 2c_{n_{k}}^2+2[W(1)]\|U_{n_{k}}\|_{W}^2$. Therefore, we can use the
		 dominated convergence theorem to conclude that $u_{n_{k}}$ converges to $c+\int_{(0,x]}UdW$ in $L^{2}_{W}(\mathbb
		{T})$.
	\end{proof}
	
	\paragraph{}Our goal now is to define a ($W$,$V$)-generalized laplacian. To this end, we will provide some auxiliary definitions.
	
	\begin{definition}
		Let $\mathcal{A}: \mathcal{D}_{W,V}(\mathbb{T})\subseteq L^{2}_{V}(\mathbb{T})\to L^{2}_{V}(\mathbb{T})$ be the Friedrichs Extension of the operator $I-D^{+}_{V}D^{-}_{W}$ (we refer the reader to Zeidler \cite[Section 5.5]{zeidler} for further details on Friedrichs extensions). 
		We can characterize the domain $\mathcal{D}_{W,V}(\mathbb{T})$ as the set of functions $f\in L^{2}_{W}(\mathbb{T})$ such that there is $\mathfrak{f}\in L^{2}_{V,0}(\mathbb{T})$ satisfying 
		\begin{equation}\label{c4}
			f(x)=a+W(x)b+\int_{(0,x]}\int_{[0,y)} \mathfrak{f}(s) dV(s)dW(y),
		\end{equation}
		where $b$ is satisfies the relation
		$$bW(1)+\int_{(0,1]}\int_{[0,y)} \mathfrak{f}(s) dV(s)dW(y)=0.$$
		Moreover,
		\begin{equation}\label{c5}
			-\int_{\mathbb{T}}D^{-}_{W}f D^{-}_{W}g dW= \langle\mathfrak{f},g\rangle
		\end{equation}
		for all g in $H_{W,V}(\mathbb{T}).$ 
	\end{definition}

\begin{remark}
	Expression \eqref{c5} follows from a straightforward adaptation of the arguments found in \cite[Lemma 4]{franco}.
\end{remark}

	Hence, by \cite[Theorem 5.5.c]{zeidler} and Theorem \ref{compactemb} above, the resolvent of the Friedrichs extension, $\mathcal{A}^{-1}$, is compact. Thus, there exists a complete ortonormal system of functions $(\nu_n)_{n\in\mathbb{N}}$ in $L^{2}_{V}(\mathbb{T})$ such that $\nu_{n}\in H_{W,V}(\mathbb{T})$ for all $n$, and $\nu_{n}$ solves the equation $\mathcal{A}\nu_n=\gamma_n \nu_n,$ for some $\{\gamma_n\}_{n\in\mathbb{N}}\subset\mathbb{R}.$ Furthermore,
	$$1\le \gamma_{1}\le\gamma_{2}\cdots\to\infty$$ 
	as $n\to\infty$. 
	
	\begin{definition}\label{laplacianWV}
		We define the $W$-$V$-Laplacian as $\Delta_{W,V}=I-\mathcal{A} : \mathcal{D}_{W,V}(\mathbb{T})\subseteq L^{2}_{V}(\mathbb{T})\to L^{2}_{V}(\mathbb{T})$.
	\end{definition}
	
We have the following integration by parts formula with respect to the ($W$,$V$)-Laplacian:
	
	\begin{proposition}\label{intbypartsprop}
		(Integration by parts formula) For every $f\in\mathcal{D}_{W,V}(\mathbb{T})$ and $g\in H_{W,V}(\mathbb{T})$, the following expression holds:
		\begin{equation}\label{intbypartsVW}
			\langle -\Delta_{W,V}f,g\rangle_{V} = \int_{\mathbb{T}}D^{-}_{W}fD^{-}_{W}gdW
		\end{equation}
	\end{proposition}
	\begin{proof}
		We have by (\ref{c4}) that $\Delta_{W,V}f=\mathfrak{f}$. The result thus follows by (\ref{c5}).
	\end{proof} 

Note that in expression \eqref{intbypartsVW}, the spaces change (more precisely, the measures in which the integrals are being done change). On the left-hand side of expression \eqref{intbypartsVW}, the integration is being done on $L^2_V(\mathbb{R})$, with respect to the measure induced by $V$, whereas in the right-hand side of \eqref{intbypartsVW}, the integration is being done on $L^2_W(\mathbb{T})$, with respect to the measure induced by $W$.
	
	\section{The Space $C^{\infty}_{W,V}(\mathbb{T})$}\label{sect4}
	
	Our goal in this section is to provide a ``nice'' space for functions in which the lateral derivatives exist pointwisely. This means, for example, that whenever we have a function on such spaces, say $f$, the Laplacian $\Delta_{W,V}f$ will agree with $D_V^+D_W^-f$, whereas this last expression exists for every point in $\mathbb{T}$. 
	This was a problem with the previous $W$-Sobolev spaces \cite{wsimas,simasvalentim2,simasvalentimjde}, and also when people dealt with such operators such as in \cite{franco, valentim, frag, farfansimasvalentim}. All these cases considered the operator $D_xD_W$ (without the lateral derivatives). The problem is that the eigenfunctions for that operator were obtained in $L^2(\mathbb{T})$ and there were no satisfactory regular spaces for such functions. We will now explain the reason for that. Let $f$ be an eigenfunction for $D_xD_W$. We know that if $f$ admits the $W$-derivative, then $f$ is c\`adl\`ag and typically discontinuous, with its discontinuity points being cointained in the set of discontinuity points of $W$. When considering the operator $D_xD_W$, this means that $D_W f$ is differentiable. Now, notice that $D_xD_W f = \alpha f$, for some $\alpha$, which is c\`adl\`ag, since $f$ is c\`adl\`ag. This means that we have a c\`adl\`ag function that is the derivative of some function. We can now use Darboux's theorem from real analysis to conclude that $D_xD_W f$ must be continuous, which in turn, implies that $f$ must be continuous. This is not expected to be true. Indeed, one can consider the case in which the set of discontinuity points of $W$ is dense. This explains why, up until the moment, no pointwise higher-order regularity results have been found for such operator. Indeed, in \cite{simasvalentimjde}, to be able to do point evaluation on ``regular'' functions they needed to define a new space of test functions and, also, prove several results to circumvent the fact that the ``natural'' space spanned by the eigenfunctions was seen as a subspace of $L^2(\mathbb{T})$, and thus, the point evaluation was not defined.
	
	We will now define our space of regular functions in which we can apply the operators pointwisely. We will also show that the eigenfunctions belong to this space.
	
To this end, begin by noting that $\nu\in L^2_V(\mathbb{T})$ is an eigenvector of $\Delta_{W,V}$ with eigenvalue $\lambda$ if, and only if, $\nu$ is an eigenvector of $\mathcal{A}$ with eigenvalue $\gamma=1+\lambda$. We, then, have that $\{\nu_{n},\lambda_{n}\}_{n\in \mathbb{N}}$ forms a complete orthonormal system of $L^{2}_{V}(\mathbb{T})$, where $\lambda_n = \gamma_n - 1$. Furthermore, $\nu_{n}\in H_{W,V}(\mathbb{T})$ for all $n$. Moreover, we have that
	$$0=\lambda_{0}\le \lambda_{1}\le\lambda_{2}\cdots\to\infty,$$
	as $n\to\infty$, where for each $k\in\mathbb{N} \setminus \{0\}$, $\lambda_k$ satisfies
	
	\begin{equation}\label{autovt}
		-\Delta_{W,V}\nu_{k}=\lambda_{k} \nu_{k}.
	\end{equation} 
	
To state the our next regularity result we need to define the following sets : 
	$$C_{0}(\mathbb{T}):=\left\{f:\mathbb{T}\to\mathbb{R}; f\;\;\mbox{is càdlàg},\;\int_{\mathbb{T}}fdV=0\right\},$$
	$$C_{1}(\mathbb{T}):=\left\{f:\mathbb{T}\to\mathbb{R}; f\;\;\mbox{is càglàd}\int_{\mathbb{T}}fdW=0\right\},$$
	and
	$$C^{n}_{W,V,0}(\mathbb{T}):=\left\{f\in C_{0}(\mathbb{T}); D^{(k)}_{W,V}f\hbox{ exists and }D^{(k)}_{W,V}f\in C_{\sigma(k)}(\mathbb{T}),\,\forall k\le n\right\},$$
	where 
	\[D^{(n)}_{W,V}:=\left\{ \begin{array}{ll}
		\underbrace{D^{-}_{W}D^{+}_{V}...D^{-}_{W}}_{n-factors},\;\mbox{if}\; n\; \mbox{is odd}; \\
		\underbrace{D^{-}_{V}D^{+}_{W}...D^{-}_{W}}_{n-factors},\;\mbox{if}\; n\; \mbox{is even},\end{array} \right. \] 
	and 
	\[\sigma(n):=\left\{ \begin{array}{ll}
		1,\;\mbox{if}\; n\; \mbox{is odd}; \\
		0,\;\mbox{if}\; n\; \mbox{is even}.\end{array} \right. \] 
	
	We are now in a position to state and prove the regularity result. Note also that we will also define our first space of regular functions, namely, $C^{\infty}_{W,V,0}(\mathbb{T})$. Later we will increase it by adding (in the sum of spaces sense) the constant functions.
	
	\begin{theorem}[Regularity of eigenfunctions]\label{eigenreg} The solutions of 
		\[\left\{ \begin{array}{ll}
			-\Delta_{W,V}\psi=\lambda \psi,\; \lambda\in\mathbb{R}\setminus\{0\};\\
			u\in \mathcal{D}_{W,V}(\mathbb{T})\end{array} \right. \] 
		belong to \begin{equation}\label{cinfinity}
			C^{\infty}_{W,V,0}(\mathbb{T}):=\bigcap_{n\in\mathbb{N}}C^{n}_{W,V,0}(\mathbb{T}).
		\end{equation}
	\end{theorem}
	\begin{proof}
		Begin by observing that we have, by $(\ref{c4})$, that
		\begin{equation}\label{eq:n}
			\psi(x)=a+bW(x)-\lambda\int_{(0,x]}\int_{[0,y)} \psi(s) dV(s)dW(y)
		\end{equation}
		V-$a.e$ with $a$ and $b$ determined by the relations
		\begin{equation}\label{eq:n+3}
			bW(1)-\lambda\int_{\mathbb{T}}\int_{[0,y)} \psi(s) dV(s)dW(y)=0\,;\,\,\,\int_{\mathbb{T}}\psi dV=0.
		\end{equation}
	This shows that $\varphi\in C_0(\mathbb{T})$. This also shows that to conclude the result, it is enough to show that $D_W^-\varphi\in C_1(\mathbb{T})$, $D_V^+(D^-_W\varphi) \in C_0(\mathbb{T})$ and that $D_V^+(D^-_W\varphi) = \varphi$, $V$-a.e. Since, then, the remaining orders follow directly by induction.
	Therefore, the next step is to prove that $D_W^-\varphi\in C_1(\mathbb{T})$. Indeed, note that $\int_{[0,y)}\psi(s)dV(s)$ is a càglàd function and
		\begin{equation}\label{eq:n+2}
			D_{W}^{-}\left(a+bW-\lambda\int_{(0,\cdot]}\int_{[0,\alpha)} \psi(\beta) dV(\beta)dW(\alpha)\right)(y)=b-\lambda\int_{[0,y)}\psi(s)dV(s)
		\end{equation}
		$W$-a.e. Note that $(\ref{eq:n+3})$ implies the mean zero conditions over $(\ref{eq:n+2})$. This shows  $D_W^-\varphi\in C_1(\mathbb{T})$. Now, observe that  $(\ref{eq:n})$ implies
		\begin{equation}\label{eq:n+1}
			b-\lambda\int_{[0,y)}\psi(s)dV(s)=b-\lambda\int_{[0,y)}\left[a+\int_{(0,s]}\left(b-\lambda\int_{[0,\alpha)} \psi(\beta) dV(\beta)\right)dW(\alpha)\right]dV(s).
		\end{equation}
		Since $a+bW(x)-\lambda\int_{(0,x]}\int_{[0,y)} \psi(s) dV(s)dW(y)$ is càdlàg, it follows by $(\ref{eq:n+1})$ that
		\begin{equation}\label{eq:n+4}
			D^{+}_{V}\left(b-\lambda\int_{[0,\cdot)}\psi(\beta)dV(\beta)\right)(s)=a+\int_{(0,s]}\left(b-\lambda\int_{[0,\alpha)} \psi(\beta) dV(\beta)\right)dW(\alpha).
		\end{equation}
		$V$-a.e. Since $(\ref{eq:n+3})$ also implies the mean zero conditions over $(\ref{eq:n+4})$, this shows that $D_V^+(D^-_W\varphi) \in C_0(\mathbb{T})$ and that $D_V^+(D^-_W\varphi) = \varphi$, $V$-a.e. 
	\end{proof}

Note that the notion of regularity provided by Theorem \ref{eigenreg} is a true notion of regularity, in the sense that we can evaluate any eigenfunction as well as their lateral derivatives of any order at every point of the domain. This is a major contribution to the regularity theory related to such operators. 

We can now increase our space by ``adding'' the constant functions.

\begin{definition}
	Let
	$$C^{\infty}_{V,W}(\mathbb{T}):=\langle 1\rangle \oplus C^{\infty}_{V,W,0}(\mathbb{T}),$$
	where $C^{\infty}_{V,W,0}(\mathbb{T})$ was defined in \eqref{cinfinity}.
\end{definition}

Since the eigenvectors of $\Delta_{V,W}$ are dense in $L^2_V(\mathbb{T})$, we can apply Theorem \ref{eigenreg} to conclude that 

\begin{proposition}\label{densenesscinfinity}
	The space $C^{\infty}_{V,W}(\mathbb{T})$ is dense in $L^2_V(\mathbb{T})$. Furthermore, the set $\left\{D^{+}_{V}g;g\in C^{\infty}_{V,W}(\mathbb{T})\right\}$ is dense in $L^2_{V,0}(\mathbb{T})$.
\end{proposition}
\begin{proof}
	The density of $C^{\infty}_{V,W}(\mathbb{T})$ in $L^2_V(\mathbb{T})$ follows from the comments before the proposition. On the other side, the density of $A=\left\{D^{+}_{V}g;g\in C^{\infty}_{V,W}(\mathbb{T})\right\}$ in $L^{2}_{V,0}(\mathbb{T})$ is due to the fact that every non null eigenvalue of $\Delta_{W,V}$ is an eigenvalue of $\Delta_{V,W}$ and each invariant subspace (that is, each eigenspace) have a fixed dimension. This implies that for each nonzero eigenvalue, we can find the set of the corresponding eigenvectors of $\Delta_{W,V}$  in $A$. Now, we can use the fact that the space generated set of eigenvectors associated to nonzero eigenvalues is dense in $L^{2}_{V,0}(\mathbb{T})$. This proves the result.
\end{proof}

\section{$W$-$V$-Sobolev spaces and weak derivatives}

Our goal now is to define the notion of $W$ and $V$ lateral weak derivatives and show that the our $W$-$V$-Sobolev space can be viewed as a space of functions that have lateral weak derivative with respect to $W$.  Thus, let us define the notion of lateral weak derivative:
	
	\begin{definition}
		A function $f\in L^{2}_{V}(\mathbb{T})$ has a $W$-left \textit{weak} derivative if, and only if, for every $g\in  C^{\infty}_{V,W}(\mathbb{T})$, there exists $F\in L^{2}_{W}(\mathbb{T})$ such that 
		\begin{equation}\label{defWd}
			\int_{\mathbb{T}}fD_{V}^{+}gdV=-\int_{\mathbb{T}}FgdW.
		\end{equation}
	\end{definition}
	\begin{remark}
		Note that since the eigenvectors of $\Delta_{V,W}$ belong to $C^{\infty}_{V,W}(\mathbb{T})$, it follows that $C^{\infty}_{V,W}(\mathbb{T})$ is dense in $L^2_V(\mathbb{T})$.  This implies the uniqueness of the $W$-left weak derivative defined by \eqref{defWd}. We will denote the $W$-left weak derivative of $f$ by $\partial^{-}_{W}f.$ 
	\end{remark}
	\begin{remark}
		If there exists $F$ satistying (\ref{defWd}) for all $g\in C^{\infty}_{V,W}(\mathbb{T})$, then by taking $g\equiv 1$, we have that $\int_{\mathbb{T}}F dW =0$, i.e., $F\in L^{2}_{W,0}(\mathbb{T}).$\end{remark}
	
	\paragraph{}Now, define the ($W,V$)-Sobolev space $$\widetilde{H}_{W,V}(\mathbb{T})=\{f\in L^{2}_{V}(\mathbb{T});  \partial _{W}^{-}f\in L^2_{W,0}(\mathbb{T})\hbox{ exists} \}.$$ 
	
	One can readily prove that $\widetilde{H}_{W,V}(\mathbb{T})$ is a Hilbert space with the energetic norm $\|f\|^{2}_{W,V}=\|f\|^{2}_{V}+\|\partial_{W}^{-}f\|^{2}_{W}$. But actually we have more:
	
	\begin{theorem}\label{sobequal}
		We have the following equality of sets $\widetilde{H}_{W,V}(\mathbb{T})=H_{W,V}(\mathbb{T})$
	\end{theorem}
	\begin{proof}
		Let $f\in H_{W,V}(\mathbb{T})$ and $(f_{n})_{n}$ is admissible sequence for $f$. Then, by the integration by parts formula (Proposition \ref{intbypartsprop}), we have that, for all $g\in C^{\infty}_{V,W}(\mathbb{T}),$
		\begin{equation}\label{x1}
			\int_{\mathbb{T}}f_{n}D^{+}_{V}gdV=-\int_{\mathbb{T}}D^{-}_{W}f_{n}gdW.
		\end{equation}
		We can take the limit as $n\to\infty$ to conclude that, for all $g\in C^{\infty}_{V,W}(\mathbb{T})$,
		$$\int_{\mathbb{T}}fD^{+}_{V}gdV=-\int_{\mathbb{T}}D^{-}_{W}fgdW,$$
		i.e. $D^{-}_{W}f = \partial^{-}_{W}f$. This implies that ${H}_{W,V}(\mathbb{T})\subseteq \widetilde{H}_{W,V}(\mathbb{T}).$
		For the reverse inclusion, take $f\in \widetilde{H}_{W,V}(\mathbb{T}).$ Choose a sequence $(f_{n})_{n\in\mathbb{N}}$ with $f_{n}\in \mathfrak{D}_{W,V}$, such that $D^{-}_{W}f_{n}\to \partial^{-}_{W}f$ in $L^{2}_{W,0}(\mathbb{T})$. This implies
		\begin{equation}\label{15}
			f_{n}(x)-f_{n}(0)=\int_{(0,x]}D^{-}_{W}f_{n}dW\to\int_{(0,x]}\partial^{-}_{W}fdW.
		\end{equation}
		On the other hand, for all $g\in C^{\infty}_{V,W}(\mathbb{T})$, we have 
		\begin{equation*}\label{16}
			\int_{\mathbb{T}}\left(f_{n}-f_{n}(0)\right)D^{+}_{V}g dV = \int_{\mathbb{T}} f_{n}D^{+}_{V}g dV = -\int_{\mathbb{T}} D^{-}_{W}f_{n}g dW
		\end{equation*}
		as $n\to \infty$. Now, we have, by (\ref{15}), that 
		$$\int_{\mathbb{T}}\left(\int_{(0,x]}\partial^{-}_{W}fdW\right)D^{+}_{V}gdV=-\int_{\mathbb{T}}\partial^{-}_{W}f g dW=\int_{\mathbb{T}}fD^{+}_{V}g dV,$$
		i.e.,
		\begin{equation}\label{3k}
			\int_{\mathbb{T}}\left(c+\int_{(0,x]}\partial^{-}_{W}fdW-f\right)D^{+}_{V}g dV=0.
		\end{equation}
		where $c$ is such that $c+\int_{(0,x]}\partial^{-}_{W}fdW-f\in L^{2}_{V,0}(\mathbb{T})$. This shows that (\ref{3k}) is true for all $g\in C^{\infty}_{V,W}(\mathbb{T})$. By the density of $\left\{D^{+}_{V}g;g\in C^{\infty}_{V,W}(\mathbb{T})\right\}$ in $L^{2}_{V,0}(\mathbb{T})$ (see Proposition \ref{densenesscinfinity}), we have that
		$$f=c+\int_{(0,x]}\partial^{-}_{W}fdW,$$
		i.e., $f\in H_{V,W}(\mathbb{T}).$
		
	\end{proof}

	\begin{remark}
		By the last theorem actually we have that $\partial^{-}_{W}f=D^{-}_{W}f$ when it exists.
	\end{remark}

\begin{remark}
	It is important to note that the proof of Theorem \ref{sobequal} corrects a mistake in a previous proof of a similar result \cite[Proposition 2.5]{wsimas}. Indeed, in that proof, Banach-Steinhaus theorem was incorrectly applied. Thus, its consequences are incorrect.
\end{remark}

	\begin{theorem}[Characterizarion of the Sobolev spaces in terms of Fourier coefficients]\label{thm:6}
		The following characterization of $H_{W,V}(\mathbb{T})$ is true 
		$$H_{W,V}(\mathbb{T})=\left\{f\in L^{2}_{V}(\mathbb{T}); f=\alpha_{0}+\sum_{i=1}^{\infty}\alpha_{i}\nu_{i}; \sum_{i=1}^{\infty}\lambda_{i}\alpha_{i}^2<\infty\right\}.$$
	\end{theorem}
	
To prove the above theorem, we need to prove some auxiliary results. First, define $$\mathcal{W}=\left\{ D^{-}_{W}f; f\in H_{W,V}(\mathbb{T}) \right\}\subset L^{2}_{W,0}(\mathbb{T}).$$ 
	
	\begin{lemma}
		$\mathcal{W}$ is a closed subspace of $L^{2}_{W,0}(\mathbb{T})$
	\end{lemma}
	
	\begin{proof}Let us assume, without loss of generality, that $D^{-}_{W}f$ with $f\in L^{2}_{V,0}(\mathbb{T})$, otherwise if $f$ belongs to $L^{2}_{V}(\mathbb{T})$, we replace it by $f-\int f dV$, and its $W$-left derivative does not change. Now, let $D^{-}_{W}f_{n}$ be a sequence in $\mathcal{W}$  converging to $g$ in $L^{2}_{W}(\mathbb{T})$. In particular, $D^{-}_{W}f_{n}$ is Cauchy in $L^{2}_{W}(\mathbb{T})$. Thus, by (\ref{i1}), $f_{n}$ is cauchy in $L^{2}_{V}(\mathbb{T})$ and this implies that $f_{n}$ is cauchy in $H_{W,V}(\mathbb{T})$ in the energetic norm. Since $H_{W,V}(\mathbb{T})$ is complete in the energetic norm, there exists $f\in H_{W,V}(\mathbb{T})$ such that $f_{n}\to f$ in energetic norm. That is, $f_{n}\to f $ in  $L^{2}_{V}(\mathbb{T})$, and $D^{-}_{W}f_{n}\to D^{-}_{W}f$ in $L^{2}_{W}(\mathbb{T})$. By the uniqueness of limit, we obtain $g=D^{-}_{W}f$.
		
	\end{proof}
	
The above lemma tells us that $\mathcal{W}$ is a Hilbert space with respect to the  $L^{2}_{W}(\mathbb{T})$ norm. Let us prove another auxiliary result.
	\begin{lemma}
		The set $\left\{\frac{1}{\sqrt{\lambda_{i}}}D^{-}_{W}\nu_{i}\right\}_{i=1}^{\infty}$, where $\nu_{k}$ is given satisfying (\ref{autovt}), is a complete orthonormal set in $\mathcal{W}$.
		
	\end{lemma}
	
	\begin{proof} First, note that
		
		\begin{align*}
			\left \langle \dfrac{1}{\sqrt{\lambda_{i}}}D^{-}_{W}\nu_{i},\dfrac{1}{\sqrt{\lambda_{i}}}D^{-}_{W}\nu_{i} \right\rangle_{W} &= \dfrac{1}{\sqrt{\lambda_{i}\lambda_{j}}}\int_{\mathbb{T}}D^{-}_{W}\nu_{i}D^{-}_{W}\nu_{j}dW\\ &= \dfrac{1}{\sqrt{\lambda_{i}\lambda_{j}}}\int_{\mathbb{T}}\nu_{i}D^{+}_{V}D^{-}_{W}\nu_{j}dV \\ &= \dfrac{\sqrt{\lambda_{i}}}{\sqrt{\lambda_{j}}}\delta_{i,j}
		\end{align*} 
		where $\delta_{i,j}$ stands for the Kronecker's delta. Let us now prove the completeness of the system. If
		$D^{-}_{W}g\perp D^{-}_{W}\nu_{i}$ for all $i=1,2,3,\dots$, then
		$$0=\int_{\mathbb{T}}D^{-}_{W}gD^{-}_{W}\nu_{j}dW = \int_{\mathbb{T}}gD^{+}_{V}D^{-}_{W}\nu_{j}dV \Rightarrow \int_{\mathbb{T}}g\nu_{i}dV=0.$$
		This means that $g$ is constant and, thus, $D^{-}_{W}g=0.$
	\end{proof}
	
	The above lemma helps us in relating the Fourier coefficients of functions in $H_{W,V}(\mathbb{T})$ with the eigenvalues of $\Delta_{V,W}$. Indeed, let $f\in H_{W,V}(\mathbb{T})$. In particular $f\in L^{2}_{V}(\mathbb{T})$. This implies that there are $\{\alpha_{i}\}_{i=0}^{\infty}$ such that
	$$f=\alpha_{0}+\sum_{i=1}^{\infty}\alpha_{i}\nu_{i},$$
	where $$\alpha_{0}=\int_{\mathbb{T}}fdV,\;\;\alpha_{i}=\int_{\mathbb{T}}f\nu_{i}dV.$$ 
	We also have $$D^{-}_{W}f=\sum_{i=1}^{\infty}\alpha_{i}D^{-}_{W}\nu_{i}.$$
	In fact, by the previous lemma
	$$D^{-}_{W}f=\sum_{i=1}^{\infty}\beta_{i}\dfrac{D^{-}_{W}\nu_{i}}{\sqrt{\lambda_{i}}},$$
	where $$\beta_{i}=\dfrac{1}{\sqrt{\lambda_{i}}}\int_{\mathbb{T}}D^{-}_{W}fD^{-}_{W}\nu_{i}dW=\sqrt{\lambda_{i}}\int_{\mathbb{T}}f\nu_{i} dV= \sqrt{\lambda_{i}}\alpha_{i}.$$ We are now in a position to prove Theorem \ref{thm:6} 
	\begin{proof} [Proof of Theorem \ref{thm:6} ] Let $f\in H_{W,V}(\mathbb{T})$. By the above remarks, we have that
		$$D^{-}_{W}f = \sum_{i=1}^{\infty} \alpha_{i}D^{-}_{W}\nu_{i}=\sum_{i=1}^{\infty}\sqrt{\lambda_{i}}\alpha_{i}\dfrac{D^{-}_{W}\nu_{i}}{\sqrt{\lambda_{i}}}.$$
		Since $\left\{\frac{1}{\sqrt{\lambda_{i}}}D^{-}_{W}\nu_{i}\right\}_{i=1}^{\infty}$ is a complete orthonormal set we have
		$$\|D^{-}_{W}f\|_{W}^{2}=\sum_{i=1}^{\infty}\lambda_{i}\alpha_{i}^{2}<\infty.$$
		To prove the reverse inclusion, note that if $f=\alpha_{0}+\sum_{i=1}^{\infty}\alpha_{i}\nu_{i}$ satisfies $$\sum_{i=1}^{\infty}\lambda_{i}\alpha_{i}^{2}<\infty,$$ then the sequence $f_{n}=\alpha_{0}+\sum_{i=1}^{n}\alpha_{i}\nu_{i}$ is cauchy in $H_{W,V}(\mathbb{T})$. This implies that it converges to $f$ in $L^{2}_{V}(\mathbb{T})$. That is, $f_{n}$ is admissible for $f$, because $f_{n}\in \mathfrak{D}_{W,V}$, and by definition
		$$D^{-}_{W}f=\lim_{n\to\infty}D^{-}_{W}f_{n}=\sum_{i=1}^{\infty}\alpha_{i}D^{-}_{W}\nu_{i}.$$
		
	\end{proof}
	\begin{corollary}\label{coro1}
		We have the following results regarding approximation of functions in the $H_{W,V}(\mathbb{T})$ by smooth functions:
		$$\overline{ {C^{\infty}_{W,V}}(\mathbb{T})}^{\|.\|_{1,2}}=H_{W,V}(\mathbb{T}).$$ 
		Moreover, to compute the $W$-left weak derivative of $h\in H^{1}_{W,V}(\mathbb{T})$, we only need to verify that
		$$\int_{\mathbb{T}}\left(h-\int_{\mathbb{T}}hdV\right)D^{+}_{V}g\;dV=\int_{\mathbb{T}}g\partial_{W}^{-}h\;dW$$
		for all $g\in C^{\infty}_{V,W,0}(\mathbb{T}).$
		
	\end{corollary}
	
	\paragraph{}We end this section with a characterization of the dual of $H_{W,V}(\mathbb{T})$:
	
	\begin{proposition}\label{dual_sobolev}
		Let $H^{-1}_{W,V}(\mathbb{T})$ be the dual of $H_{W,V}(\mathbb{T})$. We have that $f\in H^{-1}_{W,V}(\mathbb{T})$ if, and only if, there exist $f_0\in L^2_V(\mathbb{T})$ and $f_1\in L^2_W(\mathbb{T})$ such that for every $g\in H_{W,V}(\mathbb{T})$
		$$(f, g) = \int_{\mathbb{T}} f_0(\xi) g(\xi) dV(\xi) + \int_{\mathbb{T}} f_1(\xi) D_W^-g(\xi)dW(\xi).$$
	\end{proposition}
	\begin{proof}
	Since $H_{W,V}(\mathbb{T})$ is a Hilbert Space, we can use Riesz's representation theorem on $f\in H^{-1}_{W,V}(\mathbb{T})$. This means that there is $f_{0}\in H_{W,V}(\mathbb{T})$ such that for every $g\in H_{W,V}(\mathbb{T})$
		$$(f, g) = \int_{\mathbb{T}} f_0(\xi) g(\xi) dV(\xi) + \int_{\mathbb{T}} D^{-}_{W}f_{0}(\xi) D_W^-g(\xi)dW(\xi).$$
		Clearly $D^{-}_{W}f_{0}\in L_{W}^{2}(\mathbb{T}).$ The converse follows immediately from H\"older's inequality.
	\end{proof}
	
	\section{One-sided second-order elliptic operators}
	
Before introducing the one-sided second order elliptic differential equation let us first discuss some of the preceding models. As discussed in \cite{pouso} if $V$ is a càdlàg nondecreascing function, the solutions (in the sense of Carathéodory) of 
	\begin{equation}\label{apl:1}
		x'_{V}(t)=f(t,x(t))
	\end{equation}
	can be viewed as solutions of differential equations with impulses or a dynamic equation on time scaling depending on the function $V$. The author developed the necessary machinery to study (\ref{apl:1}) in the context of Carathéodory, by introducing the concept of $V$-absolutely continuous functions.
On the other side, in \cite{franco} they studied the solutions of 
	\begin{equation}\label{apl:2}
		\partial_{t}\rho=D_{x}D_{W}\Phi(\rho),
	\end{equation}
	which was naturally obtained as a hydrodynamic limit of interacting particle systems in inhomogeneous
	media, which was interpreted as existence of "membranes". More precisely, 
	they can interpret the model as hydrodynamic limit of diffusions with permeable membranes, in which the discontinuity points of $W$ tend to reflect particles, creating space discontinuities in the solutions. Therefore, based on \cite{pouso, franco} and \cite{wsimas} we purpose the study of the one-sided second order elliptic differential operator 
	\begin{equation}
		L_{W,V}(u):=-D^{+}_{V}AD_{W}^{-}u+\kappa^2u
	\end{equation}
	where $A,\kappa :\mathbb{T}\to\mathbb{R}$ are suitable functions. Indeed, motivated by (\ref{apl:1}) and (\ref{apl:2}) we understand that the existence of jumps on the functions that induce the differential operators, in our case the functions $V$ and $W$ that appear on $L_{W,V}$, is eventually related to models considering impulses, reflections or more precisely changes of the momentum associated to a physical system which are generally induced by the environment.
	
In what follows let $A:\mathbb{T} \to \mathbb{R}$ be a positive bounded function, that is, there exists $K\geq 0$ such that for every $x\in\mathbb{T}$, $|A(x)|\leq K$ and bounded away from zero, that is, there exists $A_0>0$ such that for every $x\in \mathbb{T}$, $A(x)\geq A_0$. Let, $\kappa: \mathbb{T}\to \mathbb{R}$ be a bounded function. Sometimes we will suppose that $\kappa$ is bounded away from zero, meaning that there exists some constant $\kappa_0$ such that for all $x$, $\kappa(x)>\kappa_0$. Finally consider $B_{W,V}:H_{W,V}(\mathbb{T})\times H_{W,V}(\mathbb{T}) \to \mathbb{R}$, a bilinear and symmetric map, given by
	\begin{equation}\label{bl:1}
		B_{W,V}[u.v]=\int_{\mathbb{T}}AD^{-}_{W}uD^{-}_{W}vdW+\int_{\mathbb{T}}\kappa^{2}uvdV.
	\end{equation}
	
The rest of this section is based on classical results (see, for instance \cite{evans}).
	
	\begin{definition} Let $f\in L^{2}_{V}(\mathbb{T})$. We say that $u\in H_{W,V}(\mathbb{T})$ is a weak solution of the problem
		\begin{equation}\label{elp:1}
			L_{W,V}u=f\,\, \mbox{in}\,\, \mathbb{T}.
		\end{equation}
		if
		$$B[u.v]=(f,v)$$
		for all $v\in H_{W,V}(\mathbb{T}).$
	\end{definition}

As is standard in nonfractional elliptic partial differential equations, we will apply Lax-Milgram's theorem as a tool to find weak solution(s) of the problem (\ref{elp:1}). Therefore, we first need energy estimates.
	\begin{proposition}[Energy Estimates]\label{energyest} If $B_{W,V}$ is defined as above, there are $\alpha>0$ and $\beta>0$ such that for all $u,v\in H_{W,V}(\mathbb{T})$
		$$
		\left|B_{W,V}[u,v]\right|\le \alpha \|u\|_{W,V}\|v\|_{W,V}$$
		and for $u\in H_{W,V,0}(\mathbb{T}):= L^{2}_{V,0}(\mathbb{T})\cap H_{W,V}(\mathbb{T})$, we have
		\begin{equation}\label{eng:2}
			B_{W,V}[u,u]\geq \beta\|u\|^{2}_{W,V}.
		\end{equation}

	\end{proposition}
	\begin{proof}
		By (\ref{bl:1}), the assumptions on $A$ and $\kappa$, the triangle inequality and H\"older's inequality, we have that
		$$|B_{W,V}[u,v]|\le L_{0}\|D^{-}_{W}u\|_{W}\|D^{-}_{W}v\|_{W}+L_{1}^2\|u\|_{V}\|v\|_{V},$$
		where $L_{0}:=\sup_{\mathbb{T}}|A|$ and $L_{1}:=\sup_{\mathbb{T}}|\kappa|$. Now, by using $\|D^{-}_{W}u\|_{W}\le \|u\|_{W,V}$ and $\|u\|_{V}\le \|u\|_{W,V}$, we have
		$$B_{W,V}[u,v]\le (L_{0}+L_{1}^2)\|u\|_{W,V}\|v\|_{W,V}.$$
		setting $\alpha=L_{0}+L_{1}^2$, the first part of the statement is proved. For the second assertion note that by (\ref{i1})
		\begin{equation}
			\dfrac{A_{0}}{W(1)V(1)}\|u\|_{V}^{2}\le A_{0}\|D^{-}_{W}u\|^{2}_{W}\le B_{W,V}[u,u].
		\end{equation}
		Therefore
		\begin{equation}
			\dfrac{A_{0}}{2}\min\left\{\dfrac{1}{W(1)V(1)},1\right\}\|u\|_{W,V}^2\le B_{W,V}[u,u]
		\end{equation}
		
	\end{proof}
	
	\begin{remark}\label{eng:1}
		Observe that if we consider the new operator $$L_{W,V,\lambda}u:=L_{W,V}u+\lambda u,$$
		then the energy estimates still hold for  
		$$B_{W,V,\lambda}[u,v]:=B_{W,V}[u,v]+\lambda(u,v)_{V}.$$
		In this case, the second assertion is true for all $u\in H_{W,V}(\mathbb{T})$ and $\lambda>0$. Moreover $\alpha$ and $\beta$ depends on $\lambda.$ Finally, if in addiction $\kappa$ is bounded away from zero we get (\ref{eng:2}) for $\lambda>-\kappa_{0}$.
	\end{remark}
	
	\begin{proposition}[First existence result of weak solutions]\label{prop:40}
		Given $f\in L^2_{V}(\mathbb{T}),$ there exists a unique $u\in H_{W,V,0}(\mathbb{T})$ that is a weak solution of 
		\begin{equation}\label{slc:1}
			L_{W,V}u=f\,\,\mbox{in}\,\,\mathbb{T}.
		\end{equation}
		In addiction by Remark \ref{eng:1}, for each $\lambda>0$, there is a unique $u\in H_{W,V}(\mathbb{T})$ such that
		\begin{equation}\label{atv:1}
			L_{W,V}u+\lambda u = f,\,\,\mbox{in}\,\,\mathbb{T}.
		\end{equation}
		If, additionally, $\kappa$ is bounded away from zero, then there exists $u\in H_{W,V}(\mathbb{T})$ that is a weak solution of (\ref{slc:1}) and (\ref{atv:1}) can be weakly solved for $\lambda>-\kappa_{0}.$ Furthermore, the solutions of (\ref{slc:1}) and (\ref{atv:1}) enjoy the following bounds 
		\begin{equation}\label{ine:10}
			\|u\|_{W,V}\le C\|f\|_{V}    
		\end{equation}
		for some constant $C>0$ independent of $f$ and
		$$\|u\|_{V}\le \lambda^{-1}\|f\|_{V}$$
		for $\lambda>0.$ For $\kappa$ bounded away from zero we have that 
		\begin{equation}\label{atv:2}
			\|u\|_{V}\le (\kappa_{0}+\lambda)^{-1}\|f\|_{V}.
		\end{equation}
		for $\lambda>-\kappa_{0}.$
	\end{proposition}
	\begin{proof}
		By the energy estimates (Proposition \ref{energyest}) and Remark \ref{eng:1}, the  existence and uniqueness follow from Lax-Milgram's theorem. For the bounds note that
		$$\beta\|u\|_{W,V}^2\le B_{W,V}[u,u] = (f,u)_{V} \le \|f\|_{V}\|u\|_{V} \le \|f\|_{V}\|u\|_{W,V}$$
		that is, $\|u\|_{W,V}\le C\|f\|_{V},$ for $C=\beta^{-1}.$
		Analogously,
		$$\lambda\|u\|^2_{V}\le A_{0}\|D^{-}_{W}u\|^2_{V}+\lambda\|u\|^2_{V}\le B_{W,V,\lambda}[u,u] = (f,u) \le \|f\|_{V}\|u\|_{V}.$$
		Finally, (\ref{atv:2}) can be obtained similarly.
	\end{proof}
	
	\begin{proposition}[Second result for weak solutions]
		Precisely one of the following assertions are true: 
		\begin{enumerate}
			\item For $f\in L^2_{V}(\mathbb{T})$, there exists a unique solution of 
			\begin{equation}\label{fa:1}
				L_{W,V}u=f,
			\end{equation} 
			or else
			\item There is a weak solution of $u\not\equiv 0$ of the homogeneous problem
			$$L_{W,V}u=0.$$
		\end{enumerate}
		Moreover, if 2. is true, we have $\dim \ker L_{W,V}<\infty$ and (\ref{fa:1}) has a weak solution if and only if 
		$$\int_{\mathbb{T}}vfdV=0$$
		for every $v\in\ker L_{W,V}.$
	\end{proposition}
	\begin{proof}
		Fix $\lambda>0$. From Proposition \ref{prop:40}, given $g\in L^{2}_{V}(\mathbb{T})$ there exists a unique $u\in H_{W,V}(\mathbb{T})$ such that 
		$$B_{W,V,\lambda}[u,v]=(g,v)_{V}.$$
		Therefore, by existence and uniqueness, we can invert the operator on $g$:
		$$L_{W,V,\lambda}^{-1}g:=u.$$ Now note that $u\in H_{W,V}(\mathbb{T})$ satisfy $(\ref{fa:1})$ if, and only if, 		for all $v\in H_{W,V}(\mathbb{T})$, we have
		$$B_{W,V,\lambda}[u,v]=(\lambda u+f,v)_{V}$$
 or equivalently, if, and only if,
		$$L_{W,V,\lambda}^{-1}(\lambda u+f)=u$$
		or better if and only if
		\begin{equation}\label{comp:1}
			u-Ku=h    ,
		\end{equation}
		where $h:=L_{W,V,\lambda}^{-1}f$ and $Ku:=\lambda L^{-1}_{W,V,\lambda}u$. We can now use bounds obtained on Proposition \ref{prop:40} and Theorem \ref{compactemb} to conclude that the operator $K:L_{V}^{2}(\mathbb{T})\to H_{W,V}(\mathbb{T})\subset L^{2}_{V}(\mathbb{T})$ is compact and self-adjoint, this last is due to the symmetry of $B_{W,V}$. The result now follows from Fredholm's alternative.
	\end{proof}
	
The previous result shows that $\ker L_{W,V}$ plays a key role in the study of weak solutions of $(\ref{fa:1})$. Indeed, more explicitly
	$$\ker L_{W,V}=\left\{w\in H_{W,V}(\mathbb{T}); \int_{\mathbb{T}}AD^{-}_{W}wD^{-}_{W}vdW+\int_{\mathbb{T}}\kappa^2 wvdV=0, \forall v\in H_{W,V}(\mathbb{T})\right\}.$$
	In particular, for $v=w\in \ker L_{W,V}$ we have that
	$$\dfrac{A_{0}}{W(1)V(1)}\left\|w-\dfrac{1}{V(1)}\int_{\mathbb{T}}wdV\right\|^2_{V}\le A_{0}\|D^{-}_{W}w\|^2\le B_{W,V}[w,w]\Longrightarrow w\equiv \dfrac{1}{V(1)}\int_{\mathbb{T}}wdV.$$ 
	Thus, if there exists $w\in \ker L_{W,V}$ such that $\int_{\mathbb{T}}wdV\big/V(1)\not\equiv 0$, then $\kappa\equiv 0.$ In this case, there exists $u\in H_{W,V}(\mathbb{T})$ that is a weak solution of $$D^{+}_{V}AD^{-}_{W}u=f$$
	if and only if 
	$$\int_{\mathbb{T}}fdV=0.$$
	Furthermore, the solution is unique in $H_{W,V}(\mathbb{T})$ up to a constant. On the other hand, if $\ker L_{W,V}=0$, then
	$L_{W,V}$ define a bijection between $H_{W,V}(\mathbb{T})$ and $L^{2}_{V}(\mathbb{T})$. This happens, for instance, when $\kappa$ is bounded away from zero. 
	
For the next statement fix $\kappa$ bounded away from zero
	
	\begin{proposition}The following assertions are true
		\begin{enumerate}
			\item The eigenvalues of $L_{W,V}$ are real, countable, and we enumerate them, according their multiplicity, in such a way that
			\begin{equation}\label{eig:1}
				0<\lambda_{1}\le \lambda_2 \le \ldots \to \infty.
			\end{equation}
			\item There exists an orthonormal basis $\{\phi_{k}\}_{i\in\mathbb{N}}$ of $L^{2}_{V}(\mathbb{T})$ where $\phi_{k}\in H_{W,V}(\mathbb{T})$ is a eigenvector associated to $\lambda_{k}$, i.e.
			$$L_{W,V}\phi_{k}=\lambda_{k}\phi_{k}\,\,in\,\,\mathbb{T}.$$
		\end{enumerate}
		
	\end{proposition}
	\begin{proof} 
		
		We will prove 1. and 2. simultaneously. First of all, the linear operator $L_{W,V}:H_{W,V}(\mathbb{T})\to L^{2}_{V}(\mathbb{T})$ defined by
		$$L_{W,V}u=f \Leftrightarrow B_{W,V}[u,v]=(f,v), \forall v\in H_{W,V}(\mathbb{T})$$
		is well defined and is a bijection. Indeed, $L^{-1}_{W,V}: L^{2}_{V}(\mathbb{T})\to L^{2}_{V}(\mathbb{T})$ is linear operator, which is compact and symmetric. Indeed, by $(\ref{ine:10})$ if $L_{W,V}u=f$, there exists $C>0$ such that
		$$\|u\|_{V}\le\|u\|_{W,V}\le C \|f\|_{V} \Rightarrow \|L^{-1}_{W,V}f\|_{V}\le \|f\|_{V}.$$
		As we know, the immersion of $H_{W,V}(\mathbb{T})$ in $L^{2}(\mathbb{T})$ is compact, therefore $L^{-1}_{W,V}$ is compact. The symmetry easily follows from the symmetry of $B_{W,V}$. Clearly $0$ is not an eigenvalue of $L_{W,V}$ and, if $\mu\neq 0$ and $u\neq 0$ we have $$L_{W,V}u=\mu u \Leftrightarrow L^{-1}_{W,V}u=\frac{1}{\mu}u.$$ Therefore, since the eigenvectors of a symmetric operator belongs to $\mathbb{R}$, the eigenvectors of $L_{W,V}$ belongs to $\mathbb{R}.$ If $u$ and $\mu$ are taken as above, then the inequality
		$$\kappa_{0}\|u\|^{2}\le B_{W,V}[u,u]=\mu\|u\|^{2}$$
		implies that $0<\kappa_{0}\le\mu$. Hence, we can apply the spectral theorem for self-adjoint and compact operators on a separable Hilbert space to the operator $L_{W,V}^{-1}$. Therefore, there exists a complete orthonormal set $\{\phi_{k}\}_{k\in\mathbb{N}}$ of $L^{2}_{V}(\mathbb{T})$ constituted by eigenvectors of $L^{-1}_{W,V}$ where, each $\phi_{k}$ is associated to $\frac{1}{\lambda_{k}}$ with \begin{equation*}
			\frac{1}{\lambda_{1}}\geq \frac{1}{\lambda_{2}}\geq \ldots \geq \frac{1}{\lambda_{k}} \to 0.
		\end{equation*}
		this last fact is equivalent to $(\ref{eig:1}).$
		
	\end{proof}

\section{$W$-Brownian motion and its associated Cameron-Martin space}

In this section we will consider a generalization of the Brownian motion, which we will denote by $B_W(t)$, such that has the following \emph{simultaneous} interesting features:

\begin{enumerate}
	\item The finite-dimensional distributions of $B_W$ are Gaussian;
	\item The sample paths of $B_W$ are c\`adl\`ag and may have jumps.
\end{enumerate}

We have a well-known family of processes that generalize the standard Brownian motion and satisfies condition 1, indeed, the fractional Brownian motion is such an example. However, the sample paths of the fractional Brownian motion are continuous, in fact, $\gamma$-H\"older continuous, for $\gamma<H$, where $H$ is the Hurst parameter of the fractional Brownian motion. We also have a well-known family of stochastic processes that
generalize the Brownian motion and that satisfies 2, namely the L\'evy processes. However, by Lévy–Khintchine characterization, its finite-dimensional distributions are Gaussian if, and only if, it is the Brownian motion.

Therefore, it is unusual to have a process that generalizes the Brownian motion having Gaussian finite-dimensional distributions and that has jumps. 

We will, then, obtain the Cameron-Martin space of $B_W$ when we see its sample paths living on the space $L^2_V(\mathbb{T}))$. In this case, the Cameron-Martin space is given by $H_{W,V,\mathcal{D}}(\mathbb{T}) = \{f\in H_{W,V}(\mathbb{T}): f(0) = 0\}$, where $0\in\mathbb{T}$ is our "tagged" zero in $\mathbb{T}$. The subscript $\mathcal{D}$ refers to the Dirichlet boundary condition imposed by this space. 

Let us then introduce the generalized Brownian motion $B_W(t)$ on $\mathbb{T}$ and the pathwise white noise on $H_{W,V}(\mathbb{T})$ as a functional on the dual $H^{-1}_{W,V}(\mathbb{T})$.
To such an end, we will "tag" a point at $\mathbb{T}$ and identify it as $zero$, then do the standard identification between the torus and the interval $[0,1)$. Thus, we will consider the process on the interval $[0,1)$, and the process
will, almost surely, have a jump at $0$, which is identified with $1$ (that is, the limit $\lim_{x\to 0-} B_W(t)$, almost surely will not be zero).

\begin{definition}\label{W-brownian}
	We say that $B_W(t)$ is a $W$-Brownian motion if it satisfies the following conditions:
	\begin{enumerate}
		\item $B_W(0) = 0$ almost surely;
		\item If $t>s$, then $B_W(t)-B_W(s)$ is independent of $\sigma(B_W(u); u\leq s)$;
		\item If $t>s$, then $B_W(t)-B_W(s)$ has $N(0,W(t)-W(s))$ distribution;
		\item $B_W(s)$ has c\`adl\`ag sample paths.
	\end{enumerate}
\end{definition}

\begin{definition}\label{W-brownian-in-law}
	We say that $B_W$ is a $W$-Brownian motion \emph{in law} if conditions 1-3 of Definition \ref{W-brownian} hold.
\end{definition}

\begin{figure}[h]
	\label{W-brownian-fig}
	\begin{center}
	\includegraphics[scale=0.5]{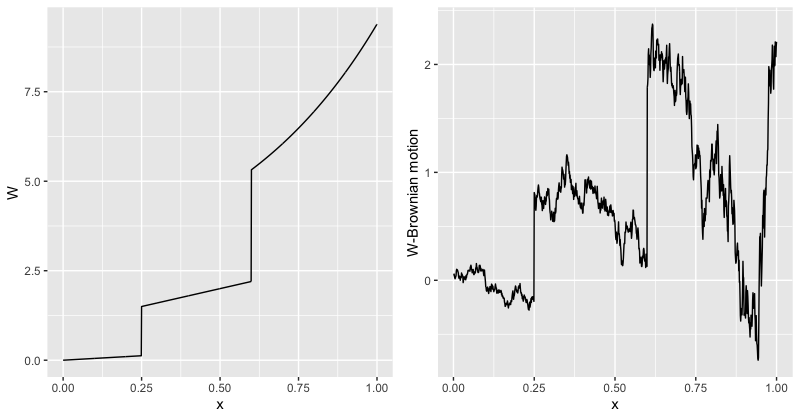}
	\end{center}
	\caption{An example of a $W$-Brownian motion for a function $W$.}
\end{figure}

We see, on Figure \ref{W-brownian-fig}, an example of a realization of a $W$-Brownian motion for a function $W$. The plot of $W$ is on the left-hand side and the realization of the $W$-Brownian motion is on the right-hand side. We can see that the $W$-Brownian motions has two jumps, one at each discontinuity of $W$.

\begin{proposition}
	The $W$-Brownian motion in law exists.
\end{proposition}
\begin{proof}
	It is an immediate consequence of Kolmogorov's extension theorem.
\end{proof}

\begin{remark}
	Notice that if $t^\ast$ is a discontinuity point of $W$, then there exists some $\varepsilon>0$ such that for every $h>0$,
	$$P(|B_W(t^\ast-h) - B_W(t^\ast)|>0) \geq \varepsilon.$$
	Indeed, it is enough to take any $\varepsilon>0$ such that $P(|N(0,\Delta W(t^\ast)|>0) > \varepsilon$, where $\Delta W(t^\ast) = W(t^\ast) - W(t^\ast-)$.
\end{remark}

\begin{remark}
	The previous remark then implies that $B_W$ is not continuous in probability so, in particular, $B_W$ is not an additive process (see, for instance, \cite{sato} for a definition of additive processes).
\end{remark}

\begin{proposition}
	If $B_W(\cdot)$ is a $W$-Brownian motion in law, then it is a martingale.
\end{proposition}
\begin{proof}
	By Condition 3 of Definition \ref{W-brownian}, $E(B_W(t) - B_W(s)) = 0$, and by Condition 2 of \ref{W-brownian}, we have that, for $t>s$,
	\begin{eqnarray*}
		E(B_W(t)|\sigma(B_W(u), u\leq s)) &=&  E(B_W(t)- B_W(s)|\sigma(B_W(u), u\leq s)) + B_W(s)\\
		&=& E(B_W(t)- B_W(s)) + B_W(s)\\
		&=& B_W(s).
	\end{eqnarray*}
\end{proof}

\begin{proposition}\label{existencebrownian}
	The $W$-Brownian motion exists.
\end{proposition}
\begin{proof}
	Let $B_W(\cdot)$ be a $W$-Brownian motion in law. We need to show that there exists a c\`adl\`ag modification of $B_W(\cdot)$. To this end, begin by noticing that since $B_W(\cdot)$ is a martingale, it is a quasimartingale in the sense of \cite{brooks}. Indeed, see, for instance, claim A.3.2 in \cite{brooks}. 
	
	Notice that since the $W$-Brownian motion is a $\mathbb{R}$-valued quasimartingale, the assumptions of Theorem 2.1 in \cite{brooks} are satisfied. Thus, in order to apply Theorem 2.1 of \cite{brooks} to conclude that $B_W(\cdot)$ has a c\`adl\`ag modification, it is enough to show that $B_W(\cdot)$ is right-continuous in probability.
	
	To such an end, notice that, for $t>s$,
	\begin{eqnarray*}
		P(|B_W(t)-B_W(s)| > \varepsilon) &\leq& \frac{E\left(\left|B_W(t)-B_W(s)\right|^2\right)}{\varepsilon^2}\\
		&=& \frac{W(t)-W(s)}{\varepsilon^2}.
	\end{eqnarray*}
	
	Therefore, since $W$ is right-continuous, $B_W(\cdot)$ is right-continuous in probability. The result now follows from applying Theorem 2.1 from \cite{brooks}.
\end{proof}

If $W$ has finitely many discontinuity points and $W$ is H\"older continuous on each continuous subinterval of the form $[d_i, d_{i+1})$, where $\{d_i; i=1,\ldots, N\}$ is the set of discontinuity points, where $N\in\mathbb{N}$, then the set of discontinuity points of the sample paths will be contained in the set of discontinuity points of $W$:
\begin{proposition}\label{kolmchen}
	If the set of discontinuity points of $W$ is finite, say $0\leq d_1 < d_2 <\ldots < d_N$, for some $n\in\mathbb{N}$, and if we let $I_1 = [0,d_1), I_2 = [d_1,d_2), \ldots, I_{N} = [d_{N-1},d_N), I_{N+1} = [d_N,1)$, and for each interval $I_k$, there exists some $\gamma_k\in (0,1]$, such that the restriction of $W$ to $I_k$ is $\gamma_k$-H\"older continuous, then $B_W$ has a modification whose sample paths are continuous on each interval $I_k, k=1,\ldots, N+1$.
\end{proposition}
\begin{proof}
	In this case, we can directly apply Kolmogorov-Chentsov's continuity theorem to the restriction of $B_W$ to each interval $I_k$ to obtain that the restriction of $B_W$ to each $I_k$ has a modification, say $B_{W,k}$, which has continuous sample paths, for $k=1,\ldots, N+1$. Notice that from the very construction, the modifications belong to the same probability space. Therefore, we can simply define the modification $\widetilde{B}_W$ on $\mathbb{T}$ as $\widetilde{B}_W(s) = B_{W,k}(s)$ is $s\in I_k$, $k=1,\ldots, N+1$, and since the intervals $I_k$ are disjoint, there are no overlaps.
\end{proof}

\begin{remark}
	It is important to notice that if we only assume $W$ to have finitely many discontinuity points, and impose no assumptions regarding H\"older continuity, we would only have that the restriction of $B_W$ to each continuity subinterval is continuous in probability, which implies that the restriction of $B_W$ to each subinterval is an additive process in law, and thus admits a version which is right-continuous, but is not strong enough to ensure the existence of a modification with continuous sample paths. Thus, we would not be able to conclude that $B_W$ would be continuous in each interval such that $W$ is continuous. 
\end{remark}

Notice that, almost surely, the sample paths of $B_W(\cdot)$ are bounded on $[0,1)$, since it is c\`adl\`ag on $[0,1)$. Therefore,
it is an $L_V^2(\mathbb{T})$-process (and also an $L^2_W(\mathbb{T})$-process). Furthermore, $B_W(\cdot)$ has orthogonal increments and is $L^2$-right-continuous, that is $$\lim_{t\to s+} \mathbb{E}((B_W(t)-B_W(s))^2) =0.$$ Therefore, the definition of the stochastic integral of any function $f\in L^2_V(\mathbb{T})$ with respect to $B_W(\cdot)$ follows from the standard theory on $L^2$-processes (see \cite{ash}). Namely, the stochastic integral is linear, 
\begin{equation}
	\int_{[0,1)} \textbf{1}_{(a,b]} dB_W = B_W(b) - B_W(a),
	\label{intsimplefunctions}
\end{equation}
and, for any $f\in L^2_W(\mathbb{T})$, we have the following isometry:
\begin{equation}
	\mathbb{E}\left[ \left( \int_{[0,1)} f dB_W\right)^2\right] = \int_{[0,1)} f^2 dV.
	\label{isometry}
\end{equation}
Now, observe that, by \eqref{intsimplefunctions}, the integral of each simple function follows a normal distribution. By \eqref{isometry}, the stochastic integral with respect to $B_W$ is an $L^2$-limit of Gaussian random variables, and thus, it is a Gaussian random variable (see \cite[Theorem 1.4.2 ]{ash}). 

Since, for each simple function $\varphi$, we have $\mathbb{E}\left( \int_{[0,1)} \varphi dB_W \right) =0$, it follows from \eqref{isometry} that the stochastic integral with respect to any deterministic function $f\in L^2_V(\mathbb{T})$ is zero. Furthermore, since the stochastic integral with respect to deterministic functions is a mean zero Gaussian random variable, it follows from \eqref{isometry}, again, that for any $f\in L^2_V(\mathbb{T})$,
\begin{equation}
	\int_{[0,1)} f dB_W \sim N\left(0, \int_{[0,1)} f^2 dV\right).
	\label{diststochint}
\end{equation}

We can now show that the integration by parts formula can be applied for any function in $H_{W,V}(\mathbb{T})$ and with the derivative, being the weak derivative $D_W^-$:

\begin{proposition}\label{pathintegral}
	Let $B_W(\cdot)$ be a $W$-Brownian motion. For any $g\in H_{W,V}(\mathbb{T})$ we have the following integration by parts formula:
	\begin{equation}\label{stochintpath}
		\int_{(0,t]} g dB_W = B_W(t)g(t) - \int_{(0,t]} B_W(s-) D_W^-{g}(s) dW(s).
	\end{equation}
\end{proposition}
\begin{proof}
	Let $g\in H_{W,V}(\mathbb{T})$. From Theorem \ref{sobchar} and stochastic Fubini's theorem (see \cite[Theorem 64, p.210]{protter}), we have
	\begin{align*}
		\int_{(0,t]} g dB_W &= \int_{(0,t]} \left( g(0) + \int_{(0,s]} D_W^-g(u)dW(u)\right) dB_W(s)\\
		&= g(0) B_W(t) + \int_{(0,t]} D_W^-g(u)\left( \int_{[u,t]} dB_W(s)\right) dW(u)\\
		&= g(0) B_W(t) + \int_{(0,t]} D_W^-g(u) (B_W(t) - B_W(u-)) dW(u)\\
		&= g(0) B_W(t) + B_W(t) (g(t) - g(0)) - \int_{(0,t]} D_W^-g(u) B_W(u-)dW(u)\\
		&= B_W(t)g(t) - \int_{(0,t]} D_W^-g(u) B_W(u-)dW(u).
	\end{align*}
\end{proof}

Notice that since $B_W(\cdot)$ has c\`adl\`ag sample paths, $B_W(\cdot)$ is almost surely bounded on $\mathbb{T}$, so the right-hand side of \eqref{stochintpath} is well-defined.

We will now obtain the Cameron-Martin space of the $W$-Brownian motion. The Cameron-Martin
spaces appear naturally on Malliavin calculus as the set of "directions" in which one can differentiate (with respect to the Malliavin derivative). Recently Bolin and Kirchner \cite{bolinkirchner} showed that the Cameron-Martin
spaces play a key role in find the best linear predictor with the "wrong" covariance structure. More generally, the Cameron-Martin space uniquely determines the Gaussian measure (see \cite[Theorem 2.9]{daprato}) on Banach spaces. Furthermore, the Cameron-Martin space is, in a sense, independent of the Banach space in which the Gaussian measure is defined (see \cite[Proposition 2.10]{daprato}). This means that one can specify a Gaussian measure by specifing the Cameron-Martin space. This means that, by computing the Cameron-Martin space of $B_W$, we are showing that the distribution of $W$-Brownian motion in $L^2_V(\mathbb{T})$ is the unique Gaussian measure  associated to the Sobolev space $H_{W,V,\mathcal{D}}(\mathbb{T})$.

We begin by noting that from Example 2.3.16 in Bogachev \cite{bogachev}, since the $W$-Brownian motion has c\`adl\`ag sample paths, and
$$\int_{[0,1)} W(t) dV(t) \leq W(1)V(1) <\infty,$$
we have that the distribution of $B_W$ on $L^2_V(\mathbb{T})$, i.e., the measure given by
$$\gamma_W(B) = \mathbb{P}\left( \omega: B_W(\cdot,\omega) \in B\right),$$
where $B$ is a borel set in $L^2_V(\mathbb{T})$ (which is a separable Hilbert space), defines
a Gaussian measure in $L^2_V(\mathbb{T})$. 

Let $K:L^2_V(\mathbb{T})\to L^2_V(\mathbb{T})$ be the covariance operator of $B_W(\cdot)$, i.e., of
the Gaussian measure induced by $B_W(\cdot)$ in $L^2_V(\mathbb{T})$.  Since $L^2_V(\mathbb{T})$ is a Hilbert space,
we can apply Theorem 2.3.1 in Bogachev \cite{bogachev} to obtain that the Fourier transform of $\gamma_W$ satisfies, for every $g\in L^2_V(\mathbb{T})$,
\begin{equation}
\widehat{\gamma}_W(g) = \int_{L^2_V(\mathbb{T})} e^{i\langle g, f\rangle_{L^2_V(\mathbb{T})} } d\gamma_W(f) = \exp\left\{ -\frac{1}{2} \langle Kg,g\rangle_{L^2_V(\mathbb{T})}\right\}.
\label{fouriertransf}
\end{equation}

We will then compute the Fourier transform \eqref{fouriertransf} explicitly to show that $K$ is a kernel operator and its kernel is, indeed, the covariance function of the $W$-Brownian motion. To such an end, by standard limiting arguments, it is enough to show the result for simple functions. We will provide the details for the case of indicator functions of intervals of the form $(a,b]$. The case of linear combination of indicators functions can be carried out similarly. Let us consider then, the function $\textbf{1}_{(a,b]}$. By using stochastic Fubini's theorem (see \cite[Theorem 64, p.210]{protter}), we have
\begin{align*}
	\widehat{\gamma}_W(\textbf{1}_{(a,b]}) &= \int_{L^2_V(\mathbb{T})} \exp\left\{i \langle \textbf{1}_{(a,b]}, f\rangle_{L^2_V(\mathbb{T})}\right\}= \mathbb{E}\left[ \exp\left\{ i\int_{(a,b]}B_W(s)dV(s)\right\} \right]\\
	&= \mathbb{E}\left[ \exp\left\{ i\int_{(a,b]}\int_{(0,s]}dB_W(u)dV(s)\right\} \right]\\
	&= \mathbb{E}\left[ \exp\left\{ i\int_{(0,a]}\int_{(a,b]}dV(s)dB_W(u) + i\int_{(a,b]} \int_{[u,b]} dV(s)dB_W(u)\right\} \right] \\
	&=  \mathbb{E}\left[ \exp\left\{ i(V(b+)-V(a+))B_W(a)+i\int_{(a,b]}(V(b+)-V(u))dB_W(u)\right\} \right] \\
	&= \exp\left\{-\frac{1}{2}\left[W(a)(V(b+)-V(a+))^2 + \int_{(a,b]}(V(b+)-V(u))^2dW \right]\right\},
\end{align*}
where in the last expression we used the fact that the increments are independent, the characteristic function of a normal distribution and the distribution of the stochastic integrals \eqref{diststochint}. In order to show that $K$ is a kernel operator with kernel $k(s,t) = W(t\land s)$, we need to show that if we take $\widehat{K}:L^2_V(\mathbb{T})\to L^2_V(\mathbb{T})$, defined by
\begin{equation}
	\widehat{K}f(t) = \int_{[0,1)} W(t\land s) f(s) dV(s).
	\label{covarianceoper}
\end{equation}
Then,
\begin{equation}
\langle \widehat{K} \textbf{1}_{(a,b]}, \textbf{1}_{(a,b]}\rangle_{L^2_V(\mathbb{T})} =  W(a)(V(b+)-V(a+))^2 + \int_{(a,b]}(V(b+)-V(u))^2dW.
\label{koperator}
\end{equation}

Let us then show that the above expression is true. We have,

\begin{align*}
	\int_{[0,1)} \int_{[0,1)} W(t\land s) \textbf{1}_{(a,b]}(t) \textbf{1}_{(a,b]}(s) dV(t) dV(s) &= \int_{(a,b]} \int_{(a,b]} W(t\land s) dV(t) dV(s)\\
	&= \int_{(a,b]}\int_{(a,s]} W(t)dV(t)dV(s)+\\
	&+\int_{(a,b]}\int_{(s,b]}W(s)dV(t) dV(s).
\end{align*}
The last expression is equal to
$$
\int_{(a,b]}\left( \int_{(a,s]} \int_{(0,t]}dW(u)dV(t) + \int_{(s,b]}\int_{(0,s]}dW(u)dV(t) \right) dV(s),
$$
which, by Fubini's theorem, can be written as
$$\int_{(0,b]}\left(\int_{(a,b]} \int_{(a,b]} \left(\textbf{1}_{(0,t]}(u) \textbf{1}_{(a,s]}(t) \textbf{1}_{(a,b]}(s) + \textbf{1}_{(0,b]}(u) \textbf{1}_{(s,b]}(t) \textbf{1}_{(a,b]}(s) \right)dV(t)dV(s) \right) dW(u).$$
Let us split the integral with respect to $dW(u)$ into two parts, one for $u\leq a$, and the other when $u>a$. The integral over of $dW(u)$ over $(0,a]$ is equal to
$$\int_{(0,a]} \left( (\textbf{1}_{(a,s]}(t) \textbf{1}_{(a,b]}(s)  + \textbf{1}_{(s,b]}(t) \textbf{1}_{(a,b]}(s) )  dV(t)dV(s) \right) dW(u),$$
and thus, equal to
$$ \int_{(0,a]} \int_{(a,b]}\int_{(a,b]} \textbf{1}_{(a,b]}(t) \textbf{1}_{(a,b]}(s) dV(t)dV(s)dW(u)= W(a)(V(b+)-V(a+))^2.$$
This is the first part of \eqref{koperator}. Let us now compute the integral with respect to $dW(u)$ over $(a,b]$:
$$\int_{(a,b]}\int_{(a,b]} \int_{(a,b]} \left(\textbf{1}_{(a,t]}(u) \textbf{1}_{(a,s]}(t) \textbf{1}_{(a,b]}(s) + \textbf{1}_{(a,b]}(u) \textbf{1}_{(s,b]}(t) \textbf{1}_{(a,b]}(s) \right)dV(t)dV(s)dW(u).$$
To such an end, observe that:
$$\textbf{1}_{(a,t]}(u) \textbf{1}_{(a,s]}(t) \textbf{1}_{(a,b]}(s) = \textbf{1}_{(a,b]}(u) \textbf{1}_{[u,b]}(s)\textbf{1}_{[u,s]}(t)$$
and
$$\textbf{1}_{(a,b]}(u) \textbf{1}_{(s,b]}(t) \textbf{1}_{(a,b]}(s) = \textbf{1}_{(a,b]}(u) \textbf{1}_{[u,b]}(s)\textbf{1}_{(s,b]}(t).$$
By summing both expressions above, we arrive at
$$\textbf{1}_{(a,t]}(u) \textbf{1}_{(a,s]}(t) \textbf{1}_{(a,b]}(s) + \textbf{1}_{(a,b]}(u) \textbf{1}_{(s,b]}(t) \textbf{1}_{(a,b]}(s) = \textbf{1}_{(a,b]}(u) \textbf{1}_{[u,b]}(s) \textbf{1}_{[u,b]}(t).$$
Therefore, the resulting integral is
$$\int_{(a,b]}\int_{(a,b]}\int_{(a,b]}  \textbf{1}_{(a,b]}(u) \textbf{1}_{[u,b]}(s) \textbf{1}_{[u,b]}(t) dV(t)dV(s)dW(u) = \int_{(a,b]} (V(b+)-V(u))^2dW(u).$$
This is the second expression in \eqref{koperator} and completes the proof for the case of indicator functions. As mentioned previously, the case of linear combinations of indicator functions follows similarly. So, the covariance operator of $B_W(\cdot)$ is, indeed, a kernel operator in $L^2_V(\mathbb{T})$ with kernel $k(t,s)=W(t\land s).$

We will now obtain the Cameron-Martin space associated to the $W$-Brownian motion. To such an end,
recall that the covariance operator of the $W$-Brownian motion was given in \eqref{covarianceoper}. Let us now obtain a nicer expression for $\widehat{K}f$ using Fubini's theorem:
\begin{align}\label{nicergreen}
	\nonumber \widehat{K}f(t) &= \int_{[0,1)} W(t\land s) f(s) dV(s) = \int_{[0,t)}W(s)f(s)dV(s) + \int_{[t,1)} W(t)f(s)dV(s)\\
	\nonumber&= \int_{[0,t)}\int_{(0,s]} dW(u) f(s)dV(s) + \int_{[t,1)}\int_{(0,t]} dW(u)f(s) dV(s)\\
	\nonumber&= \int_{(0,t]} \int_{[u,t)} f(s) dV(s) dW(u) + \int_{(0,t]} \int_{[t,1)} f(s)dV(s)dW(u)\\
	\nonumber&= \int_{(0,t]} \int_{[u,1)} f(s)dV(s)dW(u)\\ &=W(t)\int_{[0,1)}f(s)dV(s)-\int_{(0,t]} \int_{[0,u)} f(s)dV(s)dW(u).
\end{align}

Thus, $\widehat{K}(f)(t) =W(t)\int_{[0,1)}f(s)dV(s)-\int_{(0,t]} \int_{[0,u)} f(s)dV(s)dW(u)$. So, in particular, $\widehat{K}(f)(0) = 0$ and the $W$-left-derivative of $\widehat{K}(f)$ given by $$D_W^-(\widehat{K}f)(t) = \int_{[0,1)}f(s)dV(s)- \int_{[0,t)} f(s)dV(s)$$ satisfies $D_W^-(\widehat{K}f)(1)=0$. This motivates the definition of the subspace of $H_{W,V}(\mathbb{T})$ with a tagged zero, and Dirichlet boundary conditions at zero:
$$H_{W,V,\mathcal{D}}(\mathbb{T}) = \left\{f\in H_{W,V}(\mathbb{T})\cap L^{2}_{V,0}(\mathbb{T}): f(0) = 0\right\}.$$
Note that in this space, the following is an equivalent norm to the Sobolev norm (denote by $\langle\cdot,\cdot\rangle_{W,V,\mathcal{D}}$ the inner product associated to this norm):
$$\|f\|_{W,V,\mathcal{D}}^2 = \int_{\mathbb{T}} (D_W^- f)^2dW.$$
Indeed, on one hand we always have for any $f\in H_{W,V}(\mathbb{T})$, $\|f\|_{W,V,\mathcal{D}}\leq \|f\|_{W,V}$. On the other hand, from Theorem \ref{sobchar}, it follows from the elementary inequality $(a+b)^2\leq 2(a^2+b^2)$ and Jensen's inequality that there exists some constant $C>0$ such that
$$\|f\|^2_{L^2_V(\mathbb{T})} \leq C\left( (f(0))^2 + \|D_W^-f\|_{L^2_W(\mathbb{T})}^2\right).$$

Let us now show that the Cameron-Martin space of the $W$-Brownian motion on $L^2_V(\mathbb{T})$ is, indeed, $H_{W,V,\mathcal{D}}(\mathbb{T})$.  We know that
the Cameron-Martin space of the $W$-Brownian motion is given by $\mathcal{H}=\widehat{K}^{1/2}(L^2_V(\mathbb{T}))$, since $\widehat{K}$ is the covariance operator of the $W$-Brownian motion on $L^2_V(\mathbb{T})$. Equivalently, $\mathcal{H}$ is the completion of $\widehat{K}(L^{2}_{V}(\mathbb{T}))$ relative to the norm $\|\cdot\|_{\mathcal{H}}=\left\<\widehat{K}(\cdot),\cdot\right\>_{L^{2}_{W}(\mathbb{T})}$. Now, we will see that in the $\widehat{K}(L^2_{V}(\mathbb{T}))$ the norm $\|\cdot\|_{\widehat{K}}$ is well know. In fact, given $f,g \in \widehat{K}(L^2_{V}(\mathbb{T}))$ there exist functions $u,v\in L^2_V(\mathbb{T})$ such that $f = \widehat{K}u$ and $g = \widehat{K}v$. Hence, let $h(w) = \int_{[w,1)} u(t)dV(t)$, so, by Fubini's theorem and \ref{nicergreen},
\begin{align*}
	\langle f, g\rangle_{\mathcal{H}} &= \langle \widehat{K} u, v\rangle_{L^2_V(\mathbb{T})}\\
	&= \int_{[0,1)} \left( \int_{(0,s]} \int_{[w,1)} u(t)dV(t)dW(w)\right) v(s)dV(s)\\
	&= \int_{[0,1)} \int_{[0,s)} h(w) v(s)dW(w)dV(s) = \int_{[0,1)} h(w) \left(\int_{[w,1)}v(s)dV(s)\right) dW(w)\\
	&= \int_{[0,1)} \left( \int_{[w,1)}u(s)dV(s)\right) \left( \int_{[w,1)}v(s)dV(s)\right) dW(w)\\
	&= \langle D_W^-f, D_W^- g\rangle_{L^2_W(\mathbb{T})} = \langle f,g\rangle_{W,V,\mathcal{D}}.
\end{align*}

This shows that the Cameron-Martin norm is given by $\|\cdot\|_{W,V,\mathcal{D}}$. Finally, a straightforward adaptation of Theorem \ref{thm:1} shows that $\widehat{K}(L^2_V(\mathbb{T}))$, is dense in the space $H_{W,V,\mathcal{D}}(\mathbb{T})$. 
Therefore, since the Cameron-Martin space is the completion of $\widehat{K}(L^2_V(\mathbb{T}))\subset H_{W,V,\mathcal{D}}(\mathbb{T})$
with respect to the norm  $\|\cdot\|_{W,V,\mathcal{D}}$ we have that $\mathcal{H} = H_{W,V,\mathcal{D}}(\mathbb{T})$.

\begin{remark}
	It is interesting to observe that since we looked at the $W$-Brownian motion as an element of $L^2_V(\mathbb{T})$, we were able to show that its Cameron-Martin space is $H_{W,V}(\mathbb{T})$. It is clear that we can embed the $W$-Brownian motion into other $L_{\widetilde{V}}^2(\mathbb{T})$ spaces. Indeed, since $\mathbb{T}$ is compact and the sample paths of the $W$-Brownian motion are c\`adl\`ag, they are bounded, so belong to any $L_{\widetilde{V}}^2(\mathbb{T})$, where $\widetilde{V}:\mathbb{R}\to\mathbb{R}$ is increasing and periodic in the sense of \eqref{periodic}. This shows that the $V$ is not very important for the $W$-Brownian motion, what really matters is the $W$ and it is reflected at its Cameron-Martin space, $H_{W,V,\mathcal{D}}(\mathbb{T})$. In $H_{W,V,\mathcal{D}}(\mathbb{T})$, the weak derivative in $W$ is the true restriction, since by Theorem \ref{sobchar}, we have $H_{W,V,\mathcal{D}}(\mathbb{T})\subset L_{\widetilde{V}}^2(\mathbb{T})$. This is in consonance with the Cameron-Martin space uniquely determining the Gaussian measure. Indeed, the $V$ function does not have much impact on the $W$-Brownian motion, and also, does not have much impact on its Cameron-Martin space, in the sense that it is contained on any $L_{\widetilde{V}}^2(\mathbb{T})$.
\end{remark}

\section{Applications to stochastic partial differential equations}

We will now apply the theory to solve a class of stochastic partial differential
equations that generalize a non-fractional case of the well-known Matérn equations on a domain with periodic boundary conditions. 

For this section, we will work on $H_{W,V,\mathcal{D}}(\mathbb{T})$ with a tagged zero and, furthermore, recall that we assumed that $W$ is continuous at $x=0$. Therefore, from Theorem \ref{sobchar}, we have that for any function $g\in H_{W,V}(\mathbb{T})$,
$$\lim_{x\to 1-} g(x) = g(0) = \lim_{x\to 0+} g(x).$$
This means that if $g\in H_{W,V,\mathcal{D}}(\mathbb{T}),$
$$\lim_{x\to 1-} g(x) = 0 = \lim_{x\to 0+} g(x).$$
Therefore, it follows from the integration by parts formula, \eqref{stochintpath}, that under these assumptions, for any $g\in H_{W,V,\mathcal{D}}(\mathbb{T}),$ we
have
	$$\int_{\mathbb{T}} g(s) dB_W(s) = - \int_{\mathbb{T}} B_W(s-) D_W^-{g}(s) dW(s).$$
	So that
	\begin{eqnarray}
		\left|\int_{\mathbb{T}} g(s) dB_W(s)\right| &\leq& \int_{\mathbb{T}} |B_W(s) D_W^-{g}(s)| dW(s)\nonumber \\
		&\leq& \|B_W\|_{L^2_W(\mathbb{T})} \|D_W^{-}g\|_{L^2_W(\mathbb{T})}\nonumber\\
		&\leq& \left(\sup_{t\in\mathbb{T}}\|B_W(t)\|\right) W(1) \|g\|_{H_{W,V}(\mathbb{T})}.\label{functionalwhite}
	\end{eqnarray}

	So, the stochastic integral defines, almost surely, a bounded linear functional on $H_{W,V,\mathcal{D}}(\mathbb{T})$.
	We can now define the pathwise $W$-gaussian white noise on the Cameron-Martin space $H_{W,V,\mathcal{D}}(\mathbb{T})$.

\begin{definition}\label{pathwhitenoise}
	We define the pathwise $W$-gaussian white noise, $\dot{B}_W \in H^{-1}_{W,V,0}(\mathbb{T}):=(H_{W,V,\mathcal{D}}(\mathbb{T}))^\ast$, as the functional
	\begin{equation}\label{bponto}
		\dot{B}_W(g) = -\int_{\mathbb{T}} B_W(s) D_W^-g(s) dW(s).
	\end{equation}
\end{definition}

We will now study the existence and uniqueness of weak solutions of the following stochastic partial differential equation:
\begin{equation}\label{elliptic}
	\kappa^2 u - D_V^+ (H D_W^- u) = \dot{B}_W
\end{equation}
on the space $H_{W,V,\mathcal{D}}(\mathbb{T})$.

It is important to notice that equation \eqref{elliptic} can be seen as $W$-generalized counterpart of a special non-fractional case of the well-known Matérn stochastic partial differential equation.

Let us now, provide the definition of weak solution of the above equations:

\begin{definition}\label{weaksolmatern}
	We say that $u\in H_{W,V,\mathcal{D}}(\mathbb{T})$ is a weak solution of equation \eqref{elliptic} if for every $g\in H_{W,V,\mathcal{D}}(\mathbb{T})$, the following identity is true:
	$$\int_{\mathbb{T}} \kappa^2 u g dV + \int_{\mathbb{T}} HD_W^- u D_W^-g dW = \int_{\mathbb{T}} gdB_W,$$
	or, equivalently, in terms of functionals, if the following identity is true:
	$$B_{L_{W,V}}(u, g) = \dot{B}_W(g),$$
	where $B_{L_{W,V}}$ is the bilinear functional given by $$\int_{\mathbb{T}} \kappa^2 u g dV + \int_{\mathbb{T}} HD_W^- u D_W^-g dW$$
	for $u,g\in H_{W,V,\mathcal{D}}(\mathbb{T})$ and $\dot{B}_W$ is the pathwise $W$-Gaussian white noise defined in (\ref{bponto}).
\end{definition}

By comparing with the solutions of the standard Mat\'ern equation, we can readily see that equation \eqref{pathwhitenoise} can be suitable for modelling situations in which the process finds barriers, as well as, having eventually non-diffusive behavior throughout the domain (corresponding to situations in which the function $W$ is continuous, but not H\"older continuous, in some interval).

We, then, have the following theorem on existence and uniqueness of solutions of \eqref{elliptic}:

\begin{theorem}\label{exisuniqellipt}
	Let $\kappa$ and $H$ be bounded functions that are also bounded away from zero, where $H$ is positive. Then, equation \eqref{elliptic} has a unique solution with, almost surely, c\`adl\`ag sample paths.
\end{theorem}
\begin{proof}
	From \eqref{functionalwhite}, $\dot{B}_W$ is, almost surely, an element of $H^{-1}_{W,V,0}(\mathbb{T})$. Therefore, from Lax-Milgram theorem, \eqref{elliptic} has a unique solution. Finally, from Theorem \ref{sobchar}, $u$ has, almost surely, c\`adl\`ag sample paths.
\end{proof}
%

\chapter{Fractional one-sided measure theoretic Sobolev spaces, generalized Maclaurin expansion and applications to generalized stochastic partial differential equations}

The study of functional equations may lead to problems with solutions that are discontinuous. One natural question is regarding the regularity of such discontinuous functions in some sense. This motivated the definition of a differential operator, in which one differentiates a function with respect to another function that may have discontinuity points. Recently, these differential operators are drawing attention for some researchers. See, for instance,$\cite{frag, pouso,wsimas,trig,uta}$ and $\cite{franco}$. 
	
	Let us illustrate how the discontinuities of the function that induces the differential operator can afftect the solutions of differential equations driven by such operators. Consider, for instance, the equation
	$$\rho_{t}-D_{x}D_{W}\phi(\rho)=0,$$
	which was studied in \cite{franco}, where $D_x$ is the usual differential operator and $D_W$ is a differential operator with respect to some right-continuous and strictly increasing function. They showed that this equation can model diffusion on permeable membranes in which the discontinuity nature of the solution is associated to reflection of particles on the discontinuity points of $W$. On the other hand, if we turn our attention to impulsive differential equations, for instance, 
$$
		\begin{cases}
			D_{t}x(t)=f(x,t);\\
			\Delta x(t_{k})=I_{k}(t_{k}),
		\end{cases}
$$
	where $\Delta {x}(t):=x(t)-x(-t)$ and  $\{t_{k}\}$ is a finite set of real numbers. Then, the solutions of the above equations naturally present discontinuities due to the moments of impulse. For more details regarding this last equation we refer the reader to \cite{pouso} and \cite{marcia}. 
	
Based on these problems, \cite{keale} purpose the study of the one-sided differential equations with respect to the operator 
$$
		L_{W,V}:=\kappa^{2}-D^{+}_{V}HD^{-}_{W},
	$$
	where $D_V^+$ stands for the right-derivative with respect to $V$ and $D_W^-$ stands for left-derivative with respect to $W$. Both of these operators will be introduced rigorously in the next section. In \cite{keale}, they also introduced the so-called $W$-$V$-Sobolev space $H_{W,V}(\mathbb{T})$, which is a natural environment for the solutions of the above equation, under general conditions on $\kappa$ and $H$. Also in \cite{keale}, the authors proved the regularity of the eigenvectors associated to the problem
	\begin{equation}\label{eigeneq}
		\begin{cases}
			\Delta_{W,V}u=\lambda u;\\
			u\in \mathcal{D}_{W,V}(\mathbb{T}).
		\end{cases}
	\end{equation}

It is important to observe that there is a need for the study of the eigenvalues in the above problem. For instance, in \cite{farfansimasvalentim}, they introduced a space of test functions, which they denoted by $\mathcal{S}_W(\mathbb{T}^d)$, and proved that it is nuclear. However, due to the lack of properties regarding the eigenvalues of the differential operator $D_xD_W$, the definition of $\mathcal{S}_W(\mathbb{T}^d)$ became very abstract and artificial. 

Furthermore, given a differential operator $L$, it is crucial in the study of stochastic partial differential equations to know when the fractional operator $L^{-s}$, $s>0$, is trace class or Hilbert-Schmidt. These questions can only be answered by a detailed study on the eigenvalues of the operator $L$. 

Therefore, the main goal of this work is to provide the tools to study the eigenvalues of the generalized differential operator and, then, obtain sharp estimates of their growth. Let us now describe the remaining of the paper. We begin the paper by recalling some of the results on $W$-$V$-Sobolev spaces introduced by \cite{keale}. Then, we study the space of infinitely differential functions with respect to the operator $D_V^+D_W^-$, $C^{\infty}_{W,V}(\mathbb{T})$, and determine under what conditions a function $f$ in this space can be represented as
	$$f(x)=\sum_{k=1}^{\infty} D^{(n)}_{W,V}f(0)F_{k}(x),$$
	where the operator $D^{(n)}_{W,V}$ is obtained by recursively applying the operators $D_V^+$ and $D_W^-$. $F_{k}$ can be interpreted as a generalized polynomial. This last result provides us the tools to obtain a series expansion of the eigenvectors of the operator $D_V^+D_W^-$. By using the representation of the eigenvectors as series, we are able to use some results in complex analysis to obtain a sharp lower bound for the growth of the eigenvalues of the operator $D_V^+D_W^-$. We, then, introduce the fractional spaces $H^{s}_{W,V}(\mathbb{T})$ based on the Fourier characterization of $H_{W,V}(\mathbb{T})$ proved in \cite{keale}. At this point we are able to extend the definitions and results to dimension $d$ by taking tensor products. Finally, we apply the theory to prove existence and uniqueness of solutions of a fractional stochastic partial differential equation that generalizes the well-known Mat\'ern equation when the domain is the $d$-dimensional torus.
	
	\section{$W$-$V$-Sobolev spaces and the generalized Laplacian}
	
	In this section we will briefly recall some definitions and results obtained in \cite{keale} with respect to $W$-$V$-Sobolev spaces. First, we say that a function is c\`adl\`ag (from french, ``continue à droite, limite à gauche'') if the function is right-continuous and has limits from the left. Similarly, we say that a function is c\`agl\`ad, if the function is left-continuous and has limits from the right.  
	
	Fix two strictly increasing functions $W,V : \mathbb{T}\to\mathbb{T}$, with $W$ being càdlàg and $V$ being càglàd. We also assume they satisfy the following periodic conditions :
	\begin{equation}\label{eqsmedida}
		\forall x \in \mathbb{R},\quad 
		\begin{cases}
			W(x+1)-W(x)=W(1)-W(0);\\
			V(x+1)-V(x)=V(1)-V(0), 
		\end{cases}
	\end{equation}
	where, without loss of generality we will assume that both $W$ and $V$ are continuous at zero, that is, $W(0-)=W(0)=0$ and $V(0+)=V(0)=0$. Indeed, since they are increasing, they can only have countably many discontinuities points, so if they are not continuous at zero, we can simply choose another point in which both of them are continuous and translate the functions to make them continuous at zero.  Note that picking an compact interval $I$ in $\mathbb{R}$ with length $1$, we are able to define finite measures $dV$ and $dW$ on $\mathbb{T}$ satisfying $dW((a,b])=W(b)-W(a)$ and $dV([a,b))=V(b)-V(a)$. Thus $dW$ and $dV$ are defined on the families $\{(a,b]:a<b, a,b\in\mathbb{R}\}$ and $\{[a,b):a<b,a,b\in\mathbb{R}\}$, respectively, and both these classes generate the Borel $\sigma$-algebra of $\mathbb{T}_{I}\cong\mathbb{T}$, where $\mathbb{T}_{I}$ is the torus obtained by identifying the boundary points of $I$.
	
		Note that in the above definition we allow the measures induced by $W$ and $V$ to have atoms. This is a weaker assumption than some of the common assumptions found in literature, e.g., \cite{uta}, \cite{trig}, among others. 
	
	We will denote the $L^2$ space with the measure induced by $V$ by $L^2_V(\mathbb{T})$ and its norm (resp. inner product) by $\|\cdot\|_V$ (resp. $\langle \cdot,\cdot\rangle_V$). Similarly, we will denote the $L^2$ space with the measure induced by $W$ by $L^2_W(\mathbb{T})$ and its norm (resp. inner product) by $\|\cdot\|_W$ (resp. $\langle \cdot,\cdot\rangle_W$). The subspace of $L^2_V(\mathbb{T})$ (resp. $L^2_W(\mathbb{T})$) containing the functions
	in $L^2_V(\mathbb{T})$ (resp. $L^2_W(\mathbb{T})$)  such that $\int_{\mathbb{T}} fdV =0$ (resp. $\int_{\mathbb{T}} fdW=0$) will be denoted by $L^2_{V,0}(\mathbb{T})$ (resp. $L^2_{W,0}(\mathbb{T})$).
		
	\begin{definition}\label{defgendif2}
	We say that a function $f : \mathbb{T}\to\mathbb{R}$ is $W$-left differentiable if for every $x\in\mathbb{T}$ the limit 
	$$D^{-}_{W}f(x):=\lim_{h\to 0^{-}}\frac{f(x+h)-f(x)}{W(x+h)-W(x)}$$
	exists for all $x\in \mathbb{T}$. Similarly, we say that a function $g : \mathbb{T}\to\mathbb{R}$ is $V$-right differentiable if the limit 
	$$D^{+}_{V}g(x):=\lim_{h\to 0^{+}}\frac{g(x+h)-g(x)}{V(x+h)-V(x)}$$
	exists for all $x\in \mathbb{T}$.
\end{definition}
	
Define the sets 
	$$C_{0}(\mathbb{T}):=\left\{f:\mathbb{T}\to\mathbb{R}; f\;\;\mbox{is càdlàg},\;\int_{\mathbb{T}}fdV=0\right\},$$
$$C_{1}(\mathbb{T}):=\left\{f:\mathbb{T}\to\mathbb{R}; f\;\;\mbox{is càglàd}\int_{\mathbb{T}}fdW=0\right\},$$
and
$$C^{n}_{W,V,0}(\mathbb{T}):=\left\{f\in C_{\sigma(n+1)}(\mathbb{T}); D^{(n)}_{W,V}f\hbox{ exists and }D^{(n)}_{W,V}f\in C_{\sigma(n)}(\mathbb{T})\right\},$$
where 
\begin{equation}\label{nderiv}
D^{(n)}_{W,V}:=\left\{ \begin{array}{ll}
	\underbrace{D^{-}_{W}D^{+}_{V}...D^{-}_{W}}_{n-factors},\;\mbox{if}\; n\; \mbox{is odd}; \\
	\underbrace{D^{-}_{V}D^{+}_{W}...D^{-}_{W}}_{n-factors},\;\mbox{if}\; n\; \mbox{is even},\end{array} \right. 
\end{equation}
and 
\[\sigma(n):=\left\{ \begin{array}{ll}
	1,\;\mbox{if}\; n\; \mbox{is odd}; \\
	0,\;\mbox{if}\; n\; \mbox{is even}.\end{array} \right. \] 

	Finnaly, let us define our space of smooth functions. The elements of this space will work as test functions for our differential operators. More precisely, let
	$$C^{\infty}_{W,V}(\mathbb{T}):=\langle 1 \rangle\oplus\left(\bigcap_{n=1}^{\infty} C^{n}_{W,V,0}\right).$$
	
	These functions allow us to define weak derivatives:
	
		\begin{definition}\label{defWleft}
		A function $f\in L^{2}_{V}(\mathbb{T})$ has a $W$-left \textit{weak} derivative if, and only if, for every $g\in  C^{\infty}_{V,W}(\mathbb{T})$, there exists $F\in L^{2}_{W}(\mathbb{T})$ such that 
		\begin{equation}\label{defWd2}
			\int_{\mathbb{T}}fD_{V}^{+}gdV=-\int_{\mathbb{T}}FgdW.
		\end{equation}
	In such a case, the $W$-left weak derivative of $f$ will be denoted by $D_W^-f$. We use the same notation as the strong lateral derivative in view of Remark 9 of \cite{keale}.
	\end{definition}

We can now provide combine Definition 4 and Theorem 5 in \cite{keale} to arrive at the following definition of $W$-$V$-Sobolev spaces.

\begin{definition}
	The $W$-$V$-Sobolev space is the Hilbert space 
$${H}_{W,V}(\mathbb{T})=\{f\in L^{2}_{V}(\mathbb{T});  f\hbox{ has $W$-left weak derivative}\},$$ 
with norm
$$\|f\|^{2}_{W,V}=\|f\|^{2}_{V}+\|D_{W}^{-}f\|^{2}_{W}.$$
\end{definition}

It is noteworthy that our set of smooth functions, $C^{\infty}_{V,W}(\mathbb{T})$, is dense in $H_{W,V}(\mathbb{T})$. Indeed, the following is \cite[Proposition 2]{keale}:
\begin{proposition}\label{densenesscinfinity2}
	The space $C^{\infty}_{V,W}(\mathbb{T})$ is dense in $L^2_V(\mathbb{T})$. Furthermore, the set $\left\{D^{+}_{V}g;g\in C^{\infty}_{V,W}(\mathbb{T})\right\}$ is dense in $L^2_{V,0}(\mathbb{T})$.
\end{proposition}

We also have a characterization of the $W$-$V$-Sobolev spaces that can be seen as a counterpart to the result that says that any function in $H^1(\mathbb{T})$ is absolutely continuous. More precisely, the following is a direct consequence of \cite[Theorem 2]{keale}.
	\begin{theorem}\label{sobchar2}
	A function $f\in L^{2}_{V}(\mathbb{T})$ belongs to $H_{W,V}(\mathbb{T})$ if, and only if,
$$
		f(x)=f(0)+\int_{(0,x]}D_W^-(s)dW(s),
$$
	$V$-a.e.
\end{theorem}
	
Let us now introduce a generalization of the Laplacian. To this end, we will first introduce a new space of functions, to which the Laplacian will be a self-adjoint operator.

	\begin{definition}
	Let $\mathcal{D}_{W,V}(\mathbb{T})$ be the set of functions $f\in L^{2}_{W}(\mathbb{T})$ such that there exists $\mathfrak{f}\in L^{2}_{V,0}(\mathbb{T})$ satisfying 
	\begin{equation}\label{c422}
		f(x)=a+W(x)b+\int_{(0,x]}\int_{[0,y)} \mathfrak{f}(s) dV(s)dW(y),
	\end{equation}
	where $b$ is satisfies the relation
	$$bW(1)+\int_{(0,1]}\int_{[0,y)} \mathfrak{f}(s) dV(s)dW(y)=0.$$
	\end{definition}

We have the following integration by parts formula involving functions in $\mathcal{D}_{W,V}(\mathbb{T})$  and in ${H}_{W,V}(\mathbb{T})$. This formula is proved in \cite[Proposition 1]{keale}.

	\begin{proposition}\label{intbypartsprop2}
	(Integration by parts formula) For every $f\in\mathcal{D}_{W,V}(\mathbb{T})$ and $g\in H_{W,V}(\mathbb{T})$, the following expression holds:
	\begin{equation}\label{intbypartsVW2}
		\langle -\Delta_{W,V}f,g\rangle_{V} = \int_{\mathbb{T}}D^{-}_{W}fD^{-}_{W}gdW
	\end{equation}
\end{proposition}

We can now define the $W$-$V$-Laplacian:

	\begin{definition}\label{laplacianWV2}
	We define the $W$-$V$-Laplacian as $\Delta_{W,V}: \mathcal{D}_{W,V}(\mathbb{T})\subseteq L^{2}_{V}(\mathbb{T})\to L^{2}_{V}(\mathbb{T})$, given by $\Delta_{W,V} f = \mathfrak{f}$, where $\mathfrak{f}$
	is defined by \eqref{c422}.
\end{definition}

We have from \cite[Definition 5]{keale} that $\mathcal{D}_{W,V}(\mathbb{T})$ is the domain of the Friedrichs extension of $I-D^{+}_{V}D^{-}_{W}$ denoted by $\mathcal{A}$ (we refer the reader to Zeidler \cite[Section 5.5]{zeidler} for further details on Friedrichs extensions). Therefore, the formal $W$-$V$-Laplacian, $\Delta_{W,V}: \mathcal{D}_{W,V}(\mathbb{T})\subseteq L^{2}_{V}(\mathbb{T})\to L^{2}_{V}(\mathbb{T})$, is self-adjoint. Furthermore, by \cite[Theorem 3]{keale}, $(I-\Delta_{W,V})^{-1}$ is well-defined and compact. Therefore, there exists a complete ortonormal system of functions $(\nu_n)_{n\in\mathbb{N}}$ in $L^{2}_{V}(\mathbb{T})$ such that $\nu_{n}\in H_{W,V}(\mathbb{T})$ for all $n$, and $\nu_{n}$ solves the equation $(I - \Delta_{W,V})\nu_n=\gamma_n \nu_n,$ for some $\{\gamma_n\}_{n\in\mathbb{N}}\subset\mathbb{R}.$ Furthermore,
$$1\le \gamma_{1}\le\gamma_{2}\cdots\to\infty$$ 
as $n\to\infty$. Now, observe that $\nu\in L^2_V(\mathbb{T})$ is an eigenvector of $\Delta_{W,V}$ with eigenvalue $\lambda$ if, and only if, $\nu$ is an eigenvector of $\mathcal{A}$ with eigenvalue $\gamma=1+\lambda$. We, then, have that $\{\nu_{n},\lambda_{n}\}_{n\in \mathbb{N}}$ forms a complete orthonormal system of $L^{2}_{V}(\mathbb{T})$, where $\lambda_n = \gamma_n - 1$. Moreover, we have that
\begin{equation}\label{asymplambda}
0=\lambda_{0}\le \lambda_{1}\le\lambda_{2}\cdots\to\infty,
\end{equation}
as $n\to\infty$, where for each $k\in\mathbb{N} \setminus \{0\}$, $\lambda_k$ satisfies

\begin{equation}\label{autovt2}
	-\Delta_{W,V}\nu_{k}=\lambda_{k} \nu_{k}.
\end{equation} 

This allows us to define the elements in $H_{W,V}(\mathbb{T})$ in terms of its Fourier coefficients. Indeed, \cite[Theorem 6]{keale} tells us that
\begin{equation}\label{fourierHWV}
			H_{W,V}(\mathbb{T})=\left\{f\in L^{2}_{V}(\mathbb{T}); f=\alpha_{0}+\sum_{i=1}^{\infty}\alpha_{i}\nu_{i}; \sum_{i=1}^{\infty}\lambda_{i}\alpha_{i}^2<\infty\right\}.
\end{equation}

	Finally, it is important to notice that we can interchange the places of $V$ and $W$ and obtain ``dual'' versions of the above results. More precisely,
	
	\begin{remark}\label{orderVW2}
		In the rest of the paper we need to pay attention with the position of $W$ and $V$ on the subindex. We will use the subscript $W,V$ for every structure stritly related to the operator $D^{+}_{V}D^{-}_{W}$. Similarly, whenever we use the subscript $V,W$, we will be referring to the, analogous, structure related to the operator $D^{-}_{W}D^{+}_{V}$. Notice that by doing this, we will keep changing between c\`adl\`ag and c\`agl\`ad functions. We ask the viewer to be attentive to these details as they are subtle.
	\end{remark}

	\section{Generalized Maclaurin expansion for functions in $C^{\infty}_{W,V}(\mathbb{T})$.}

The goal of this first section is to provide an analogue to Maclaurin series expansion for functions in $C^{\infty}_{W,V}(\mathbb{T})$. Our method is based on Taylor's expansion with integral remainder. We begin by identifying which expression will ``play the role'' of the terms of the form $\dfrac{x^n}{n!}$, and we which terms will ``play the role'' of the integral remainder in our version of the Maclaurin series expansion. 

In the second moment we will find estimates for our remainder and obtain semi-differential properties related to the sequence $F_{n}(x,s)$ introduced in (\ref{term:1}). Finally our main Theorem in this section will provide a sufficient condition to a function in $C^{\infty}_{W,V}(\mathbb{T})$ to be represented according a generalized version of the Maclaurin series expansion, whose expression is given in (\ref{sdetaylor}).
	
Let us begin with the computations. Let $f\in C^{\infty}_{W,V}(\mathbb{T})$. We have, by \eqref{sobchar2} and \eqref{intbypartsVW2}, that
	\begin{align*}
		f(x)-f(0) &= \int_{(0,x]}D^{-}_{W}f(s)dW(s) = \int_{(0,x]}D^{-}_{W}f(s)D^{-}_{W}W dW(s) \\
		&= D^{-}_{W}f(x)W(x)-\int_{[0,x)}D^{+}_{V}D^{-}_{W}f(s)W(s)dV(s) \\
		&= D^{-}_{W}f(0)W(x)+\int_{[0,x)}D^{+}_{V}D^{-}_{W}f(s)\left[W(x)-W(s)\right]dV(s) \\
		&= D^{-}_{W}f(0)W(x)+I_{2}(x),
	\end{align*}
	where $I_{2}(x)=\int_{[0,x)}D^{+}_{V}D^{-}_{W}f(s)\left[W(x)-W(s)\right]dV(s)$. Now, define 
	$$F_{2}(x,s)=V(s)W(x)-\int_{[0,s)}W(\xi)dV(\xi)=\int_{[0,s)}\left[W(x)-W(\xi)\right]dV(\xi).$$ 
	We have that $D^{+}_{V,s}F_{2}(x,s)=W(x)-W(s)$. We can use integration by parts again, however this time interchanging the roles of $V$ and $W$, see Remark \ref{orderVW2}, to obtain
	\begin{align*}
		I_{2}(x) &= \int_{[0,x)}D^{+}_{V}D^{-}_{W}f(s)\left[W(x)-W(s)\right]dV(s) \\
		&= D^{+}_{V}D^{-}_{W}f(x)F_{2}(x,x)-D_{V}^{+}D_{W}^{-}f(0)F(0)-\int_{(0,x]}D^{-}_{W}D^{+}_{V}D^{-}_{W}f(s)F_{2}(s)dW(s) \\
		&= D^{+}_{V}D^{-}_{W}f(x)F_{2}(x,x)-\int_{(0,x]}D^{-}_{W}D^{+}_{V}D^{-}_{W}f(s)F_{2}(x,s)dW(s)\\
		&= D^{+}_{V}D^{-}_{W}f(0)F_{2}(x,x)+\int_{(0,x]}D^{-}_{W}D^{+}_{V}D^{-}_{W}f(s)\left[F_{2}(x,x)-F_{2}(x,s)\right]dW(s)\\
		&= D^{+}_{V}D^{-}_{W}f(0)F_{2}(x,x)+I_{3}(x).
	\end{align*}
	where $I_{3}(x)=\int_{(0,x]}D^{-}_{W}D^{+}_{V}D^{-}_{W}f(s)\left[F_{2}(x,x)-F_{2}(x,s)\right]dW(s).$
	In short,
	$$
	f(x)-f(0) = D^{-}_{W}f(0)W(x)+D^{+}_{V}D^{-}_{W}f(0)F_{2}(x,x)+I_{3}(x).
	$$
	By induction, we have
	\begin{align}\label{taylorpolynomial}
		f(x)-f(0) = \sum_{k=1}^{n-1}\left[D^{(n)}_{W,V}f(0)\right]F_{k}(x,x)+I_{n}(x),
	\end{align}
	where $D^{(n)}_{W,V}$ was defined in \eqref{nderiv}, $F_{1}(x,s)=W(s)$,
	\begin{equation}\label{term:1}
		F_{n}(x,s)=\begin{cases}
			\displaystyle\int_{[0,s)}\left[F_{n-1}(x,x)-F_{n-1}(x,\xi)\right]dV(\xi),\; \mbox{n even};\\
			\displaystyle\int_{(0,s]}\left[F_{n-1}(x,x)-F_{n-1}(x,\xi)\right]dW(\xi),\; \mbox{n odd},
		\end{cases}
	\end{equation}
	and the remainder is given by
	\[I_{n}(x)=\left\{ \begin{array}{ll}
		\displaystyle\int_{[0,x)}D^{(n+1)}_{W,V}f(s)\left[F_{n}(x,x)-F_{n}(x,s)\right]dW(s),\; \mbox{n even};\\
		\displaystyle\int_{(0,x]}D^{(n+1)}_{W,V}f(s)\left[F_{n}(x,x)-F_{n}(x,s)\right]dV(s),\; \mbox{n odd}.\end{array} \right. \]
	\paragraph{}
	The next Lemma will help us to estimate the remainder $I_{n}(x)$.
	\begin{lemma}\label{taylorbound} Fix $x\in \mathbb{T}$. If $(F_{n})_{n\in\mathbb{N}}$ is the sequence defined in \eqref{term:1}, then, for $x,s\in\mathbb{T}$, we have the following bound:
	\end{lemma}
	\begin{equation}\label{estimativadoresto}
		|F_{2n}(x)-F_{2n}(s)|\le \dfrac{\left|F_{2}(x)-F_{2}(s)\right|^n}{n!},
	\end{equation}
	where for convenience $F_{n}(s):=F_{n}(x,s).$
	\begin{proof}
		Take $s<x$ and use \eqref{term:1} to obtain
		\begin{align}\label{recorrencia1}
			\nonumber F_{2n}(x)-F_{2n}(s) &=\int_{[s,x)}F_{2n-1}(x)-F_{2n-1}(\xi)dV(\xi)\\ \nonumber 
			&= \int_{[s,x)}\left(\int_{(\xi,x]}F_{2(n-1)}(x)-F_{2(n-1)}(\eta)dW(\eta) \right) dV(\xi) 
			\\ \nonumber 
			&\le  \int_{[s,x)}\left(F_{2(n-1)}(x)-F_{2(n-1)}(\xi+)\right)(W(x)-W(\xi))dV(\xi) \\
			&= \int_{[s,x)}F_{2(n-1)}(x)-F_{2(n-1)}(\xi+)dF_{2}(\xi).
		\end{align}
		In a similar manner, we have
		\begin{equation}\label{recorrencia2}
			F_{2n}(x)-F_{2n}(\xi+)\le \int_{(\xi,x)} F_{2(n-1)}(x)-F_{2(n-1)}(\beta+)dF_{2}(\beta).   \end{equation}
		By looking at the inequalities (\ref{recorrencia1}) and (\ref{recorrencia2}), we note that the result follows by showing that the bound holds for the following expression:
		\begin{equation}\label{termointegral}
			\int_{[s,x)}dF_{2}(\xi_{1})\int_{(\xi_{1},x)}dF_{2}(\xi_{1})\cdots \int_{(\xi_{n-1},x)}(F_{2}(x)-F_{2}(\xi_{n}+))dF_{2}(\xi_{n}).
		\end{equation}
		Now, set $c=F_{2}(x)$ and observe that
		$$D^{+}_{F_{2}}(c-F_{2})^{n+1}(\xi)=-\sum_{j=0}^{n}\left(c-F_{2}(\xi+)\right)^{j}\left(c-F_{2}(\xi)\right)^{n-j}\le -(n+1)(c-F_{2}(\xi+))^{n}.$$
		
		Now, by integrating the above inequality and noting that $\frac{(c-F_{2})^{n+1}}{n+1}$ is $F_{2}$-absolutely continuous in the sense of \cite{pouso}, we obtain that
		\begin{equation}\label{desig1}
			\dfrac{(c-F_{2}(s))^{n+1}}{n+1}\geq \int_{[s,x)}(c-F_{2}(\xi+))^{n}dF_{2}(\xi).
		\end{equation}
		Now, given any $s^\ast <x$, we can take the limit from the right as $s\to s^\ast+$ on (\ref{desig1}) to obtain that
		\begin{equation}\label{desig2}
			\dfrac{(c-F_{2}(s^\ast+))^{n+1}}{n+1}\geq \int_{(s^\ast,x)}(c-F_{2}(\xi+))^{n}dF_{2}(\xi).
		\end{equation}
		The inequality (\ref{estimativadoresto}) then follows from recursively applying (\ref{desig1}) and (\ref{desig2}) on (\ref{termointegral}).
	\end{proof}
	The next Lemma can be viewed as a version of Leibniz integral rule with respect to the $W$-left-derivative. It will help us when computing the $W$-left-derivative of $F_{n}.$
	\begin{lemma}\label{leibnizrule}
		Let $g:\mathbb{T}\times \mathbb{T}\to \mathbb{R}$ be a function such that:
		\begin{enumerate}
			\item $\forall x\in \mathbb{T},\,g(x,\cdot):\mathbb{T}\to \mathbb{R}$ is a càglàd function;
			\item $\exists D^{-}_{W,1}g:\mathbb{T}\times \mathbb{T}\to \mathbb{R}$ such that $\forall x\in\mathbb{T}$, $D^{-}_{W,1}g(x,\cdot):\mathbb{T}\to \mathbb{R}$ is a càglàd function and 
			
			$$\lim_{h\to 0} \sup_{\xi\in\mathbb{T}}\left|\dfrac{g(x,\xi)-g(x-h,\xi)}{W(x)-W(x-h)}-D^{-}_{W,1}g(x,\xi)\right|=0;$$
			\item $\forall x\in\mathbb{T},\, g(x,x)=D_{W,1}g(x,x)=0$.
		\end{enumerate}
		Then 
		$$D^{-}_{W}\left(\int_{(0,\cdot]}g(\cdot,\xi)dW(\xi)\right)(x)=\int_{(0,x]}D^{-}_{W,1}g(x,\xi)dW(\xi).$$
	\end{lemma}
	\begin{proof}
		Fix $x\in\mathbb{T}$ and $h>0$. Now, consider the quotient
		\begin{align}
			\nonumber \Delta(x,h) &=\dfrac{1}{W(x)-W(x-h)}\left(\int_{(0,x]}g(x,\xi)dW(\xi)-\int_{(0,x-h]}g(x-h,\xi)dW(\xi)\right)\\ \nonumber 
			&= \dfrac{1}{W(x)-W(x-h)}\int_{(x-h,x]}g(x,\xi)dW(\xi)\\
			&+\int_{(0,x-h]}\left[\dfrac{g(x,\xi)-g(x-h,\xi)}{W(x)-W(x-h)}-D_{W,1}^{-}g(x,\xi)\right]dW(\xi)
			\\ \nonumber &+\int_{(0,x-h]}D_{W,1}^{-}g(x,\xi)dW(\xi).
		\end{align}
		Now, observe that items $\textit{1.-3.}$ directly imply that $$\lim_{h\to 0}\Delta(x,h)=\int_{(0,x]}D_{W,1}^{-}g(x,\xi)dW(\xi).$$
	\end{proof}
	
	\begin{lemma}\label{taylorlemma}
		Define $G_{1}(x,s)=V(s)$ and 
		\begin{equation}\label{Gnformula}
			G_{n}(x,s)=\left\{ \begin{array}{ll}
			\displaystyle\int_{(0,s]}\left[G_{n-1}(x,x)-G_{n-1}(x,\xi)\right]dW(\xi),\; \mbox{n even};\\
			\displaystyle\int_{[0,s)}\left[G_{n-1}(x,x)-G_{n-1}(x,\xi)\right]dV(\xi),\; \mbox{n odd}.\end{array} \right. 
		\end{equation}
		Then
		
		\[
		\left\{ \begin{array}{ll}
			D^{-}_{W,1}F_{n}(\xi,s)=G_{n-1}(\xi,s);\\\\
			D^{-}_{W}F_{n}(\xi,\xi)=G_{n-1}(\xi,\xi),
		\end{array}
		\right.
		\]
		where $D^{-}_{W,1}$ was defined in Item 2 of Lemma \ref{leibnizrule}.
		
	\end{lemma}
	
	\begin{proof}
		We will proceed by induction in $n$. First, for $n=2$ we have 
		$$F_{2}(x,s)=\int_{[0,s)}W(x)-W(\xi)dV(\xi)=V(s)W(x)-\int_{[0,s)}W(\xi)dV(\xi),$$
		which implies that $D^{-}_{W,1}F_{2}(x,s)=V(s)=G_{1}(x,s)$. Moreover
		$\frac{F_{2}(x-h,s)-F_{2}(x,s)}{W(x-h)-W(x)}\to V(s)$ when $h\to0$, with this limit holding uniformly in $s$. By the integration by parts formula (see Proposition \ref{intbypartsprop2}), we have that
		$$F_{2}(x,x)=\int_{(0,x]}V(\alpha)dW(\alpha).$$
		Therefore, $D^{-}_{W}F_{2}(x,x)=V(x)=G_{1}(x,x)$. Assume that for an odd index $n$ we have for all $k\le n$ and all $x$, the existence of $D^{-}_{W,1}F_{k}(x,s)=G_{k-1}(x,s)$, uniformly in $s$, and that $D^{-}_{W}F_{k}(x,x)=G_{k-1}(x,x)$. Now, turning our attention to the case $F_{n+1}(x,x)$, we have, by the Fubini's theorem and the definition of $F_{n}$, that
		\begin{align}\label{fubinada}
			\nonumber F_{n+1}(x,x) &=\int_{[0,x)}F_{n}(x,x)-F_{n}(x,\xi)dV(\xi)\\ \nonumber 
			&= \int_{[0,x)}\left(\int_{(\xi,x]}F_{n-1}(x,x)-F_{n-1}(x,\alpha)dW(\alpha)\right)dV(\xi)
			\\ &=\int_{(0,x]}[F_{n-1}(x,x)-F_{n-1}(x,\alpha)]V(\alpha)dW(\alpha).
		\end{align}
		By using the induction hypothesis, we obtain that the function $g(x,\alpha)=[F_{n-1}(x,x)-F_{n-1}(x,\alpha)]V(\alpha)$ satisfies all conditions to apply the Lemma \ref{leibnizrule}, with $D^{-}_{W,1}g(x,\alpha)=[G_{n-2}(x,x)-G_{n-2}(x,\alpha)]V(\alpha)$. Thus, Lemma \ref{leibnizrule} yields
		$$D^{-}_{W}F_{n+1}(x,x)=\int_{(0,x]}[G_{n-2}(x,x)-G_{n-2}(x,\alpha)]V(\alpha)dW(\alpha).$$
		By using a reasoning similar to the one in (\ref{fubinada}), we obtain that $D^{-}_{W}F_{n+1}(x,x)=G_{n}(x,x).$
	\end{proof}
	\begin{remark}\label{rmk111}
		Observe that Lemma \ref{taylorlemma}, provides us with another point of view of how to look at $F_{n}(x,x)$. Indeed, let $(p_{0}(x),q_{0}(x))=(1,W(x))$ and, for $n>0$ define recursively
		$$(p_{n+1}(x),q_{n+1}(x))=\left(\int_{(0,x]}\int_{[0,s)}p_{n}(\xi)dV(\xi)dW(s),\int_{(0,x]}\int_{[0,s)}q_{n}(\xi)dV(\xi)dW(s)\right).$$
		Now, let $F_{0}(x,x)\equiv 1$. Then, we have that, for $n\geq 0,$ $$(p_{n}(x),q_{n}(x))=(F_{2n}(x,x),F_{2n+1}(x,x)).$$
		The reader is invited to determine analogue expressions for $G_{n}.$
	\end{remark}
	\begin{theorem}\label{analyticalVW}
		Let $f\in C^{\infty}_{W,V}(\mathbb{T})$ be a function such that $$\max\left\{\|D^{(2n)}_{W,V}f\|_{\infty},\|D^{(2n+1)}_{W,V}f\|_{\infty}\right\}\le c_{n},$$  where $c_{n}=o\left(\dfrac{n!}{e^{n}}\right)$. Then, we have the following expansion for $f$:
		\begin{equation}\label{sdetaylor}
			f(x)=f(0)+\sum_{k=1}^{\infty}\left[D^{(k)}_{W,V}f(0)\right]F_{k}(x,x).
		\end{equation}
	Furthermore, the convergence in \eqref{sdetaylor} is uniform in $x\in\mathbb{T}$. 
		Moreover the conditions for that $f$ and $D^{(n)}_{W,V}f$ must satisfy in order to be well defined on the torus are:
		
		\[\left\{ \begin{array}{ll}
			\sum_{k=1}^{\infty}D^{k}_{W,V}f(0)F_{k}(1,1)&=0;\\\\
			\sum_{k=1}^{\infty}D^{k+1}_{W,V}f(0)G_{k}(1,1)&=0; \\\\ \sum_{k=1}^{\infty}D^{k+2}_{W,V}f(0)F_{k}(1,1)&=0;\\ 
			&\vdots \\  \end{array} \right. \] 
	\end{theorem}
	
	\begin{proof}
		We begin by using (\ref{estimativadoresto}) to estimate the remainder $I_{n}$, which is given by the ``polynomial'' expansion \eqref{taylorpolynomial}, in terms of $\|D^{n}f\|_{\infty}$. Indeed,
		
		\begin{align*}
			I_{2n}(x)&=\int_{(0,x]}D_{W,V}^{(2n+1)}f(s)F_{2n}(x,x)-F_{2n}(x,s)dW(s)\\ 
			&\le \dfrac{(F_{2}(1-,1-))^{n}}{n!}\|D_{W,V}^{(2n+1)}f\|_{\infty},
		\end{align*}
		and 
		\begin{align*}
			I_{2n+1}(x)&=\int_{(0,x]}D_{W,V}^{(2n+2)}f(s)F_{2n+1}(x,x)-F_{2n+1}(x,s)dW(s)\\
			&\le \|D_{W,V}^{(2n+2)}f\|_{\infty}\int_{[0,x)}F_{2n+1}(x,x)-F_{2n+1}(x,s)dV(s) \\
			&=\|D_{W,V}^{(2n+2)}f\|_{\infty}\int_{[0,x)}\left(\int_{(s,x]}F_{2n}(x,x)-F_{2n}(x,\alpha)dW(\alpha)\right)dV(s) \\
			&\le \|D_{W,V}^{(2n+2)}f\|_{\infty}\int_{[0,x)}F_{2n}(x,x)-F_{2n}(x,\alpha+) dF_{2}(s) \\
			&\le \|D_{W,V}^{(2n+2)}f\|_{\infty}\dfrac{F_{2}(1-,1-)^{n+1}}{(n+1)!}.
		\end{align*}
		These estimates of $I_{n}(x)$ yield the uniform convergence of $$S_{n}(x):=f(0)+\sum_{k=1}^{n-1}\left[D^{(k)}_{W,V}f(0)\right]F_{k}(x,x)$$ to $f(x)$.
		Finally, the conditions that $f$ and $D_{W,V}^{n}f$ must satisfy to be well-defined on the torus follow from Lemma \ref{taylorlemma} and the uniform convergence of $S_{n}(x)$.
	\end{proof} 

To summarize, expression \eqref{taylorpolynomial} can be seen as a Taylor's expansion (centered at zero) with integral remainder, whereas  Theorem \ref{analyticalVW}, provides a condition for a function $f\in C^{\infty}_{W,V}(\mathbb{T})$ to have a series representation given by \eqref{sdetaylor}. The functions in $C^{\infty}_{W,V}(\mathbb{T})$ that satisfy \eqref{sdetaylor} are our analogue of analytical functions.

\section{Characterization of the eigenvectors of the $W$-$V$-Laplacian}\label{sect42}
	
	Our goal in this section is to use Theorem \ref{analyticalVW} to define suitable functions, that generalize the usual cosine and sine, in the sense that if $V(x)=W(x)=x$, then they will be the cosine and sine functions. By using these functions, we will be able to fully characterize all eigenvectors of the $W$-$V$-Laplacian, $\Delta_{W,V}$.
	
	Let us begin, by defining the functions that will generalize the cosine and sine functions, respectively. Indeed, let $\alpha>0$ and define the following functions from $\mathbb{R}$ to $\mathbb{R}$: $$C_{W,V}(\alpha,x)=\sum_{n=0}^{\infty}\alpha^{2n}(-1)^{n}F_{2n}(x,x),\,\,\,\,\,\,\,S_{W,V}(\alpha,x)=\sum_{n=0}^{\infty}\alpha^{2n+1}(-1)^{n}F_{2n+1}(x,x)$$
	and $$C_{V,W}(\alpha,x)=\sum_{n=0}^{\infty}\alpha^{2n}(-1)^{n}G_{2n}(x,x),\,\,\,\,\,\,\,S_{V,W}(\alpha,x)=\sum_{n=0}^{\infty}\alpha^{2n+1}(-1)^{n}G_{2n+1}(x,x).$$
	
	By computing $F_{n}$ and $G_{n}$ when $V(x)=W(x)=x$, we have that the functions $C_{W,V}$ and $S_{W,V}$ (or $C_{V,W}$ and $S_{V,W}$), indeed, agree with the cosine and sine functions, respectively.
		
	Note that the functions $C_{W,V}$ and $S_{W,V}$ (or $C_{V,W}$ and $S_{V,W}$) are well-defined. Indeed, by using Lemma \ref{taylorbound}, we obtain that each of these series of functions are well defined, converge uniformly, and converges absolutely on each compact interval of $\mathbb{R}$. 
	
	Now, by using Lemma \ref{leibnizrule}, we have that
	\begin{equation}\label{difcossinWV}
	D^{-}_{W}C_{W,V}(\alpha,x)=-\alpha S_{V,W}(\alpha,x),\,\,\,\,\,\,\,D^{-}_{W}S_{W,V}(\alpha,x)=\alpha C_{V,W}(\alpha,x)
	\end{equation}
	and
	\begin{equation}\label{difsincosVW}
	D^{+}_{V}C_{V,W}(\alpha,x)=-\alpha S_{W,V}(\alpha,x),\,\,\,\,\,\,\,D^{+}_{V}S_{V,W}(\alpha,x)=\alpha C_{W,V}(\alpha,x).
	\end{equation}
	The above expressions generalize the well-known $\sin'(x) = \cos(x)$ and $\cos'(x)=-\sin(x)$.

	Now, we can use the relations \eqref{difcossinWV} and \eqref{difsincosVW} to obtain that $S_{W,V}(\alpha,x)$ is the unique solution in $C^{\infty}_{W,V}(\mathbb{T})$ of 
	\begin{equation}\label{sineqpart}
		\left\{ \begin{array}{ll}
			-\Delta_{W,V}u=\alpha^2 u;\\
			u(0)=0, D^{-}_{W}u(0)=\alpha. \end{array} \right. 
	\end{equation}
	and that $C_{W,V}(\alpha,x)$  is the unique solution in $C^{\infty}_{W,V}(\mathbb{T})$ of 
	\begin{equation}\label{coseqpart}
	\left\{ \begin{array}{ll}
		-\Delta_{W,V}u=\alpha^2 u;\\
		u(0)=1, D^{-}_{W}u(0)=0. \end{array} \right. 
	\end{equation}

Now, observe that since $S_{W,V}(\alpha,\cdot)$ solves \eqref{sineqpart} and $C_{W,V}(\alpha,\cdot)$ solves \eqref{coseqpart}, it is easy to see that for any $\alpha\neq 0$, $S_{W,V}(\alpha,\cdot)$ and $C_{W,V}(\alpha,\cdot)$ are linearly independent.
	
	Therefore, by the linear independence of $S_{W,V}(\alpha,x)$ and $C_{W,V}(\alpha,x)$, for $\alpha\neq 0$, we have that any solution $u_{\alpha}$, of $-\Delta_{W,V}u=\alpha^2 u$ is given by
	$$u_{\alpha}(x)=A\cdot C_{W,V}(\alpha,x)+B\cdot S_{W,V}(\alpha,x)$$ where A and B are determined by the initial conditions at $x=0$.
	
	The next result characterizes the eigenvectors of $\Delta_{W,V}$ on the torus by using the functions $C_{W,V}(\alpha,\cdot)$ and $S_{W,V}(\alpha,\cdot)$. The reader is invited do compare this characterization with the Taylor's expansions of the eigenvectors of the standard Laplacian, $-\Delta$, on the torus.
	\begin{proposition}\label{caraceige}
		If $(\lambda_{i},\nu_{i})_{i>0}$ satisfy 
		(\ref{eigeneq}), then there exist $a_i,b_i\in\mathbb{R}$ such that
		\begin{equation}\label{s01}
			\nu_{i}(x)=a_{i}C_{W,V}\left(\sqrt{\lambda_{i}},x\right)+\dfrac{b_{i}}{\sqrt{\lambda_{i}}}S_{W,V}\left(\sqrt{\lambda_{i}},x\right),
		\end{equation}
		$V$-$a.e.$. Furthermore, for each $i$, the vector $(a_{i},b_{i})\neq(0,0)$ is obtained as solution of the system:
		\begin{equation}\label{bcondi}
			\begin{cases}
				a_{i}\left[C_{W,V}\left(\sqrt{\lambda_{i}},1\right)-1\right]+\dfrac{b_{i}}{\sqrt{\lambda_{i}}}S_{W,V}(\sqrt{\lambda_{i}},1)=0; \\
				\dfrac{a_{i}}{\sqrt{\lambda_{i}}}S_{V,W}(\sqrt{\lambda_{i}},1)-\dfrac{b_{i}}{\lambda_{i}}\left[C_{V,W}\left(\sqrt{\lambda_{i}},1\right)-1\right]=0.
			\end{cases}
		\end{equation}
	\end{proposition}
	
	\begin{proof}
		We begin by noting that, by \eqref{c422}, the eigenvector associated to $\lambda_{i}>0$ satisfies 
		\begin{equation}\label{caracauto}
			\nu_{i}(x)=a_{i}+b_{i}W(x)-\lambda_{i}\int_{(0,x]}\int_{[0,s)}\nu_{i}(\xi)dV(\xi)dW(s),
		\end{equation}
		with $a_{i}$ and $b_{i}$ fulfilling the relations $\int_{\mathbb{T}}\nu_{i}dV=0$ and $b_{i}W(1)+\int_{\mathbb{T}}\int_{[0,s)}\nu_{i}(\xi)dV(\xi)dW(s)=0.$ So up to a set of $V$-measure zero we can recursively substitute (\ref{caracauto}) into itself, obtaining on the right side an absolutely and uniformly convergent series. Since we have absolute convergence, we can separate the into even and odd terms, thus obtaining (\ref{s01}). Finally, by applying (\ref{caracauto}) on $\int_{\mathbb{T}}\nu_{i}dV=0$ and $b_{i}W(1)+\int_{\mathbb{T}}\int_{[0,s)}\nu_{i}(\xi)dV(\xi)dW(s)=0$, we obtain the relations in (\ref{bcondi}).
	\end{proof}
	\begin{remark}
		We used Remark \ref{rmk111} when doing the substitution steps in the proof of Proposition \ref{caraceige} to reduce the amount of computations.
	\end{remark}

As a direct consequence of Proposition \ref{caracauto}, we have (another proof of) the regularity of eigenvectors of $\Delta_{W,V}$ (compare with \cite[Theorem 4]{keale}):

	\begin{corollary}[Regularity of the eigenvectors]\label{regeigenvect}
		If $u\in L^{2}_{V}(\mathbb{T})\setminus \{0\}$ satisfies
		\begin{equation}
			\begin{cases}
				\Delta_{W,V}u=\lambda u;\\
				u\in \mathcal{D}_{W,V}(\mathbb{T}).
			\end{cases}
		\end{equation}
		for some $\lambda>0$, then there is $v\in C^{\infty}_{W,V}(\mathbb{T})\cap L^{2}_{V,0}(\mathbb{T})$ such that $u=v$ V-$a.e.$
	\end{corollary}
	
	\section{Higher order $W$-$V$-Sobolev spaces}
	
	We will now introduce the $W$-$V$-Sobolev spaces of higher orders in a natural manner. First, let us provide a Definition of $V$-right weak derivative, which is similar to Definition \ref{defWleft}.
	
	\begin{definition}
Let $f\in L^{2}_{W}(\mathbb{T})$. We say that $f$ has $V$-right weak derivative if there exists $F\in L^{2}_{V}(\mathbb{T})$ such that, for all $g\in C^{\infty}_{W,V}(\mathbb{T})$, we have
\begin{equation}\label{intbparts}
	\int_{\mathbb{T}}fD_{W}^{-}g dW=-\int_{\mathbb{T}}FgdV.
\end{equation}
 If there is a function F satisfying (\ref{intbparts}) this function is unique and we denote it by $F:=D_{V}^{+}f$. We use the same notation as the one used in lateral derivative in view of Remark 9 in \cite{keale}.
	\end{definition}

We are now in a position to define the higher order $W$-$V$-Sobolev spaces.
	
	\begin{definition}\label{highersob}
		Given $k\in\mathbb{N}$, we define the $W$-$V$-Sobolev space of order $k$ as
		$$H^{k}_{W,V}(\mathbb{T})=\left\{f\in L^{2}_{V}(\mathbb{T}); \exists D_{W,V}^{(n)}f\in L^{2}_{\kappa(n)}(\mathbb{T})\,\forall n=1,\ldots,k\right\},$$
		where $D_{W,V}^{(n)}$ is defined in the same manner as in expression \eqref{nderiv},
		and 
		\begin{equation}\label{kappa_order}
			\kappa(n):=\left\{ \begin{array}{ll}
				W,\;\mbox{if}\; n\; \mbox{is odd}; \\
				V,\;\mbox{if}\; n\; \mbox{is even}.\end{array} \right. 
		\end{equation}
		We endow $H^{k}_{W,V}(\mathbb{T})$ with the norm
		$$\|f\|^2_{k,W,V}(\mathbb{T}):=\|f\|_{V}^{2}+\sum_{i=i}^{k}\|\partial^{(i)}_{W,V}f\|_{\kappa(i)}^{2}.$$
	\end{definition}
	
	\begin{remark}
		It is easy to notice that the space $H^{k}_{W,V}(\mathbb{T})$, endowed with the norm $\|f\|^2_{k,W,V}$,
		is a Hilbert space.
	\end{remark}
	
	\begin{remark}\label{approxsmoothhigh}
		By proceeding as in \cite[Corollary 2]{keale}, we could also define $H^k_{W,V}(\mathbb{T})$ as $$H^{k}_{W,V}(\mathbb{T})=\overline{C^{\infty}_{W,V}(\mathbb{T})}^{\|.\|_{k,W,V}}.$$
	\end{remark}
	The next result characterize the space $H^{k}_{W,V}(\mathbb{T})$ in terms of its Fourier coeficients, which generalizes expression \eqref{fourierHWV}:
	\begin{theorem}\label{fouriersob}
		We have the following characterization 
		$$H^{k}_{W,V}(\mathbb{T})=\left\{f\in L^{2}_{V}(\mathbb{T}); f=\alpha_{0}+\sum_{i=1}^{\infty}\alpha_{i}\nu_{i}\hbox{ and } \sum_{i=1}^{\infty}\lambda_{i}^{k}\alpha_{i}^2<\infty\right\}.$$
	\end{theorem}
	\begin{proof} Let $k$ be odd and $f=\alpha_{0}+\sum_{i=1}^{\infty}\alpha_{i}\nu_{i}$. Then, by using Lemma \ref{Wderivclosed}, we have that $$\sum_{i=1}^{\infty}\alpha_{i}\lambda_{i}^{\frac{k-1}{2}}\sqrt{\lambda_{i}}\frac{D^{-}_{W}\nu_{i}}{\sqrt{\lambda_i}}=D^{(k)}_{W,V}f\in L^{2}_{W}(\mathbb{T}).$$
		Indeed, by Parseval's identity we have that
		$$ \sum_{i=1}^{\infty}\alpha_{i}^2\lambda_{i}^{k}=\|D^{(k)}_{W,V}f\|^{2}_{W}<\infty.$$
		Conversely, let $f=\alpha_{0}+\sum_{i=1}^{\infty}\alpha_{i}\nu_{i}\hbox{ and } \sum_{i=1}^{\infty}\lambda_{i}^{k}\alpha_{i}^2<\infty$. Further, recall that, by \eqref{asymplambda}, $\lambda_{i}\to\infty$ as $i\to\infty.$ This implies that, for large $i\in\mathbb{N}$,  we have $\lambda_{i}>1$. Therefore, if $\sum_{i=1}^{\infty}\alpha_{i}^2\lambda_{i}^{k}<\infty$, then $\sum_{i=1}^{\infty}\alpha_{i}^2\lambda_{i}^{j}<\infty$, for $j\le k.$ Now, let $(f_k)_{k\in\mathbb{N}}$ be the sequence given by $$f_{k}=\alpha_{0}+\sum_{i=1}^{k}\alpha_{i}\nu_{i}.$$
		We have that $f_{k}\in C^{\infty}_{W,V}(\mathbb{T})$ and that $D^{(n)}_{W,V}f_{k}$ is Cauchy in $L^{2}_{\kappa(n)}(\mathbb{T})$ for $n=1,\ldots,k$. Therefore, $f_{k}\to f$ with respect to the norm $\|.\|_{k,W,V}.$ The result, thus, follows from Remark \ref{approxsmoothhigh}.
	\end{proof}
	
	We have the following strong regularity for these higher order $W$-$V$-Sobolev spaces:
	
	\begin{theorem}\label{regularity}
		For every $k\geq 1$:
		$$H^k_{W,V}(\mathbb{T}) \subset C_{W,V}^{k-1}(\mathbb{T}),$$
		where $C^0_{W,V}(\mathbb{T}) = \{f:\mathbb{T}\to\mathbb{R};\, \hbox{$f$ is càdlàg and $D(f)\subset D(V)$}\}$, where $D(f)$ stands for the set of discontinuity points of $f$.
	\end{theorem}
	\begin{proof}
		Begin by noticing that the case $k=1$ follows directly from \cite[Theorem 2]{keale}. By using induction on $k$ and \cite[Theorem 2]{keale}, again, we obtain the result for any $k\geq 1$.
	\end{proof}
An immediate application of Theorem \ref{regularity} yields the following Corollary regarding strong regularity of $H^\infty_{W,V}(\mathbb{T})$:
	\begin{corollary}\label{hinfinityreg}
		$$\bigcap_{k=1}^{\infty}H^{k}_{W,V}(\mathbb{T}) = : H^{\infty}_{W,V}(\mathbb{T}) = C^\infty_{W,V}(\mathbb{T}).$$
	\end{corollary}
	
	\section{Fractional order Sobolev spaces $H^s_{W,V}(\mathbb{T})$}\label{fracsect}
	
In this section we will consider a different approach for defining higher order Sobolev spaces in such a way that it will agree with the definition given in the previous section, but it will also be suitable to define fractional order Sobolev spaces. One key difference in the approach we will use is that we will only deal with the space $L^2_V(\mathbb{T})$, that is, we will not be switching between the spaces $L^2_V(\mathbb{T})$ and $L^2_W(\mathbb{T})$. 

We begin by characterizing the fractional order $W$-$V$-Sobolev spaces for non-negative orders.
	
	\begin{definition}\label{fracsob}
		Let $s\geq 0$, we define the $s$th order $W$-$V$-Sobolev space by
		$$H^s_{W,V}(\mathbb{T}) = D((I-\Delta_{W,V})^{s/2})= \left\{f \in L^2_V(\mathbb{T});\quad f = \alpha_0 + \sum_{i=1}^\infty \alpha_i \nu_i;\quad \sum_{i=0}^\infty \gamma_i^s \alpha_i^2<\infty\right\},$$
		endowed with the norm
		$$\|f\|_{H^s_{W,V}(\mathbb{T})}^2 = \|(I-\Delta_{W,V})^{s/2}f\|_{L^2_V(\mathbb{T})}^2 = \sum_{i=0}^\infty \gamma_i^s \alpha_i^2,$$
		where $(\gamma_i)_{i\in\mathbb{N}}$ are the eigenvalues of $I-\Delta_{W,V}$ and were introduced in the comments just below Definition \ref{laplacianWV2}.
	\end{definition}
	
	The next Proposition shows that the above definition agrees with the definitions given the previous section when $s\in\mathbb{N}:$
	
	\begin{proposition}
		If $s=k\in\mathbb{N}$, then, we have the following equality of sets:
		$$H^s_{W,V}(\mathbb{T}) = H^k_{W,V}(\mathbb{T}),$$
		where the left-hand side refers to the set in Definition \ref{fracsob} and the right-hand side refers to the set in Definition \ref{highersob}. Moreover, 
		the norms $\|\cdot\|_{H^k_{W,V}(\mathbb{T})}$ and $\|\cdot\|_{k,W,V}$ are equivalent.
	\end{proposition}
	\begin{proof}
		Notice that $\gamma_i = 1 + \lambda_i$. So, if $\sum_{i=1}^\infty \alpha_i^2<\infty$, we have that
		$$\sum_{i=0}^\infty \gamma_i^k\alpha_i^2<\infty \Longleftrightarrow \sum_{i=1}^\infty \lambda_i^k \alpha_i^2<\infty.$$
		Hence, by Theorem \ref{fouriersob}, we have the equality of sets. It remains to be proved that the norms $\|\cdot\|_{H^s_{W,V}(\mathbb{T})}$ and $\|\cdot\|_{k,W,V}$ are equivalent.
		
		To this end, it readily follows by repeated application of the integration-by-parts formula (see Proposition \ref{intbypartsprop2}) that
		for any function $f\in H^k_{W,V}(\mathbb{T})$, with
		$$f = \alpha_0 + \sum_{i=1}^\infty\alpha_i \nu_i,$$
		and any $0\leq j\leq k$,
		\begin{equation}\label{fouriernormhigher}
			\|D_{W,V}^{(j)}f\|_{\kappa(j)}^2 = \sum_{i=1}^\infty \lambda_i^j \alpha_i^2 \leq \sum_{i=0}^\infty \gamma_i^j \alpha_i^2 = \|f\|_{H^j_{W,V}(\mathbb{T})}^2,
		\end{equation}
		where $\kappa$ is given by equation \eqref{kappa_order}. Therefore, 
		$$\|f\|_{k,W,V}^2 \leq \|f\|_{H^0_{W,V}(\mathbb{T})}^2 + \cdots + \|f\|_{H^k_{W,V}(\mathbb{T})}^2\leq (k+1)\|f\|_{H^k_{W,V}(\mathbb{T})}^2,$$
		where the last inequality comes from the fact that $\gamma_j\geq 1$, so $\|f\|_{H^i_{W,V}(\mathbb{T})}^2\leq \|f\|_{H^j_{W,V}(\mathbb{T})}^2$ if $i\leq j$. Thus,
		$$\|f\|_{k,W,V}\leq \sqrt{k+1} \|f\|_{H^j_{W,V}(\mathbb{T})}.$$
		Conversely, note that for each $i,k\in\mathbb{N}$ we have $\gamma_i^k \leq 2^k(1 + \lambda_i^k)$. Therefore, 
		$$\|f\|_{H^k_{W,V}(\mathbb{T})}^2 \leq 2^k\left(\sum_{i=1}^\infty \alpha_i^2 + \sum_{i=1}^\infty \lambda_i^k \alpha_i^2\right) = 2^k(\|f\|_V^2 + \|\partial^{(k)}_{W,V}f\|_{\kappa(n)}^2) \leq 2^k\|f\|_{k,W,V}^2.$$
	\end{proof}
	
	\begin{remark}
		Note that for $f\in H^s_{W,V}(\mathbb{T})$ the norm $\|f\|_{H^s_{W,V}(\mathbb{T})} = \|(I-\Delta_{W,V})^{s/2}f\|_{L^2_{V}(\mathbb{T})}$ is defined only in terms of $L^2_V(\mathbb{T})$. However, this norm is not suitable for dealing with odd natural numbers. Indeed, if $s=2k-1,$ where $k\in\mathbb{N}$ is an odd natural number, then $\|f\|_{H^s_{W,V}(\mathbb{T})}$ can be seen as a ``non-local'' counterpart to the norm $\|f\|_{s,W,V}$, which only involves local operators, but ``pays the price'' of depending on the space $L^2_W(\mathbb{T})$.
	\end{remark}
	
	For $s\geq 0$, let $H^{-s}_{W,V}(\mathbb{T}) := (H^s_{W,V}(\mathbb{T}))^\ast$, that is, $H^{-s}_{W,V}(\mathbb{T})$ is the dual of $H^s_{W,V}(\mathbb{T})$. We have the following characterization for $H^{-s}_{W,V}(\mathbb{T})$:
	
	\begin{proposition}\label{dualsob}
		For $s\geq 0$, we have that
		$$H^{-s}_{W,V}(\mathbb{T}) \cong \left\{f = \sum_{i=1}^\infty \alpha_i\nu_i; \quad \sum_{i=1}^\infty \gamma_i^{-s}\alpha_i^2<\infty\right\},$$
		with norm
		\begin{equation}\label{formuladual}
			\|f\|_{H^{-s}_{W,V}(\mathbb{T})}^2 = \sum_{i=1}^\infty \gamma_i^{-s}\alpha_i^2
		\end{equation}
		and dual pairing
		$$(f,g) = \sum_{i=0}^\infty \alpha_i \langle \nu_i, g\rangle_V,$$
		with 
		$$\alpha_i = (f,\nu_i).$$
	\end{proposition}
	\begin{proof}
		Let $f\in H^{-s}_{W,V}(\mathbb{T})$. By Riesz's representation theorem, there exists $u\in H^s_{W,V}(\mathbb{T})$ such that for every $g\in  H^s_{W,V}(\mathbb{T})$, we have
		$$(f,g) = \langle u, g\rangle_{H^{s}_{W,V}(\mathbb{T})}.$$
		Now, since $u\in H^s_{W,V}(\mathbb{T})$, we can write 
		$$u = \sum_{i=1}^\infty \beta_i \nu_i,$$
		where
		$$\sum_{i=1}^\infty \gamma_i^s \beta_i^2 <\infty.$$
		Hence,
		$$(f,g) = \sum_{i=0}^\infty \gamma_i^s \beta_i \<\nu_i,g\>_V.$$
		So, if we define $\alpha_i = \gamma_i^s\beta_i$, we have that $(f,\nu_i) = \<u,\nu_i\>_{H^s_{W,V}(\mathbb{T})} = \gamma_i^s\beta_i$ and
		$$\sum_{i=0}^\infty \gamma_i^{-s}\alpha_i^2 = \sum_{i=0}^\infty \gamma_i^s\beta_i^2 <\infty.$$
		Finally, by Riesz's representation theorem, we have that
		$$\|f\|_{H^{-s}_{W,V}(\mathbb{T})}^2 = \|u\|_{H^s_{W,V}(\mathbb{T})}^2 = \sum_{i=1}^\infty \gamma_i^s \beta_i^2 = \sum_{i=1}^\infty \gamma_i^{-s}\alpha_i^2.$$
		This proves one direction. The other direction is simpler and we leave it for the reader.
	\end{proof}
	
	\section{Asymptotic behavior of the eigenvalues of $\Delta_{W,V}$ and the embedding $L_{V}^{2}(\mathbb{T}) \hookrightarrow H^{-s}_{W,V}(\mathbb{T})$}
	
	We begin this section by recalling the definition of our generalizations of the cosine and sine functions, namely $C_{W,V}(\alpha,\cdot)$ and $S_{W,V}(\alpha,\cdot)$, studied in Section \ref{sect42}. We begin by proving a fundamental relation that is analogous, and generalizes, the well-known relation $\sin^2(x) + \cos^2(x)=1$.
	
	\begin{proposition}\label{prop11}
		For any $x\in\mathbb{T}$ and any $\alpha\neq 0$, the following fundamental relation is true:
		\begin{equation}\label{relfun}
			C_{W,V}(\alpha,x)C_{V,W}(\alpha,x)+S_{W,V}(\alpha,x)S_{V,W}(\alpha,x)=1
		\end{equation}
	\end{proposition}
	\begin{proof}
		Let $(\alpha_n)_{n\in\mathbb{N}}$ be given by $$\alpha_{2k}(x)=(G_{2k}(x,x),F_{2k}(x,x))$$
		and 
		$$\alpha_{2k+1}(x)=(F_{2k+1}(x,x),G_{2k+1}(x,x))$$ for $k\geq 0$, where $F_n$ is given by \eqref{term:1} and $G_n$ is given by \eqref{Gnformula}. Let us denote $\alpha_{n}(x)=(p_{n}(x),q_{n}(x))$. By using the integration by parts formula (see Proposition \ref{intbypartsprop2}), we obtain the following relation:
		\begin{equation}\label{relarela}
			\sum_{j=0}^{2k}(-1)^{j}q_{j}(x)p_{2n-j}(x)=0.
		\end{equation}
		Finally, by using (\ref{relarela}) on 
		$$C_{W,V}(\alpha,x)C_{V,W}(\alpha,x)+S_{W,V}(\alpha,x)S_{V,W}(\alpha,x)=1+\sum_{n=1}^{\infty}(-1)^{n}\alpha^{2n}\sum_{k=0}^{2n}(-1)^{k}q_{k}(x)p_{2n-k}(x),$$
		the relation (\ref{relfun}) follows.
	\end{proof}
	\begin{remark}
		The strategy of the above proof and the relation (\ref{relfun}) were inspired by the ideas developed on \cite[Theorem 5.3]{trig}.
	\end{remark}

We will now obtain a sharp estimate related to $F_{2n}$ and $G_{2n}$, which will be fundamental in our study on the asymptotic behavior of the eigenvalues of $\Delta_{W,V}$.

	\begin{proposition}\label{prop22}
		There exists some constant $C>0$ such that 
		 \begin{equation}\label{kjrsas}
			|F_{2n}(1,1)+G_{2n}(1,1)|\le  \dfrac{C^{n}}{(n!)^{2}}.
		\end{equation}
	\end{proposition}
	\begin{proof}
		We begin by observing that the functions $\frac{W^{n+1}}{n+1}$ and $\frac{V^{n+1}}{n+1}$ are, respectively, $W$-absolutely continuous and $V$-absolutely continuous in the sense of \cite{pouso}. So, by using the derivatives in the sense of \cite{pouso}, we have that 
		$$\left(\frac{W^{n+1}}{n+1}\right)'_{W}(x)=\frac{1}{n+1}\sum_{j=0}^{n}W(x-)^{j}W(x)^{n-j}$$
		and 
		$$\left(\frac{V^{n+1}}{n+1}\right)'_{V}(x)=\frac{1}{n+1}\sum_{j=0}^{n}V(x+)^{j}V(x)^{n-j}.$$ 
		By using that $W(x)\geq W(x-)$ and $V(x+)\geq V(x)$ together with \cite[Theorem 5.4]{pouso} we obtain the following inequalities:
		\begin{equation}\label{estimaW}
			\frac{[W(x)]^{n+1}}{n+1}\geq\int_{(0,x]}W^{n}(\xi-)dW(\xi)
		\end{equation}
		and 
		\begin{equation}\label{estimaV}
			\frac{[V(x)]^{n+1}}{n+1}\geq\int_{[0,x)}V^{n}(\xi)dV(\xi).
		\end{equation}
 Finally, by using Remark \ref{rmk111}, and successively applying the inequalities (\ref{estimaW}) and (\ref{estimaV}) on $F_{2n}(x,x)$, we obtain that there exists a constant $A>0$ such that $|F_{2n}(1,1)|\le \frac{A^{n}}{(n!)^2}$. By proceeding similarly for $G_{2n}$ (with the corresponding expression for $G_{k}$ obtained as in Remark \ref{rmk111}) and successively applying (\ref{estimaW}) and (\ref{estimaV}) on $G_{2n}(x,x)$, we  obtain that there exists $B>0$ such that $|G_{2n}(1,1)|\le \frac{B^{n}}{(n!)^2}$. Therefore, inequality (\ref{kjrsas}) follows from choosing $C=2(A+B)$.
	\end{proof}

We are now in a position to obtain the main result of this section, which provides a lower bound on the growth of the eigenvalues of $\Delta_{W,V}$.

	\begin{theorem}\label{prop6}
		Let $\{\lambda_{i}\}_{i\geq1}$ be the sequence of non-negative eigenvalues (counting according their multiplicity) of $\Delta_{W,V}$. Then, there exists $\rho\in (0,\frac{1}{2}]$ such that \begin{equation}\label{eigenineqeq}
			Cn^{1/\rho}\le \lambda_{n}
		\end{equation} for some constant $C>0$. Moreover, 
	\begin{equation}\label{convsumeigen}
		\sum_{i=1}^{\infty}\dfrac{1}{\lambda_{i}^s}<\infty.
	\end{equation}
		for $s>\rho$. In particular, if $s>\rho$, then $\mathcal{A}^{-s}=\left(I-\Delta_{W,V}\right)^{-s}$ is a trace-class operator.
	\end{theorem}
	\begin{proof}
		Begin by noting that Proposition \ref{prop11} implies that the system (\ref{bcondi}) has a non trivial solution if $\lambda_{i}>0$ satisfies
		$$2=C_{W,V}(\sqrt{\lambda_{i}},1)+C_{V,W}(\sqrt{\lambda_{i}},1).$$
		That is, if $\lambda_{i}$
		are positive roots of the entire function
		$$f(z)=-2+\sum_{n\geq 0}(-1)^{n}z^{n}(F_{2n}(1,1)+G_{2n}(1,1))=\sum_{n\geq 1}(-1)^{n}z^{n}(F_{2n}(1,1)+G_{2n}(1,1)).$$
		Therefore, the eigenvalues of $\Delta_{W,V}$ are zeroes of $f(z)$.
		Now, the inequality $$|F_{2n}(1,1)+G_{2n}(1,1)|\le\dfrac{C^{n}}{(n!)^2}$$
		implies that the order of growth $\rho$ of the series $f(z)$ satisfies $0\le\rho\le \dfrac{1}{2}$. Indeed, $$\rho=\limsup_{n\to\infty}\dfrac{n\ln n}{-\ln \left|(-1)^{n}[F_{2n}(1,1)+G_{2n}(1,1)]\right|}\le \limsup_{n\to\infty}\dfrac{n\ln n}{-\ln\left[\dfrac{C^{n}}{(n!)^2}\right]}=\dfrac{1}{2}.$$ 
		Now, denote by $\mathcal{Z}=\{z_{i}\}_{i\in\mathbb{N}}$ the zeroes of $f$ indexed according their respective multiplicity and in such a way that they are ordered according their moduli:
		$$0<|z_{1}|\le |z_{2}| \le \ldots .$$
		Let $n(r)$ be number of zeroes of the function $f$ whose moduli are less or equal to $r$, that is,
		 $$n(r)=\#\{i\in\mathbb{N}; |z_{i}|<r\}.$$
		 Then, we have by \cite[Chapter 5, Theorem 2.1]{stein} that for any $\beta\ge\rho$ there exists a constant $C>0$ such that for sufficient large $r>0$,  
		$$n(r)\le Cr^{\beta}.$$
		Now, observe that $\{\lambda_{i}\}_{i\in\mathbb{R}}\subset \mathcal{Z}$, and that by \eqref{asymplambda}, we have $\lambda_{i}\to \infty$ as $i\to\infty$. Therefore, $\lim_{i\to\infty}|z_{i}|=+\infty$ which implies that $\rho\neq 0$. The last results prove, in particular, that $\rho\not\in\mathbb{Z}$.
		Now, by \cite[Chapter 5, Theorem 1]{levin}, the order $\rho$ is equal to the convergence exponent of $\mathcal{Z}$. In particular, there exists $C>0$ such that
		\begin{equation}\label{estim11}
			\limsup_{r\to+\infty}\dfrac{n(r)}{r^{\rho}}=\limsup_{n\to+\infty}\dfrac{n}{|z_{n}|^{\rho}}\le C.
		\end{equation}
		From (\ref{estim11}), there exists a subsequence $n_{k}$, such that
		\begin{equation}\label{lowerboundeq}
			\dfrac{n_{k}}{C}\le \lambda_{k}^{\rho}.
		\end{equation}
	 Now, by the definition of subsequence, we have that $n_{k}\geq k$. By combining \eqref{lowerboundeq} and $n_k\geq k$, we obtain the inequality (\ref{eigenineqeq}). Finally, by using that $\rho$ is equal to the convergence exponent of $\mathcal{Z}$, we directly obtain that if  $s>\rho$, then
		\begin{equation}\label{convseries}
			\sum_{i\geq 1}\dfrac{1}{\lambda_{i}^{s}}<\infty.
		\end{equation}
	\end{proof}

\begin{remark}
	From Weyl's asymptotics (see, for instance,\cite[Theorem 6.3.1]{davies} ), we have that the order in the lower bound \eqref{eigenineqeq} is sharp. Indeed, Weyl's asymptotics yields that the exact order for the Laplacian, which corresponds to the case in which $W(x) = V(x) = x$ is $\lambda_n \propto n^{2}.$
\end{remark}

	An immediate, and very important, consequence is that Theorem \ref{prop6} provides us with a condition on the order $s$ for the inclusion $i:L^2_V(\mathbb{T})\to H^{-s}_{W,V}(\mathbb{T})$, $i(f) = f$ to be Hilbert-Schmidt:
	\begin{proposition}\label{HSinclusion}
		Let $\rho\in (0,1/2]$ be as in Theorem \ref{prop6}. For any $s>\rho$, the inclusion $i:L^2_V(\mathbb{T})\to H^{-s}_{W,V}(\mathbb{T})$ is trace-class. In particular, for any $s>1/2$, the inclusion $i:L^2_V(\mathbb{T})\to H^{-s}_{W,V}(\mathbb{T})$ is trace-class.
	\end{proposition}
	\begin{proof}
		Consider the orthonormal basis of $L^2_V(\mathbb{T})$ given by the eigenvectors of $\mathcal{A} = I - \Delta_{W,V}$, namely, $\{\nu_i\}_{i\in\mathbb{N}}$. Therefore, by equation \eqref{formuladual}, we have that
		$$\sum_{i=1}^\infty (i(\nu_{i}),\nu_{i})_{H^{-s}_{W,V}(\mathbb{T})}=\sum_{i=1}^\infty \|\nu_i\|_{H^{-s}_{W,V}(\mathbb{T})}^2 = \sum_{i=1}^\infty \gamma_i^{-s},$$
		which converges for $s> \rho$ in view of Theorem \ref{prop6} and the inequality $\gamma_{i}\geq\lambda_{i}$ for all $i\geq 0$. Since $\rho\leq 1/2$, we have that $s>1/2$ will always satisfy $s>\rho$.
	\end{proof}
	
	\section{Extension to the $d$-dimensional case}
	
We will now extend our operator $\Delta_{W,V}$ to the $d$-dimensional torus $\mathbb{T}^d$. The idea is to take tensor products of copies of our one-dimensional operator. 

To such an end, for each $k=1,\ldots,d$, let $W_{k}:\mathbb{T}\to\mathbb{R}$ and $V_{k}:\mathbb{T}\to\mathbb{R}$, be strictly increasing functions, where $W_{k}$ and $V_{k}$ are respectively càdlàg and càglàd functions satisfying (\ref{eqsmedida}) and such that $W_{k}(0)=W_{k}(0-)=V_{k}(0+)=V_{k}(0)=0$.

Now, observe that for $k=1,\ldots,d$, the functions $W_{k}$ and $V_{k}$ induce Borel measures $dW_{k}$ and $dV_{k}$ on $\mathbb{T}$. Now, for each $k=1,\ldots,d$, consider the Borel measures $dW$, $dV$, $dW^k\otimes dV_{k}$ and $dV^k\otimes dW_{k}$ on $\mathbb{T}^{d}$, obtained as suitable product measures. More precisely, they are the unique measures satisfying the following relations:
	$$dW\left((0,x_{1}]\times\ldots\times(0,x_{d}]\right)=\prod_{i=1}^{d}W_{i}(x_{i}), \quad dV\left([0,x_{1})\times\ldots\times[0,x_{d})\right)=\prod_{i=1}^{d}V_{i}(x_{i}),$$
$$dW^k\otimes dV_{k}\left(\prod_{i=1}^{k-1}(0,x_{i}]\times[0,x_{k})\times\prod_{j=k+1}^{d}(0,x_{j}]\right)=V(x_{k})\prod_{i=1, i\neq k}^{d}W_{i}(x_{i}).$$
	and
	$$dV^k\otimes dW_{k}\left(\prod_{i=1}^{k-1}[0,x_{i})\times(0,x_{k}]\times\prod_{j=k+1}^{d}[0,x_{j})\right)=W(x_{k})\prod_{i=1, i\neq k}^{d}V_{i}(x_{i}).$$
Hence, we can write them, for $k=1,\ldots,d$, as

\begin{equation*}
	dW=\bigotimes_{i=1}^{d}dW_{i};\,\,\,\,\,\,\,dV=\bigotimes_{i=1}^{d}dV_{i};
	\end{equation*}
	\begin{equation*}
	dV^k\otimes dW_{k}=dW_{k}\otimes\bigotimes_{i=1, i\neq k}^{d}dV_{i};\,\,\,\,\,\,\,dW^k\otimes dV_{k}=dV_{k}\otimes\bigotimes_{i=1, i\neq k}^{d}dW_{i}.
	\end{equation*}
	
	Now denote by $\mathcal{A}_{W_{k},V_{k}}\subset C^{\infty}_{W_{k},V_{k}}(\mathbb{T})$ the collection of all eigenvectors of $-\nabla_{W_{k},V_{k}}$ and define the set
	$$\mathcal{A}_{W,V}=\left\{\bigotimes_{i=1}^{d}f_{i}:\mathbb{T}^d\to\mathbb{R},\, \bigotimes_{i=1}^{d}f(x_{1},\ldots,x_{d}):=\prod_{i=1}^{d}f_{i}(x_{i});\, f_{i}\in \mathcal{A}_{W_{k},V_{k}}\right\}.$$
	
	Now, let us define the lateral partial derivatives. More precisely, given $f:\mathbb{T}^{d}\to\mathbb{R}$, we say that $f$ has the partial $W_{k}$-left-derivative at $(x_{1},\ldots,x_{d})\in\mathbb{T}^{d}$ if the following limit exists
	\begin{align*}
		\partial^{-}_{W_{k}}f(x_{1},\ldots,x_{d}):=\lim_{h\to 0^{+}}\dfrac{f(x_{1},\ldots,x_{d})-f(x_{1},\ldots,x_{k-1},x_{k}-h,x_{k+1},\ldots,x_{d})}{W_{k}(x_{k})-W_{k}(x_{k}-h)}.
	\end{align*}
If $f$ has the partial $W_{k}$-left-derivative at every point of $\mathbb{T}^d$, we say that $f$ is $W_{k}$-left-differentiable and we can define the function $\partial^{-}_{W_{k}}f:\mathbb{T}^d\to \mathbb{R}$. In such a case, this function is called the partial $W_{k}$-left-derivative of $f$.  We define the partial $V_{k}$-right-derivative and $V_{k}$-right-differentiable functions in an analogous manner.

Now, we define $C^{\infty}_{W,V,0}(\mathbb{T}^{d})$ as the space of functions
  $f:\mathbb{T}^{d}\to\mathbb{R}$ such that for all $(x_{1},\ldots,x_{k-1},x_{k+1},\ldots,x_{d})\in \mathbb{T}^{d-1}$ and $k=1,\ldots,d,$ the application $$f^{(k)}(x):=f(x_{1},\ldots,x_{k-1},x,x_{k+1},\ldots,x_{d})$$ is càdlàg.
  
  Observe that for all $n\geq 1$, the application $\partial_{W_{k},V_{k}}^{(n)}f(x_{1},\ldots,x_{k-1},x,x_{k+1},\ldots,x_{d}):=D^{(n)}_{W_{k},V_{k}}f^{(k)}(x)$ on $\mathbb{T}$  is well defined and is càglàd (resp. càdlàg) for $n$ odd (resp. $n$ even), where $D_{W_k,V_k}^{(n)}$ is defined as in \eqref{nderiv}.
  Now, let $n\geq 0$, and observe that $\partial^{(n)}_{W_{k},V_{k}}f \in L^{2}_{dV^k\otimes dW_{k},0}(\mathbb{T}^{d}) \left(\mbox{resp.}\,L^{2}_{dV,0} \mathbb{T}^{d})\right)$ if $n$ is odd (resp. $n$ is even). Clearly $$\mathcal{A}_{W,V}\subset C^{\infty}_{W,V}(\mathbb{T}^{d}):=\<1\>\oplus C^{\infty}_{W,V}(\mathbb{T}^d).$$Therefore, we are able to define the operator $\mathcal{L}_{W,V}:C^{\infty}_{W,V}(\mathbb{T}^{d})\subset L^{2}_{dV}(\mathbb{T}^{d})\to L^{2}_{dV}(\mathbb{T}^{d})$ given by
	$$\mathcal{L}_{W,V}f = \sum_{i=1}^{d}\partial^{(2)}_{W_{k},V_{k}}f.$$
	Indeed, the operator $I-\mathcal{L}_{W,V}$ is symmetric, positive and monotone. Furthermore, note that $I-\mathcal{L}_{W,V}$ is densely defined. Indeed, we have that
	$$\bigotimes_{i=1}^d C^{\infty}_{W_i,V_i}(\mathbb{T})\subset C^{\infty}_{W,V}(\mathbb{T}^{d}),$$
	where
	$$\bigotimes_{i=1}^d C^{\infty}_{W_i,V_i}(\mathbb{T}) = \textrm{span}\left\{f_1\otimes\cdots\otimes f_d; f_i\in C^{\infty}_{W_i,V_i}(\mathbb{T}),i=1,\ldots,d  \right\}.$$ 
	 Therefore, by Proposition \ref{densenesscinfinity2}, we have that for each $i=1,\ldots,d$, $C^{\infty}_{W_i,V_i}(\mathbb{T})$ is dense in $L^2_{V_i}(\mathbb{T})$. Now, by standard limiting arguments, it is enough to show that we can approximate $\mathbf{1}_R$, where $R$ is a measurable rectangle on the product $\sigma$-algebra $\mathcal{B}(\mathbb{T}^d) = \mathcal{B}(\mathbb{T})\times\cdots\times \mathcal{B}(\mathbb{T})$. Now, by the definition of measurable rectangles, we have that there exist $R_i\in \mathcal{B}(\mathbb{T}), i=1,\ldots,d$, such that $R = R_1\times \cdots \times R_d$, hence $\mathbf{1}_R = \mathbf{1}_{R_1}\otimes\cdots\otimes\mathbf{1}_{R_d}$. Then, by the density of $C^{\infty}_{W_i,V_i}(\mathbb{T})$ in $L^2_{V_i}(\mathbb{T})$ and Fubini's theorem, we have that $\bigotimes_{i=1}^d C^{\infty}_{W_i,V_i}(\mathbb{T})$ is dense in $L^2_{dV}(\mathbb{T}^d)$. Therefore, we obtain that $C^{\infty}_{W,V}(\mathbb{T}^{d})$ is dense in $L^2_{dV}(\mathbb{T}^d)$.
	
	 Therefore, $I-\mathcal{L}_{W,V}$ admits a Friedrich's extension  (we refer the reader to Zeidler \cite[Section 5.5]{zeidler} for further details on Friedrichs extensions).  Thus, let $I-\Delta_{W,V}:\mathcal{D}_{W,V}(\mathbb{T}^d)\subset L^2_{dV}(\mathbb{T}^d)\to L^2_{dV}(\mathbb{T}^d)$ be the Friedrich's extension of the operator $I-\mathcal{L}_{W,V}$. Note that $C^{\infty}_{W,V}(\mathbb{T}^{d})\subset \mathcal{D}_{W,V}(\mathbb{T}^d)$.
	 
	 We then, proceed as in \cite{keale} to define the $d$-dimensional $W$-$V$-Sobolev space:
	 \begin{definition}
		The $W$-$V$-Sobolev space of order 1, $H^{1}_{W,V}(\mathbb{T}^{d})$, is defined as the energetic space (in the sense of Zeidler, \cite[Section 5.3]{zeidler}) associated to the operator $I-\mathcal{L}_{W,V}: C^{\infty}_{W,V}(\mathbb{T}^{d})\subset L^2_{dV}(\mathbb{T}^d)\to L^2_{dV}(\mathbb{T}^d)$. That is, we define the norm $$\|f\|_{1,2}^2 = \langle (I - \mathcal{L}_{W,V})f,f\rangle_{V}$$
		in $C^{\infty}_{W,V}(\mathbb{T}^{d})$ and say that $f\in L^2_{V}(\mathbb{T}^d)$ belongs to $H_{W,V}^1(\mathbb{T}^d)$ if, and only if, the following conditions hold:
		\begin{enumerate}
			\item There exists a sequence $f_n\in C^{\infty}_{W,V}(\mathbb{T}^{d})$ such that $f_n\to f$ in $L^2_V(\mathbb{T}^d)$;
			\item The sequence $f_n$ is Cauchy with respect to the energetic norm $\|\cdot\|_{1,2}$.
		\end{enumerate}
		A sequence $(f_n)_{n\in\mathbb{N}}$ in $C^{\infty}_{W,V}(\mathbb{T}^{d})$ satisfying 1 and 2 is called an \emph{admissible sequence}. 
	 \end{definition}
 
 	We can also notice that Theorem \ref{fouriersob} implies that the domain of the Friedrich's extension of $I - D_V^+D_W^-$ coincides with the $W$-$V$-Sobolev of order 2. Therefore, we can readily define the $d$-dimensional $W$-$V$-Sobolev space of order 2 as the domain of the Friedrich's extension:
 	
 	\begin{definition}
 		The $W$-$V$-Sobolev space of order 2, $H^{2}_{W,V}(\mathbb{T}^{d})$, is defined as the domain of the Friedrich's extension of the operator $I-\mathcal{L}_{W,V}: C^{\infty}_{W,V}(\mathbb{T}^{d})\subset L^2_{dV}(\mathbb{T}^d)\to L^2_{dV}(\mathbb{T}^d)$. Thus, $H^{2}_{W,V}(\mathbb{T}^{d}) = \mathcal{D}_{W,V}(\mathbb{T}^d)$. 
 	\end{definition}
 
	  Let us now show that the above definitions coincide with a corresponding definition of $W$-$V$-Sobolev spaces in terms of weak derivatives. To this end, consider the following definition of weak derivative.
	
	\begin{definition}
		We say that $f\in L^{2}_{dV}(\mathbb{T}^d)$ has a weak partial $W_{k}$-left-derivative if there exists $F\in L^{2}_{dV\otimes dW_{k}}(\mathbb{T}^{d})$ satisfying 
		$$\int_{\mathbb{T}^d}f\partial^{(1)}_{V_{k},W_{k}}g\, dV=-\int_{\mathbb{T}^d}F_{k}g\,dV^k\otimes dW_{k}$$
		$\forall g\in C^{\infty}_{V,W}(\mathbb{T}^{d}).$ In this case F is unique and we use the notation $\tilde{\partial}^{(1)}_{W_{k},V_{k}}f:=F$
	\end{definition}
	
	\begin{remark}
	    Observe that our test functions, that is the functions that belong to $C^{\infty}_{V,W}(\mathbb{T}^{d})$ not only are smooth in the sense that the operator can be applied indefinitely, they are also smooth in the sense that they can be point-wisely evaluated. To the best of our knowledge, this is the first time a space of test functions in a $d$-dimensional setup has both of these features. For instance, in \cite{wsimas, farfansimasvalentim, valentim, simasvalentimjde} all the test functions either were able to be point-wisely evaluated or were such that the operator could be applied indefinitely, but in none of the cases, it was shown that one could do both.
	\end{remark}
	
	The definition of weak partial $V_{k}$-right-derivative can be defined analogously. The next result characterizes the energetic space $H^{1}_{W,V}(\mathbb{T}^d)$ in terms of weak derivatives.
	\begin{theorem}
		If we define the space 
		$$\tilde{H}^{1}_{W,V}(\mathbb{T}^d)=\left\{f\in L^{2}_{dV}(\mathbb{T}^{d}); \exists\tilde{\partial}^{(1)}_{W_{k},V_{k}}f\in L^{2}_{dV\otimes dW_{k}}(\mathbb{T}^{d})\right\}.$$
		
		Then $H^{1}_{W,V}(\mathbb{T}^d)=\tilde{H}^{1}_{W,V}(\mathbb{T}^d).$ In particular, the space $\tilde{H}^{1}_{W,V}(\mathbb{T}^d)$ is a Hilbert Space with the energetic inner product $$\<f,g\>_{W,V}=\<f,g\>_{L_{dV}^{2}(\mathbb{T}^d)}+\sum_{i=1}^{d}\left\<\tilde{\partial}^{(1)}_{W_{i},V_{i}}f,\tilde{\partial}^{(1)}_{W_{i},V_{i}}g\right\>_{L_{dV\otimes dW_{i}}^{2}(\mathbb{T}^d)}.$$
	\end{theorem}
	
	The proof of the above theorem follows from a simple adaptation of \cite[Theorem 5]{keale} together with the fact that the eigenvectors of the $d$-dimensional $\Delta_{W,V}$ are tensor products of the eigenvectors of $\Delta_{W_k,V_k},k=1,\ldots,d$.
	
	Indeed , the set $\mathcal{A}_{W,V}$ consists precisely on the orthonormal basis of $L^{2}_{dV}(\mathbb{T}^d)$ which are eigenfunction of operator $I-\Delta_{W,V}$. Furthermore, we can index the set of eigenvalues $\{\alpha_{n}\}_{n\in\mathbb{N}}$ satisfying $$0<\alpha_{1}\le\alpha_{2}\le ... \to +\infty,$$
	as $n\to +\infty$, where the enumeration takes into account the multiplicities of each eigenvalue.
	
	 Now, note that by the definition of the space $H^{2}_{W,V}(\mathbb{T}^d)$ and by the density of $C^{\infty}_{W,V}(\mathbb{T}^d)$ on $H^{1}_{W,V}(\mathbb{T}^d)$ with respect to the $\|\cdot\|_{W,V}$ we have
	$$H^{2}_{W,V}(\mathbb{T}^d)=\left\{f\in H^{1}_{W,V}(\mathbb{T}^d); \exists \tilde{\partial}^{(2)}_{W_{k},V_{k}}f:=\tilde{\partial}^{(1)}_{W_{k},V_{k}}\left(\tilde{\partial}^{(1)}_{W_{k},V_{k}}f\right)\in L^{2}_{dV}(\mathbb{T}^d)\right\}$$
	which is a Hilbert Space with respect to the inner product defined by
	$$\<f,g\>_{2,W,V}=\<f,g\>_{W,V}+ \sum_{i=1}^{d}\left\<\tilde{\partial}^{(1)}_{W_{i},V_{i}}f,\tilde{\partial}^{(1)}_{W_{i},V_{i}}g\right\>_{L_{dV\otimes dW_{i}}^{2}(\mathbb{T}^d)}  + \sum_{i=1}^{d}\left\<\tilde{\partial}^{(2)}_{W_{i},V_{i}}f,\tilde{\partial}^{(2)}_{W_{i},V_{i}}g\right\>_{L^{2}_{dV}(\mathbb{T}^d)}.$$
	
	We can also define the norms 
		\begin{equation}\label{normoddd}
			\<f,g\>_{n,W,V}=\<f,g\>_{n-1,W,V}+ \sum_{i=1}^{d}\left\<\tilde{\partial}^{(n)}_{W_{i},V_{i}}f,\tilde{\partial}^{(n)}_{W_{i},V_{i}}g\right\>_{L_{dV\otimes dW_{i}}^{2}(\mathbb{T}^d)},
		\end{equation}
		if $n$ is odd, and 
		\begin{equation}\label{normevend}
			\<f,g\>_{n,W,V}=\<f,g\>_{n-1,W,V}+ \sum_{i=1}^{d}\left\<\tilde{\partial}^{(n)}_{W_{i},V_{i}}f,\tilde{\partial}^{(n)}_{W_{i},V_{i}}g\right\>_{L^{2}_{dV}(\mathbb{T}^d)}.
			\end{equation}
				if $n$ is even. These norms allow us to define the $n$th order $W$-$V$-Sobolev in a similar manner, in terms of integration by parts.
	
	 Now, similarly to the one dimensional case, by repeating the arguments in \cite[Theorem 4.1]{simasvalentim2}, we can characterize the spaces $H^{1}_{W,V}(\mathbb{T}^d)$ and $H^{2}_{W,V}(\mathbb{T}^d)$ as follows:
	$$H^{1}_{W,V}(\mathbb{T}^d)=\left\{\sum_{i=1}^{\infty}b_{i}\nu_{i}\in L^{2}_{dV}(\mathbb{T}^d); \sum_{i=1}^{\infty}\alpha_{i}b_{i}^2<\infty\right\},$$$$H^{2}_{W,V}(\mathbb{T}^d)=\left\{\sum_{i=1}^{\infty}b_{i}\nu_{i}\in L^{2}_{dV}(\mathbb{T}^d); \sum_{i=1}^{\infty}\alpha_{i}^{2}b_{i}^2<\infty\right\},$$
	and more generally,
	$$H^{n}_{W,V}(\mathbb{T}^d)=\left\{\sum_{i=1}^{\infty}b_{i}\nu_{i}\in L^{2}_{dV}(\mathbb{T}^d); \sum_{i=1}^{\infty}\alpha_{i}^{n}b_{i}^2<\infty\right\},$$
	where $\nu_{i}$ is the satisfies $(I-\nabla_{W,V})\nu_{i}=\alpha_{i}\nu_{i}.$ From the above characterization and following the analogue theory developed on above sections allow us to define for $s>0$ the the $d$-dimensional fractional measure theoretic $W$-$V$-Sobolev space as the set $$H^{s}_{W,V}(\mathbb{T}^d)=\mathscr{D}(I-\nabla_{W,V})^{s/2}=\left\{\sum_{i=1}^{\infty}b_{i}\nu_{i}\in L^{2}_{dV}(\mathbb{T}^d); \sum_{i=1}^{\infty}\alpha_{i}^{s}b_{i}^2<\infty\right\}.$$
	
	In what follows is important notice that if $\alpha_{n}$ is the eigenvalue of $I-\Delta_{W,V}$ associated to the eigenvector $f_{n}$, then, there exist $\gamma_{n_1,1},\ldots,\gamma_{n_k,k}$, where each $\gamma_{n,k}$ is the eigenvalue of $I-\Delta_{W_{k},V_{k}}$ associated to the eigenvector $f_{n,k}$, and they are such that $f_{n}$ is the tensor product of the $f_{n,k}$, that is, $f_n=\bigotimes_{i=1}^{d}f_{n_{k}}$ and $\alpha_n$ is the sum of the eigenvalues $\gamma_{n_k,k}$, that is, $\alpha_{n} = \sum_{k=1}^d\gamma_{n_k,k}.$ 
	
Our goal now is to study the asymptotic behavior of the eigenvalues of the $d$-dimensional $\Delta_{W,V}$. More precisely, we want to investigate for which values of $s$ we have the convergence of the series
	$$\sum_{k=1}^{\infty}\dfrac{1}{\alpha_{k}^{s}}.$$ 
	The convergence of the above series will determine under which conditions inclusion of $L^{2}_{V}(\mathbb{T})$ in $H^{-s}_{W,V}(\mathbb{T})$ is trace-class. 
	
	By the relation (\ref{eigenineqeq}) (note that to simplify the result, we are not using the convergence order $\rho$ given in the statement of Theorem \ref{prop6}, instead we are using the fact that $\rho\leq 1/2$) we know that there exists a constant $C>0$ such that
	$$n_{k}^{2}C\le \gamma_{n_{k}}.$$
	Therefore, the following inequality is true in $\mathbb{R}\cup\{\infty\}:$
	$$\sum_{k=1}^{\infty}\dfrac{1}{\alpha_{k}^{s}}\le \sum_{k\in \mathbb{N}^{d}}\dfrac{1}{\|k\|^{2s}_{2}},$$
	where $$\|k\|_{2}=\left(\sum_{j=1}^{d}k_{j}^{2}\right)^{\frac{1}{2}},$$
	 for $k=(k_{1},\ldots, k_{d})$. We, now, have that for $s>d/2$,
	  $$\sum_{k\in \mathbb{N}^{d}}\dfrac{1}{\|k\|^{2s}_{2}}<\infty.$$
	Indeed, for each $j\in\mathbb{N}$ the number of solutions on $\mathbb{N}^d$ for the equation $\sum_{k=1}^{d}n_{k}^2 = j^2$ is $O(j^{d-1})$, so
	$$\sum_{k\in \mathbb{N}^{d}}\dfrac{1}{\|k\|^{2s}_{2}}\le C\sum_{j=1}^{\infty}\dfrac{j^{d-1}}{j^{2s}}$$
	and the series on the right-hand side converges for $s>d/2.$ This last discussion allows us to state the following theorem
	\begin{theorem}\label{traceclassd}
		For $s>d/2$, the operator $(I-\nabla_{W,V})^{-s}$ is trace-class. 
	\end{theorem}
Which directly implies that:
\begin{corollary}\label{hilbertschmidtd}
	For any $s>d/2$, the inclusion $i:L^2_{dV}(\mathbb{T}^d)\to H^{-s}_{W,V}(\mathbb{T}^d)$ is trace-class.
\end{corollary}

	\begin{remark}
		In particular case where $W(x)=V(x)=x$ we recover the classical results for Laplacian on $\mathbb{T}^{d}$. Furthermore, since they are exact for the Laplacian, our general lower bound is sharp for all finite measures $W$ and $V$ defined on $\mathbb{T}^d.$
	\end{remark}

	\section{Applications}
	
	\subsection{The operator $L^{s}_{W,V}$ and the $\dot{H}_{L,W,V}$ spaces}
	
Let $H:\mathbb{T}^d \to \mathbb{R}^d$, be given by $H(x) = (H_1(x),\ldots,H_d(x))$, where for each $k=1,\ldots,d$, we have $H_k\in C^{\infty}_{V,W}(\mathbb{T}^d)$. Further, assume that $H$ is bounded away from zero, that is, there exists $H_0>0$ such that for every $x\in \mathbb{T}$ and every $k$, $H_k(x)\geq H_0$. Let, also, $\kappa: \mathbb{T}^d\to \mathbb{R}$ be a function in $C^{\infty}_{W,V}(\mathbb{T}^d)$ which is bounded away from zero.

Through this section we will consider the operator $L_{W,V}$ associated to the bilinear form
\begin{equation}\label{bilinearLVW}
B_{L_{W,V}}[u,v] = \int_{\mathbb{T}^d} \kappa^2 uv\,dV + \sum_{k=1}^d \int H_{k}\left( \partial^{(1)}_{V_{k},W_{k}}u\right)\left(\partial^{(1)}_{V_{k},W_{k}}v\right)\,dV^k\otimes dW_{k}.
\end{equation}

By the assumptions, the operator $L_{W,V}$ defined above is symmetric, positive and monotone. Furthermore,
by a simple calculation, we have that for $f\in C^{\infty}_{W,V}(\mathbb{T}^d)$ and every $k$, the following derivative of the product holds:
\begin{equation}\label{prodrule}
\partial_{V_k}^+\left( H_k \partial_{W_k}^- f\right)(x) = \left(\partial_{V_k}^+H_k(t)\right)\left( \partial_{W_k}^-f(t)\right) + H_k(t+) \partial_{V_k}^+\partial_{W_k}^-f(t).
\end{equation}

Therefore, equation \eqref{prodrule} implies that $L_{W,V}$ is densely defined, as it is well-defined in $C^{\infty}_{W,V}(\mathbb{T}^d)$. Hence, we can use Friedrich's extension (see \cite[Section 5.5]{zeidler}) on $L_{W,V}$ to obtain a self-adjoint extension, which we will also denote by $L_{W,V}$.
	
Now, observe that the assumptions on $H$ and $\kappa$ imply that there exist constants $C_{1}>0$ and $C_{2}>0$ depending only on $H$ and $\kappa$. Such that $\forall f \in H_{W,V}^1(\mathbb{T}^d)$ the following inequalities are verified
	
	\begin{equation}\label{courant}
		C_{1}\langle I-\Delta_{W,V}f,f\rangle_{V} \le B_{L_{W,V}}[f,f] \le C_{2}\langle I-\Delta_{W,V}f,f\rangle_V.
	\end{equation}
	Applying the Courant min-max principle on (\ref{courant}) we have the following inequalities related to the eigenvalues $\gamma_{i}$ of $I-\Delta_{W,V}$ and the eigenvalues $\gamma_{i,L}$ of $L_{W,V}$ :
	\begin{equation}\label{ineqauto}
		C_{1}\gamma_{i}\le \gamma_{i,L}\le C_{2}\gamma_{i}.
	\end{equation}
	So, the asymptotic behavior of $\{\gamma_{i}\}_{i\geq 0}$ and  $\{\gamma_{i}\}_{i\geq 0}$ are the same.
	\paragraph{}Denote by $\lambda_{i,L}$ the $i$th eigenvalue of $L_{W,V}$ with corresponding eigenvector given by $e_{i,L}.$ Let now, for $s\geq 0$ the Hilbert space
	$$\dot{H}_{L,W,V}^s(\mathbb{T}^{d}) := \mathscr{D}(L_{W,V}^{s/2}) = \left\{v\in L_{dV}^2(\mathbb{T}^d):  \sum_{j\in\mathbb{N}}  \gamma_{j,L}^{s} \<v, e_{j,L}\>_{{V}}^2<\infty\right\},$$
	with inner product and norm given by
	$$\<u,v\>_{\dot{H}_{L,W,V}^s} := \<L_{W,V}^{s/2}u, L_{W,V}^{s/2}v\>_{{V}} = \sum_{j\in\mathbb{N}}{\gamma_{j,L}^s} \<u, e_{j,L}\>_{V}\<v, e_{j,L}\>_{V}$$
	and
	$$\|v\|_{\dot{H}_{L,W,V}^s}^2 = \<v, v\>_{\dot{H}_{L,W,V}^s}.$$
	
The following Theorem is a consequence of (\ref{ineqauto}) and Theorem \ref{traceclassd}.
	
	\begin{theorem}\label{coro4}
		For any $s>d/2$, the operator $L^{-s}_{W,V}:L^{2}_{W,V}(\mathbb{T})\to H^{2s}_{W,V}(\mathbb{T})\subset L^{2}_{W,V}(\mathbb{T})$ is trace class.
	\end{theorem}

Now, let us obtain a result on elliptic regularity of $L_{W,V}$:

	\begin{proposition}\label{prop99}
		Fix $s\geq0$ and $m>0$. If $f\in \dot{H}^m_{L,W,V}(\mathbb{T}^d)$, then the function $u\in\dot{H}_{L,W,V}^s(\mathbb{T}^d)$ given by the unique weak solution of $L^s_{W,V}u=f$ satisfies $u\in \dot{H}_{L,W,V}^{m+2s}(\mathbb{T}^d)$
	\end{proposition}
	\begin{proof}
		Indeed, if 
		$$f=\sum_{i\in\mathbb{N}}\beta_{i}e_{i,L}$$
		and
		$$u=\sum_{i\in\mathbb{N}}\alpha_{i}e_{i,L}$$
		then by definition
		$$\beta_{i}=\<L^{s/2}_{W,V}u,L^{s/2}_{W,V}e_{i,L}\>_V=\gamma_{i,L}^{s}\<e_{i,L},u\>_V=\gamma_{i,L}^{s}\alpha_{i}.$$
		By, using that $f\in\dot{H}^m_{L,W,V}(\mathbb{T}^d)$ and the above identity
		$$\sum_{i\in\mathbb{N}}\gamma_{i,L}^{m+2s}\alpha_{i}^{2}=\sum_{i\in\mathbb{N}}\gamma_{i,L}^{m}\beta_{i}^{2}<\infty$$
		which implies $u\in\dot{H}_{L,W,V}^{m+2s}(\mathbb{T}^d).$
	\end{proof}
	For the next regularity result consider the following notation $$\dot{H}_{L,W,V}^{\infty}(\mathbb{T}^d):=\bigcap_{m>0} \dot{H}_{L,W,V}^{m}(\mathbb{T}^d)$$
	
	\begin{theorem}\label{reguregu}
		Fix $s\geq0$. If $f\in\dot{H}_{L,W,V}^{\infty}(\mathbb{T}^d)$ and $u\in\dot{H}_{L,W,V}^s(\mathbb{T}^d)$ is the unique weak solution of the problem  $$L^{s}_{W,V}u=f.$$ Then $u\in \dot{H}_{L,W,V}^{\infty}(\mathbb{T}^d).$ 
	\end{theorem}
	\begin{proof}
		Follows directly from the propositions \ref{prop99}.
	\end{proof}
	\begin{corollary}
		For $s\geq0$ and $f\in C^{\infty}_{W,V}(\mathbb{T}^d)$ the unique weak solution $u\in H^{s}_{W,V}(\mathbb{T})$ of $$(I-\Delta_{W,V}u)^{s}u=f$$
		belongs to $H^{\infty}_{W,V}(\mathbb{T})$. In particular, if $d=1$, then $u \in C^{\infty}_{W,V}(\mathbb{T})$.
	\end{corollary}
	\begin{proof}
		Indeed, in the particular case where $\kappa\equiv H\equiv 1$, we have $\dot{H}_{L,W,V}^s(\mathbb{T}^d)= H^s_{W,V}(\mathbb{T}^d)$, so, the result follows from Theorem \ref{reguregu}. Finally, if $d=1$, we have by Corollary \ref{hinfinityreg} that $H^{\infty}_{W,V}(\mathbb{T})=C^{\infty}_{W,V}(\mathbb{T})$.
	\end{proof}

\subsection{$H^{\infty}_{W,V}(\mathbb{T}^d)$ is a nuclear space}

Let $X$ be a vector space and let  $\<\cdot,\cdot\>_n, \;\; n\in\mathbb{N}$ be a sequence of inner products such that their norms are increasing, that is, for every $x\in X$ and $n<m$, we have $\|x\|_n\leq \|x\|_m$, where $\|\cdot\|_n$ is the norm associated to $\<\cdot,\cdot\>_n$. Denote by $X_n$ the completion of $X$ with respect to $\|\cdot \|_n$. Now, define
\begin{equation*}
	X_\infty = \bigcap_{n=1}^\infty X_n.
\end{equation*}

Then, $(X_\infty, (\|\cdot\|_n)_{n\in\mathbb{N}})$ is called a Countably Hilbert space and is a Fréchet space with respect to the metric
\begin{equation}\label{metric}
	d(f,g) = \sum^\infty_{n=1}2^{-n}\frac{\|f-g\|_n}{1+\|f-g\|_n}.
\end{equation}
Note that since the norms are increasing, we have
$$X_m\subset X_n\;\;\text{for all}\;\;m\ge n.$$

We say that countably Hilbert space $X_\infty$ is Nuclear if, for each $n \geq 0$, there exists $m>n$ such that the inclusion $i_{m,n}:X_m\to X_n$ is Hilbert-Schmidt. That is, given an orthonormal basis in $X_m$, say $\{e_j\}_{j\geq 1}$, we have
$$\sum_{j=1}^\infty\|e_j\|^2_n<\infty.$$

Let us now show that $H^{\infty}_{W,V}(\mathbb{T}^d)$ is a nuclear space. In particular, if $d=1$ we will have that $C^\infty_{W,V}(\mathbb{T})$ is a nuclear space, which generalizes the well-known fact that the Schwarz space $C^\infty(\mathbb{T})$ is a nuclear space.

This is a major advance in the field, as we can choose $V(x)=x$ to obtain that the very abstract space $\mathcal{S}_W(\mathbb{T}^d)$ considered in \cite{farfansimasvalentim} is actually the concrete $H^\infty_{W,V}(\mathbb{T}^d)$, which if $d=1$ is actually $C_{W,V}^\infty(\mathbb{T})$. We believe that $H^\infty_{W,V}(\mathbb{T}^d) = C^\infty_{W,V}(\mathbb{T}^d)$ but it is still an open problem for us.

Then, note that the sequence of norms $\|\cdot\|_{n,W,V}$ given in \eqref{normoddd} and \eqref{normevend} are increasing. If we choose $X = C_{W,V}^\infty(\mathbb{T}^d)$, we obtain that for each $n$, 
$$X_n = H^n_{W,V}(\mathbb{T}^d).$$
All that is left to do is to show that for every $n$, we can find $m\geq n$ such that the inclusion $i_{m,n}:H^m_{W,V}(\mathbb{T}^d) \to H^n_{W,V}(\mathbb{T}^d)$ is Hilbert-Schmidt. From Corollary \ref{hilbertschmidtd}, given $n\in\mathbb{N}$, for any $m\in\mathbb{N}$ such that $m > n+ d/4$, the inclusion $i_{m,n}$ is Hilbert-Schdmit. This shows that $X_\infty = H^\infty_{W,V}(\mathbb{T}^d)$ is a nuclear space.
	
	The fact that $H^\infty_{W,V}(\mathbb{T}^d)$ makes it an ideal space of test functions as it can be used in the Mitoma's compactness criterion (see \cite{Mit}) and also as test functions to obtain existence and uniqueness of evolution stochastic differential equations taking values on dual of nuclear spaces.

	\subsection{Second-order elliptic stochastic partial differential equations}

	We will now provide the tools to work with the fractional counterpart to the stochastic partial differential equation considered in Section 8 of \cite{keale}. For this approach we will only ``operate'' on the $L^2_V(\mathbb{T})$ space as remarked at the beginning of Section \ref{fracsect}.
	
	We begin by providing a definition for the $V$-Gaussian white noise on $L^2_V(\mathbb{T})$:
	
	\begin{definition}\label{gaussianwhitenoiseV}
		We define the $V$-Gaussian white noise as the $L^2_V(\mathbb{T}^d)$-isonormal Gaussian process, that is, a process $\{\dot{B}_V(h); h\in L^2_V(\mathbb{T}^d)\}$ defined on a complete probability space $(\Omega,\mathcal{F},P)$, such that for every $g,h\in L^2_V(\mathbb{T}^d)$, $\dot{B}_V(g)$ and $\dot{B}_V(h)$ are centered $\mathbb{R}$-valued gaussian random variables such that $E(\dot{B}_V(g)\dot{B}_V(h)) = \int_{\mathbb{T}^d} gh dV$.
	\end{definition}
	
	\begin{remark}\label{whitenoiselinear}
		It follows directly (see \cite[Chapter 1]{nualart} for further details) that $\dot{B}_V$ is a linear isometry between $L^2_V(\mathbb{T}^d)$ and a closed subspace of $L^2(\Omega,\mathcal{F},P)$.
	\end{remark}
	
	\begin{remark}
		The existence of the $V$-Gaussian white noise follows directly from Kolmogorov's extension theorem.
	\end{remark}

Let us connect the one-dimensional $V$-Gaussian white noise with the $V$-Brownian motion considered in \cite{keale}.

\subsubsection{The one-dimensional $V$-Gaussian white noise}

Recall from \cite{keale} that a $V$-Brownian motion in law is a process $B_V(\cdot)$ such that
\begin{definition}\label{W-brownian-in-law2}
	We say that $B_V(t)$ is a $V$-Brownian motion in law if it satisfies the following conditions:
	\begin{enumerate}
		\item $B_V(0) = 0$ almost surely;
		\item If $t>s$, then $B_V(t)-B_V(s)$ is independent of $\sigma(B_V(u); u\leq s)$;
		\item If $t>s$, then $B_V(t)-B_V(s)$ has $N(0,V(t)-V(s))$ distribution.
	\end{enumerate}
If, additionally, $B_V(\cdot)$ has c\`adl\`ag sample paths, we say that $B_V(\cdot)$ is a $V$-Brownian motion.
\end{definition}
	
	We can recover the $V$-Brownian motion in law in Definition \ref{W-brownian-in-law2} from the one-dimensional $V$-Gaussian white noise by defining 
	\begin{equation}\label{brownian_from_noise}
		B_V(t) = \dot{B}_V(\boldsymbol{1}_{[0,t]}).
	\end{equation}
	
	\begin{proposition}\label{whitenoisebrownianinlaw}
		The process $B_V(\cdot)$ defined by \eqref{brownian_from_noise} is a $V$-Brownian motion in law.
	\end{proposition}
	\begin{proof}
		We have that $B_V(0)$ follows a Gaussian distribution with mean 0 and variance $V(0) = 0$. Therefore, $B_V(0) = 0$ almost surely. So, condition 1 follows.
		
		Now, since $\dot{B}_V(\cdot)$ is linear, we have that, for $t>s$,
		\begin{eqnarray*}
			B_V(t) - B_V(s) &=& \dot{B}_V(\boldsymbol{1}_{[0,t]}) - \dot{B}_V(\boldsymbol{1}_{[0,s]})\\
			&=& \dot{B}_V(\boldsymbol{1}_{(s,t]}).
		\end{eqnarray*}
		So, $B_V(t) - B_V(s)$ follows a gaussian distribution with mean 0 and variance $\int_{\mathbb{T}} \boldsymbol{1}_{(s,t]} dV = V(t) - V(s)$. Thus, condition 3 holds.
		
		Finally, condition 2 is a simple consequence of the fact that uncorrelated jointly gaussian random variables are independent and that for any $u\leq s < t$, we have $E((B_V(t)-B_V(s))B_V(u)) = E(\dot{B}_V(\boldsymbol{1}_{(s,t]})\dot{B}_V(\boldsymbol{1}_{[0,u]})) = \langle \boldsymbol{1}_{(s,t]},\boldsymbol{1}_{[0,u]}\rangle_V = 0,$ since the intervals $(s,t]$ and $[0,u]$ are disjoint.
	\end{proof}

Up to a modification, the $V$-Gaussian white noise induces a $V$-Brownian motion.
	
	\begin{corollary}\label{brownianfromnoise}
		The process $B_V(\cdot)$ has a modification that is a $V$-Brownian motion.
	\end{corollary}
	\begin{proof}
		It follows directly from the proof of Proposition 10 in \cite{keale}.
	\end{proof}
	
	Notice, then, that in dimension 1, Proposition \ref{whitenoisebrownianinlaw} and Corollary \ref{brownianfromnoise} connect the $V$-Brownian motion with the $V$-Gaussian white noise on $L^2_V(\mathbb{T})$.
	
	Let $f = \sum_{i=1}^n \alpha_i \boldsymbol{1}_{I_i}$, be a simple function, with $n\in\mathbb{N}$, $\alpha_1,\ldots,\alpha_n\in\mathbb{R}$ and $I_1,\ldots, I_n$ are disjoint intervals. Define the stochastic integral of simple functions as
	$$\int_{\mathbb{T}} f(s) dB_V(s) = \sum_{i=1}^n \alpha_i \dot{B}_V(\boldsymbol{1}_{I_i}) = \dot{B}_V(f).$$
	
	It then follows immediately that if $f$ is a simple function, then the following isometry holds:
	\begin{equation}\label{isometryintegral}
		E[(\dot{B}_V(f))^2] = E\left[\left(\int_{\mathbb{T}} f dB_V\right)^2\right] = \int_{\mathbb{T}} f^2 dV.
	\end{equation}
	
	Let $h\in L^2_V(\mathbb{T})$ and $h_n$ be a sequence of simple functions converging to $h$ in $L^2_V(\mathbb{T})$. It follows from \eqref{isometryintegral}, that $\int_{\mathbb{T}}h_ndB_V$ is a Cauchy sequence in $L^2(\Omega,\mathcal{F},P)$, so we define
	\begin{equation}\label{stochintegral}
		\int_{\mathbb{T}} h dB_V = \lim_{n\to\infty} \int_{\mathbb{T}} h_n dB_V = \lim_{n\to\infty} \dot{B}_V(h_n),
	\end{equation}
	where the limit is taken in $L^2(\Omega,\mathcal{F},P)$.
	
	It also follows directly from the isometry \eqref{isometryintegral} that for any $h\in L^2_V(\mathbb{T})$
	\begin{equation}\label{whitenoisestochint}
	\dot{B}_V(h) = \int_{\mathbb{T}} h dB_V.
	\end{equation}

	\begin{remark}\label{whitenoiseL2Sobolev}
		Observe that \cite[Proposition 12]{keale} and \eqref{whitenoisestochint} show that the functional on $H_{W,V,\mathcal{D}}(\mathbb{T}):=\{f: f\in H^1_{W,V}(\mathbb{T}), f(0)=0\}$ induced by the $V$-Gaussian white noise on $L^2_V(\mathbb{T})$ introduced in Definition \ref{gaussianwhitenoiseV} through the $L^2_V(\mathbb{T})$-isonormal gaussian process coincides with the pathwise $V$-Gaussian white noise (with respect to the measure induced by $V$) introduced in \cite[Definition 12]{keale}.
	\end{remark}

	\subsubsection{$V$-Gaussian white noise on dimension $d$}
	Consider a dimension $d\geq 1$. We will now provide the well-known white noise expansion, which will enable us to solve the fractional equation. We will provide the proof for completeness since it is short.
	
	\begin{proposition}\label{whitenoise-exp}
		There exists an iid sequence of standard normal random variables $\xi_1,\xi_2,\ldots$ such that for every $h\in L^2_V(\mathbb{T}^d)$
		$$\dot{B}_V(h) = \sum_{i=0}^\infty \xi_i \<h,\nu_i\>_V,$$
		where the sum converges in $L^2(\Omega,\mathcal{F},P)$ and $P$-a.s.
	\end{proposition}
	\begin{proof}
		Begin by noticing that since $\{\nu_i\}_{i\in\mathbb{N}}$ are orthogonal, $\{\dot{B}_V(\nu_i)\}_{i\in\mathbb{N}}$ is a sequence of uncorrelated and jointly gaussian distributions. Thus, the random variables $\dot{B}_V(\nu_1),$  $\dot{B}_V(\nu_2),\ldots$ are independent and gaussian with variance $\|\nu_i\|_V=1$. So they form an iid sequence of standard normal distributions.
		
		By Remark \ref{whitenoiselinear}, $\dot{B}_V$ is a linear isometry between $L^2_V(\mathbb{T}^d)$ and a closed subspace of $L^2(\Omega,\mathcal{F},P)$. Hence, it is easy to check that the sum converges in $L^2(\Omega,\mathcal{F},P)$.
		
		Finally, let $M_n = \sum_{i=0}^n \xi_i \<h,\nu_i\>_V.$
		Then, $(M_n)_{n\in\mathbb{N}}$ is a martingale such that
		$$\sup_{n\in\mathbb{N}} E|M_n| \leq \sqrt{E((\dot{B}_V(h))^2)} = \|h\|_V <\infty.$$
		Therefore, by Theorem 4.2.11 in \cite{durrett}, $M_n$ converges a.s., which completes the proof.
	\end{proof}
	
	\begin{remark}\label{remarkexpwhite}
		The expansion in Proposition \ref{whitenoise-exp} motivates the following expansion for $\dot{B}_V$:
		\begin{equation}\label{eqexpwhite}
			\dot{B}_V = \sum_{i=0}^\infty \xi_i \nu_i.
		\end{equation}
		However, the above series does not converge on $L^2_V(\mathbb{T}^d)$.
	\end{remark}
	
	In the next Proposition we use the characterization in Proposition \ref{dualsob} to obtain a space in which the expansion in Remark \ref{remarkexpwhite} converges.
	
	\begin{proposition}\label{expwhitedual}
		If $T:L^2_V(\mathbb{T}^d)\to H$ is a Hilbert-Schmidt operator, where ${H}$ is some Hilbert space with norm $\|\cdot\|_H$. Then, the expansion
		\begin{equation}\label{expwhitenoiseHS}
			T\dot{B}_V := \sum_{i=0}^\infty \xi_i T\nu_i
		\end{equation}
		is a well-defined $L^2(\Omega,\mathcal{F},P)$-random variable in $H$, where the random variables $\xi_1,\xi_2,\ldots$ are the iid standard normal random variables defined in Proposition \ref{whitenoise-exp}.
	\end{proposition}
	\begin{proof}
		We have that
		$$
		\left\|\sum_{i=0}^N\xi_i T\nu_i\right\|_H^2 = \sum_{i,j=1}^N \xi_i\xi_j \<T\nu_i,T\nu_j\>_H.
		$$
		Thus,
		$$E\left(\left\|\sum_{i=0}^N\xi_i T\nu_i\right\|_H^2\right) = \sum_{i=1}^N \|T\nu_i\|_H^2 \leq \sum_{i=1}^\infty \|T\nu_i\|_H^2 <\infty,$$
		since $T$ is Hilbert-Schmidt. Therefore, the sum
		$\sum_{i=0}^N\xi_i T\nu_i$
		converges to some $H$-valued random variable in $L^2(\Omega,\mathcal{F},P)$. 
	\end{proof}
	
	\begin{corollary}\label{whitenoisedual}
		We have that if $s > d/2$, then $P$-almost surely $\dot{B}_V|_{H^s_{W,V}(\mathbb{T}^d)}\in H^{-s}_{W,V}(\mathbb{T}^d)$, where $\dot{B}_V|_{H^s_{W,V}(\mathbb{T}^d)}$ is the restriction of the $V$-Gaussian white noise to $H^s_{W,V}(\mathbb{T}^d)$. Furthermore, the following expansion converges in $H^{-s}_{W,V}(\mathbb{T}^d)$:
		$$\dot{B}_V = \sum_{i=0}^\infty \xi_i \nu_i.$$
	\end{corollary}
	\begin{proof}
		By Corollary \ref{hilbertschmidtd}, if $s > d/2,$ then the inclusion $i:L^2_V(\mathbb{T}^d)\to H^{-s}_{W,V}(\mathbb{T}^d)$ is trace-class and therefore Hilbert-Schmidt. By Proposition \ref{expwhitedual}, 
		$$i(\dot{B}_V) = \sum_{i=0}^\infty \xi_i i(\nu_i)$$
		is a well defined $L^2(\Omega,\mathcal{F},P)$-random variable in $H^{-s}_{W,V}(\mathbb{T}^d)$. Furthermore, by Proposition \ref{whitenoise-exp}, we have the following equality for every $h\in H^s_{W,V}(\mathbb{T}^d)\subset L^2_V(\mathbb{T}^d)$:
		$$\dot{B}_V(h) = i(\dot{B}_V)(h).$$
		Therefore, $\dot{B}_V|_{H^s_{W,V}(\mathbb{T}^d)} = i(\dot{B}_V)$.
	\end{proof}
	
	We are now in a position to solve the Mat\'ern-like fractional elliptic equation driven by $V$-generalized gaussian white noise:
	
	\begin{equation}\label{fracstoch}
		L_{W,V}^\beta u = \dot{B}_V,
	\end{equation}
	where $L_{W,V}$ is the operator induced by the bilinear form \eqref{bilinearLVW}. Note that if $H=I$, then the equation becomes
	$$(\kappa^2 I - \Delta_{W,V})^\beta u =\dot{B}_V.$$
	
	\begin{remark}
		Observe that the stochastic partial differential equation (\ref{fracstoch}) generalizes the Mat\'ern equation on the $d$-dimensional torus.
	\end{remark}
	
	\begin{remark}
		Notice that unlike the stochastic partial differential equation considered in Section 8 of \cite{keale}, the fractional equation \eqref{fracstoch} is driven by a gaussian white noise on $L^2_V(\mathbb{T}^d)$.
	\end{remark}
	
	\begin{theorem}
		Let $L_{W,V}$ be the operator with bilinear form \eqref{bilinearLVW}. If $\beta>d/4$, then the solution of \eqref{fracstoch} given by
		$$u = L_{W,V}^{-\beta} \dot{B}_V$$
		is a well-defined centered gaussian $L^2(\Omega,\mathcal{F},P)$-random variable taking values in the space $\dot{H}^{
	2\beta-\frac{d}{2}-\epsilon}_{L,W,V}(\mathbb{T}^d)$ for all $\epsilon >0$ and covariance operator given by $L_{W,V}^{-2\beta}$ a.s.
	\end{theorem}
	\begin{proof}
		By Theorem \ref{coro4}, the operator $L_{W,V}^{-s}$ is trace-class for $s>d/2$. In particular, if $\beta>d/4$, the restriction of $L_{W,V}^{-\beta}$ to $L^2_V(\mathbb{T})$ is Hilbert-Schmidt. Hence, from Proposition \ref{expwhitedual} and \cite[Lemma  2.1]{daviboo}, if $\beta > d/4$, then $L_{W,V}^{-\beta}\dot{B}_V$ is a well-defined $L^2(\Omega,\mathcal{F},P)$-random variable taking values in $\dot{H}^{
	2\beta-\frac{d}{2}-\epsilon}_{L,W,V}(\mathbb{T}^d)$. Finally, by writing $u = L_{W,V}^{-\beta}\dot{B}_V$ in terms of expansion in equation \eqref{expwhitenoiseHS} we readily obtain that $u$ is a centered gaussian random variable with covariance operator given by $L_{W,V}^{-2\beta}$.
	\end{proof}
\chapter{Final Discussion}

At this point of the work, we believe we were able to show the importance of the $W$-$V$-Sobolev spaces $H_{W,V}^{s}(\mathbb{T})$, of the space of test functions $C_{W,V}^{\infty}(\mathbb{T})$. Further, we showed that the $W$-Brownian motion is deeply connected to the $W$-$V$-Sobolev spaces. Moreover, by using the theory, we were able to solve some elliptic partial differential equations and stochastic partial differential equations.

However, some interesting questions can be raised in view of the nature and of the results we obtained. As shown above, we can extend our one-dimensional model to higher dimensions, thus obtaining analogues to $W$-$V$-Sobolev spaces and for space of test functions. By following the ideas in Chapter 2 it is straightforward to establish the existence and uniqueness for problems related to the elliptic differential operator $$
E_{W,V}u=\sum_{i=1}^d -\partial^{+}_{V_{i}}\left(a_{i}\partial^{-}_{W_{i}}\right)u+cu,$$ where $c$ and $a_{i}$ satisfy some suitable conditions. The operator $E_{W,V}$ is a generalization of the classical elliptic differential operator $$\tilde{E}u=-\sum_{i=1}^d \partial_{x_{i}}\left(a_{i}\partial_{x_{i}}\right)u+cu.$$ But, as we know, we can study the operator $\tilde{E}$ in a more complete manner, which renders the operator asymmetric. More precisely, one can consider the operator given by 
$$Lu=-\sum_{i,j=1}^d\partial_{x_{i}}\left(a_{i,j}\partial_{x_{j}}\right)u+\sum_{i=1}^{d}b_{i}\partial_{x_{i}}u+cu$$ where the matrix $\textbf{A}=(a_{ij})_{i,j=1,\ldots, d}$ and and the vector $\textbf{b}=(b_{i})_{i=1,\ldots,d}$ satisfies some suitable conditions. Therefore, the following question is natural: 
\begin{center}
    ``How can we generalize the operator $E_{W,V}$ in such manner to study its complete version where necessarily the matrix $\textbf{A}$ is not taken as a diagonal and $\textbf{b}\neq 0?$''
\end{center}
Still talking about the operator $E_{W,V}$ and motivated by the results presented in \cite{simasvalentim2}, we can point the following interesting question about the regularity:
\begin{center}
    ``What is the shape of regularity results that we can establish for the operator $\sum_{i=1}^d -\partial^{+}_{V_{i}}\left(\partial^{-}_{W_{i}}\right)u+u$?,  and how about $E_{W,V}$?''
\end{center}
Back to the topic of high-dimensional extensions of the the one dimensional model of $-\Delta_{W,V}$. Note that our $d$-dimensional differential operator $-\Delta_{\textbf{W,V}}$, can be seen as an operator depending on $\textbf{W}$ and $\textbf{V}$. If one considers $\textbf{W}$ and $\textbf{V}$ as probability distribution functions, the case we considered means that we assumed the jointly distribution to the product (which would mean they are the distribution of independent random variables). We can also have the following question
\begin{center}
    ``Is it possible consider generalized second order operators by considering formal derivatives with respect an Borel measure on $\mathbb{T}^d$ not necessarily given by the product measure?.''
\end{center}
Finally, turning our attention to the $W$-Brownian motion $B_{W}$, and pointing that, the jumps of $B_{W}$ is subordinated to the jumps of $W$ and expecting that this model of stochastic process can be useful in future works related to stocks in the financial market one nice question about future models related to this object is 
\begin{center}
    ``By embedding $B_W$ on a random environment, in such a manner that the functions $V$ and $W$ are random, thus making the jump sites random, is there a choice of the random environment such that the (unconditional) finite dimensional distributions of $B_W$ are still Gaussian?''
\end{center}





\cleardoublepage
\addcontentsline{toc}{chapter}{Bibliography}

%
	

\end{document}